\documentclass{article}

\let\epsilon=\varepsilon
\usepackage[utf8]{inputenc}
\usepackage{mathtools}
\usepackage{algorithm}
\usepackage{myalgorithmic}
\usepackage{amsthm}
\usepackage{amssymb}
\usepackage{authblk}
\usepackage{bm}
\usepackage{booktabs}
\usepackage[tableposition=top]{caption}
\usepackage{comment}
\usepackage{etoolbox}
\usepackage{float}
\usepackage{graphicx}
\usepackage{needspace}
\usepackage{siunitx}
\usepackage{tikz}
\usepackage[margin=1in]{geometry}
\usepackage{multirow}
\usepackage{hyperref}
\usepackage{paralist}
\usepackage{xcolor}

\usetikzlibrary{3d,calc}

\usepackage{import}

\usepackage[backend=biber, style=numeric, giveninits=true]{biblatex}

\DeclareFieldFormat*{title}{\mkbibemph{#1}}
\DeclareFieldFormat*{citetitle}{\mkbibemph{#1}}
\DeclareFieldFormat{journaltitle}{#1}

\renewbibmacro*{in:}{%
  \ifentrytype{article}
    {}%
    {\printtext{\bibstring{in}\intitlepunct}}%
    }

\newbibmacro*{pubinstorg+location+date}[1]{%
  \printlist{#1}%
  \newunit%
  \printlist{location}%
  \newunit%
  \usebibmacro{date}%
  \newunit}

\renewcommand{\Re}{\mathop{\mathrm{Re}}%
}
\newtheorem{theorem}{Theorem}
\newtheorem{lemma}[theorem]{Lemma}
\newtheorem{proposition}[theorem]{Proposition}
\newtheorem{remark}[theorem]{Remark}
\newtheorem{definition}{Definition}

\newtheorem{hypothesis}{Hypothesis}

\newcommand{\ilist}[2]{%
  \ifstrequal{#1}{1}{U_{#2}}{%
  \ifstrequal{#1}{2}{V_{#2}}{%
  \ifstrequal{#1}{3}{W_{#2}}{%
  \ifstrequal{#1}{3close}{W^\mathrm{close}_{#2}}{%
  \ifstrequal{#1}{3far}{W^\mathrm{far}_{#2}}{%
  \ifstrequal{#1}{4}{X_{#2}}{%
  \ifstrequal{#1}{4close}{X^\mathrm{close}_{#2}}{%
  \ifstrequal{#1}{4far}{X^\mathrm{far}_{#2}}{}%
}}%
}}%
}}%
}}

\def\Mpole{\mathsf{M}}
\def\Locfar{\mathsf{L}^{\text{far}}%
}
\def\Potnear{\mathsf{P}^{\text{near}}%
}
\def\PotW#1{\mathsf{P}^{\ilist{3}{}}_{#1}}
\def\Lqbxnear#1{\mathsf{L}^{\text{qbx},\text{near}}_{#1}}
\def\LqbxW#1{\mathsf{L}^{\text{qbx},\ilist{3}{}}_{#1}}
\def\Lqbxfar#1{\mathsf{L}^{\text{qbx},\text{far}}_{#1}}

\def\ptpot{\Phi}

\newcommand\coeffs[1]{\langle #1^m_n \rangle}

\newcommand{\MToM}[4]{\mathrm{M2M}_{{#3} \to {#4}}^{{#1} \to {#2}}}
\newcommand{\LToL}[4]{\mathrm{L2L}_{{#3} \to {#4}}^{{#1} \to {#2}}}
\newcommand{\MToL}[4]{\mathrm{M2L}_{{#3} \to {#4}}^{{#1} \to {#2}}}

\newlength{\longrowlength}
\newlength{\figurewidth}
\newcommand{\nmax}{{n_{\mathrm{max}}%
}}

\newcommand{\pqbx}{{p_\mathrm{qbx}%
}}
\newcommand{\pfmm}{{p_\mathrm{fmm}%
}}
\newcommand{\padd}{{p_\mathrm{add}%
}}
\newcommand{\pquad}{{p_\mathrm{quad}%
}}

\newcommand{\closedbox}{\overline {B_\infty}}
\newcommand{\closedball}{\overline {B_2}}

\newcommand{\rdanger}[1]{{r^\mathrm{danger}_{#1}%
}}

\newcommand{\ancestors}{\mathsf{Ancestors}}
\newcommand{\descendants}{\mathsf{Descendants}}
\newcommand{\tcr}{\mathsf{TCR}}
\newcommand{\parent}{\mathsf{Parent}}
\newcommand{\adequatesep}{\prec}

\newcommand{\nelements}{K}
\newcommand{\iel}{k} % element index
\newcommand{\ncenterbases}{N_C/2}
\newcommand{\ictr}{i} % center base point index
\newcommand{\mapping}{\Psi}
\newcommand{\ctr}{c_{\ictr}^{\pm}%
}
\newcommand{\tgtorder}{N_t}
\newcommand{\curvthresh}{\kappa_{\text{max}}%
}
\newcommand{\surfel}[1]{%
  \lvert%
  \partial_{s_1} \mapping_\iel#1 \times%
  \partial_{s_2} \mapping_\iel#1%
  \rvert_2%
}
\newcommand{\etastg}[2]{\eta^{{\text{stage-}#1}}_{#2}}
\newcommand{\disturbtol}{\epsilon_{\text{exp-disturb}%
}}
\newcommand{\tgtassoctol}{\epsilon_{\text{tgt}%
}}

\newcommand{\algbrand}{GIGAQBX}

\newcommand{\pt}[1]{#1}

\newcommand{\norm}[1]{\lvert #1 \rvert_2}

\makeatletter
\tikzoption{canvas is plane}[]{\@setOxy#1}
\def\@setOxy O(#1,#2,#3)x(#4,#5,#6)y(#7,#8,#9)%
  {\def\tikz@plane@origin{\pgfpointxyz{#1}{#2}{#3}}%
   \def\tikz@plane@x{\pgfpointxyz{#4}{#5}{#6}}%
   \def\tikz@plane@y{\pgfpointxyz{#7}{#8}{#9}}%
   \tikz@canvas@is@plane
  }
\makeatother

\newcommand\DrawCube[5][black]{%
\begin{scope}[canvas is plane={O(#2,#3,#4+#5)x(#2+#5,#3,#4+#5)y(#2,#3+#5,#4+#5)}]
  \draw[#1] (-1,-1) rectangle (1,1);
\end{scope}
\begin{scope}[canvas is plane={O(#2,#3,#4-#5)x(#2+#5,#3,#4-#5)y(#2,#3+#5,#4-#5)}]
  \draw[#1] (-1,1) -- (1,1);
  \draw[#1,dotted] (-1,-1) -- (1,-1);
\end{scope}
\begin{scope}[canvas is plane={O(#2+#5,#3,#4)x(#2+#5,#3,#4+#5)y(#2+#5,#3+#5,#4)}]
  \draw[#1] (-1,-1) rectangle (1,1);
\end{scope}
\begin{scope}[canvas is plane={O(#2-#5,#3,#4)x(#2-#5,#3,#4+#5)y(#2-#5,#3+#5,#4)}]
  \draw[#1] (-1,1) -- (1,1);
  \draw[#1,dotted] (-1,-1) -- (1,-1);
  \draw[#1,dotted] (-1,-1) -- (-1,1);
\end{scope}
}

\newcommand\DrawSphere[5][black]{%
\begin{scope}[canvas is plane={O(#2,#3,#4)x(#2+#5,#3,#4)y(#2,#3+#5,#4)}]
  \draw[#1] (0,0) circle (1);
\end{scope}
\begin{scope}[canvas is plane={O(#2,#3,#4)x(#2,#3,#4+#5)y(#2,#3+#5,#4)}]
  \draw[#1] (0,-1) arc (270:270+180:1);
  \draw[#1,dotted] (0,1) arc (90:270:1);
\end{scope}
\begin{scope}[canvas is plane={O(#2,#3,#4)x(#2+#5,#3,#4)y(#2,#3,#4+#5)}]
  \draw[#1] (1,0) arc (0:180:1);
  \draw[#1,dotted] (-1,0) arc (180:360:1);
\end{scope}
\begin{scope}[canvas is plane={O(#2,#3,#4)x(#2+#5/sqrt 2,#3,#4+#5/sqrt 2)y(#2,#3+#5,#4)}]
  \draw[#1] (0,-1) arc (270:270+180:1);
  \draw[#1,dotted] (0,1) arc (90:270:1);
\end{scope}
\begin{scope}[canvas is plane={O(#2,#3,#4)x(#2-#5/sqrt 2,#3,#4+#5/sqrt 2)y(#2,#3+#5,#4)}]
  \draw[#1] (0,-1) arc (270:270+180:1);
  \draw[#1,dotted] (0,1) arc (90:270:1);
\end{scope}
}

\tikzset{%
  >=latex, % nice arrows
  inner sep=0pt,
  outer sep=2pt,
  mark coordinate/.style={inner sep=0pt,outer sep=0pt,minimum size=3pt,
    fill=black,circle,color=black},
  length/.style={color=black},
  declare function={
   gammax(\t) = \t+1;
   gammay(\t) = (1 - 0.4 * \t) - 0.2 * sin((5 * \t) r);
   gammadx(\t) = 1;
   gammady(\t) = - cos((5 * \t) r) - 0.4;
   gammads(\t) = sqrt(gammadx(\t)^2 + gammady(\t)^2);
   gammanx(\t) = -gammady(\t) / gammads(\t);
   gammany(\t) = gammadx(\t) / gammads(\t);
   finex(\t) = \t+1;
   finey(\t) = 0.4 * tanh(5 * \t + 3.5) + 0.75;
   coarsex(\t) = \t+1;
   coarsey(\t) = -0.1 * sin((1.5 * \t + 0.2) r);
   coarsedx(\t) = 1;
   coarsedy(\t) = -0.1 * 1.5 * cos((1.5 * \t + 0.2) r);
   coarseds(\t) = sqrt(coarsedx(\t)^2 + coarsedy(\t)^2);
   coarsenx(\t) = -coarsedy(\t) / coarseds(\t);
   coarseny(\t) = coarsedx(\t) / coarseds(\t);
  },
  coord/.style={inner sep=0pt, outer sep=0pt, minimum size=3pt, circle},
}

\definecolor{qbxcolor}{RGB}{43,131,186}
\definecolor{srccolor}{RGB}{63,140,52}
\definecolor{localcolor}{RGB}{215,25,28}
\definecolor{flagcolor}{RGB}{244, 109, 67}
\colorlet{cubecolor}{gray}
\colorlet{badcolor}{localcolor}
\colorlet{goodcolor}{srccolor}

\newenvironment{algbreakable}[1]{
    \refstepcounter{algorithm}
    \needspace{3\baselineskip}
    \noindent\rule{\textwidth}{0.5pt}
    \textbf{Algorithm~\arabic{algorithm}:} #1\\
    \rule[1.25ex]{\textwidth}{0.5pt}
    \vspace{-4ex}
  }
  {
    \vspace{1ex}
    \rule{\textwidth}{0.5pt}
  }

\def\algstage#1#2{\vspace{2ex}\begin{minipage}{0.97\textwidth}\STATE{\textit{#1}}#2\end{minipage}}

\newenvironment{DIFnomarkup}{}{}

\begin{document}
% {{{ front matter

\title{A Fast~Algorithm for Quadrature~by~Expansion in Three~Dimensions}

\author[1]{Matt Wala\thanks{\texttt{wala1@illinois.edu}}}
\author[1]{Andreas Klöckner\thanks{\texttt{andreask@illinois.edu}}}
\affil[1]{Department of Computer Science, University of Illinois at Urbana-Champaign}
\date{March 29, 2019}

\maketitle

\begin{abstract}
  This paper presents an accelerated quadrature scheme for the evaluation of
  layer potentials in three dimensions. Our scheme combines a generic, high
  order quadrature method for singular kernels called Quadrature by Expansion
  (QBX) with a modified version of the Fast Multipole Method (FMM). Our scheme extends a recently
  developed formulation of the FMM for QBX in two dimensions, which, in that
  setting, achieves mathematically rigorous error and running time bounds.
  In addition to generalization to three dimensions,
  we highlight some algorithmic and mathematical opportunities for improved
  performance and stability. Lastly, we give numerical evidence supporting the accuracy,
  performance, and scalability of the algorithm through a series of
  experiments involving the Laplace and Helmholtz equations.
\end{abstract}

% }}}

% #############################################################################
\section{Introduction}%
\label{sec:intro}
% -----------------------------------------------------------------------------
% {{{

Integral equation methods are an attractive approach for the solution of
boundary value problems of elliptic partial differential equations (PDEs).  The
mathematical features that make integral equation methods attractive in two
dimensions also hold in the three dimensional case, where their impact is felt
even more drastically:
exterior problems pose no more complication than their
interior counterparts, a dimensional reduction in the number of degrees of
freedom is achieved, and the conditioning of the numerical discretization
mirrors that of the physical problem. However, due to the difficulties of scale
and engineering involved, the practical realization in three dimensions is more
difficult than in two.

The premise of integral equation methods is based on the reformulation of the
underlying differential equation in integral form. The
solution to a homogeneous elliptic boundary value problem (BVP) may be represented as a \emph{layer potential}, such as
the single-layer potential $\mathcal{S}$, a surface convolution integral over the boundary $\Gamma$:
\begin{equation}
  \mathcal{S} \mu(x) \coloneqq \int_\Gamma \mathcal{G}(x, y) \mu(y)
  \, dS(y) \\
  \label{eq:slp-definition}
\end{equation}
where the density function $\mu: \Gamma \to \mathbb{R}$ is unknown,
and $\mathcal{G}$ is the free-space Green's function for the (homogeneous) PDE\@.
For instance, for the Laplace equation $\triangle u=0$ in three dimensions,
\begin{equation}
\mathcal G(x, y) \coloneqq {(4\pi)}^{-1}\norm{x-y}^{-1}.
\label{eq:laplace-fs-green}
\end{equation}

A number of challenges when solving a BVP with integral equation methods
are apparent, particularly in three dimensions. First,
solving the BVP requires the linear operator representing the
restriction of $\mathcal{S}$ (or, depending on the integral equation
formulation, some other layer potential) to the
boundary $\Gamma$. When the restriction of
$\mathcal{S}$ to the boundary is discretized, it becomes a finite dimensional
linear operator $\mathcal{L}$. Unlike the discretization of differential
operators, the matrix representation of $\mathcal{L}$ is dense. Even
considering the effects of dimensional reduction (from volume to surface), the number of degrees of
freedom for three dimensional problems is sufficiently large that explicit
formation of a matrix for $\mathcal{L}$ can be prohibitively expensive.

A second challenge remains the same as the two dimensional case: obtaining
suitable, low complexity quadrature for the singular integrals involving the
Green's function. While quadrature techniques have long been studied for two
dimensional kernels, they are often built to be special-purpose, and those
having three dimensional analogs are comparatively fewer. A related but more
subtle issue is quadrature for the nearly singular integrals of the potential
that arise when the evaluation point is close to but not on the boundary.

% {{{ lit review

A rough overview of the subject of singular quadrature is given
in~\cite{klockner:2013:qbx} to which we refer the reader. Some notable schemes
featuring singular quadrature rules for one, two, and three dimensional problems
include~\cite{helsing_2008b,
  lowengrub_1993,goodman_1990,haroldson_1998,mayo_1985,beale_lai_2001,davis_1984,hackbusch_sauter_1994,graglia_2008,
  jarvenpaa_2003,schwab_1992,khayat_2005,bruno_2001,ying_2006,bremer_nonlinear_2010,farina_2001,strain_1995,johnson_1989,
  sidi_1988,carley_2007,atkinson_1995,lyness_numerical_1967,chapko_numerical_2000,hao_high-order_2014,bremer:fastdirect:2012,
  bremer:nystrom:2012,zhao:2010,mayo:2007}.  Most schemes intended for use when the
number of evaluation or source points is large, such as for the application of
layer potentials like~\eqref{eq:slp-definition}, feature an
acceleration component. A useful tool for this has been the Fast Multipole
Method~\cite{greengard:1987:fmm},~e.g.\ in~\cite{ying_2006}. A variety of other
acceleration methods have been utilized, such as fast direct solvers
(e.g.~\cite{bremer:fastdirect:2012}), recursive compressed inverse
preconditioning~\cite{helsing_2008b}, particle-mesh Ewald summation
(e.g.~\cite{zhao:2010}), or methods based on the Fast Fourier Transform (e.g.~\cite{mayo_1985}).

Quadrature by Expansion (QBX,~\cite{klockner:2013:qbx}) is a quadrature method that has been recently developed that
promises to unify the treatment of the two- and three-dimensional, on-surface and off-surface
cases for layer potentials like~\eqref{eq:slp-definition}, including those with
hypersingular kernels. QBX is applicable in
this generality since it only relies on the analyticity of the underlying
potential as a function inside or outside the domain, and, under mild
assumptions, the existence of a smooth extension of the potential onto the
boundary of the domain. Additionally, with some care, QBX is amenable to
acceleration with the Fast Multipole Method, as we demonstrate in this
contribution.

Because the FMM forms local expansions of source potentials, it appears well
suited to a scheme known as `global' QBX, in which a local expansion mediates
the potential due to all sources in the geometry.  The first published practical
realization of a QBX-FMM coupling,
in~\cite{rachh2015integral,rachh:2017:qbx-fmm} is a global scheme that achieves
high accuracy and acceleration in two dimensions. However, in this algorithm, the error introduced
by the FMM does not obey an error bound of the form $O((1/2)^{\pfmm+1})$,
where $\pfmm$ is the FMM order, which may be expected of a `point' FMM in
two dimensions~\cite{petersen:1995:fmm-error-est}. In order to achieve a given
accuracy tolerance $\epsilon > 0$, one requires an often considerably greater
FMM order than would be applicable for a `point' FMM, and, additionally,
the amount by which to increase the order must be empirically determined.
A more recent version~\cite{gigaqbx2d} of the QBX FMM redevelops the FMM algorithm with a
guaranteed error bounds resembling that of the `point' FMM. Both of these
algorithms are based on an `analytical' FMM\@. In contrast, the contribution
in~\cite{rahimian:2017:qbkix} develops Quadrature by Kernel-Independent
Expansion, a `numerical' version of QBX meant for use with the
kernel-independent FMM\@.

Other recent research on QBX~\cite{afklinteberg:2017:adaptive-qbx} has focused on automating the selection of
parameters for quadrature, radius, and expansion
order. Although thus far most of the work
on QBX has been restricted to two dimensions, theoretical work anticipates the
extensibility of QBX to the three dimensional
setting~\cite{epstein:2013:qbx-error-est,afklinteberg:2016:quadrature-est}. An
exception is~\cite{afklinteberg:2016:viscous-flow}
which uses QBX in three dimensions over spheroidal bodies.

An alternative to global QBX recently demonstrated to be viable in three
dimensions is `local' QBX\@. In contrast to global QBX, in local QBX only the
field in a neighborhood of each target point is mediated through local
expansions~\cite{barnett:2014:close-eval}. This enables more geometric
flexibility when placing expansion centers. This scheme is straightforward to
integrate with the FMM, as only an approximation to the `point' far field is
required. Despite these advantages over global QBX, local QBX appears to require a
generically higher quadrature order, and thus higher quadrature oversampling to
control the additional error introduced by the transition from the QBX near-field
to the point-source-based computation of the far field. The contribution~\cite{siegel2017local}
develops a three-dimensional local QBX algorithm with optimizations to decrease
the cost of applying the QBX expansions.

% }}}

% {{{ list of contributions

This paper describes an accelerated global QBX scheme in three dimensions which
builds and extends on \algbrand, our previous scheme for two dimensions featuring rigorous
error bounds~\cite{gigaqbx2d}. While much of the theory and many algorithmic
aspects are directly analogous to the two dimensional case, we introduce a
number of enhancements that help keep the scheme feasible and practical:
\begin{itemize}
  \item We replace the geometry processing and refinement scheme of~\cite{rachh:2017:qbx-fmm}
  with one that is applicable to surfaces in three dimensions.
  Specifically, we introduce new measures of quadrature resolution, and we
  replace the 2-to-1 length requirement in that scheme (which does not
  straightforwardly generalize to three dimensions) with a two-stage refinement
  scheme that separates the calculation of the QBX expansion radii from the mesh
  resolution of the quadrature discretization.  This is the subject of
  Sections~\ref{sec:accuracy-control} and~\ref{sec:error-estimates}.

  \item As part of this geometry processing, we report on an empirically
  effective criterion that aids in controlling the truncation error based on
  mesh element geometry, presented in Section~\ref{sec:truncation-error}.

  \item Algorithm~\ref{alg:target-assoc} presents a simplified version
  of the target-to-center association algorithm of~\cite{rachh:2017:qbx-fmm}.

  \item We provide an improved version of the `target confinement rule'
  of~\cite{gigaqbx2d}. This rule governs the relationship between boxes as they
  occur in the GIGAQBX FMM and the QBX expansions used to approximate the
  layer potential near the source geometry and thus plays a central role
  in determining the cost of the algorithm.  In~\cite{gigaqbx2d}, we used
  `square' target confinement regions, i.e.\ ones whose geometry is governed
  by the $\ell^\infty$ norm. By more closely matching the true convergence behavior
  of the QBX expansions through the definition of the target confinement
  region with the help of the $\ell^2$ norm, we obtain a considerable cost
  reduction.  Amidst error estimates generalized to the three-dimensional case,
  this is the subject of Section~\ref{sec:algorithm}.

  \item We derive complexity estimates and state conditions under which
  one may expect linear time complexity of the algorithm in three dimensions.
  This is the subject of Section~\ref{sec:complexity}.
\end{itemize}

% Other things that potentially should be on this list:
% \begin{itemize}
% \item Tree refine weights
% \item Tree weighting strategy \url{https://gitlab.tiker.net/inducer/pytential/issues/72}
% \end{itemize}

% }}}

Before we turn to the subject of these contributions, the next section presents
background material on the subject of QBX in three dimensions and the
considerations required for accuracy and acceleration.
% }}}

% #############################################################################
\section{Mathematical Preliminaries}%
\label{sec:background}
% -----------------------------------------------------------------------------
% {{{
As a model problem, consider the exterior Neumann problem for the Laplace
equation in three dimensions, for a smooth bounded domain $\Omega$. Given
continuous Neumann boundary data $g$, the problem is to find $u$ such that
\begin{DIFnomarkup}
\begin{alignat*}{2}
  \triangle             u &= 0 & \quad & \text{in } \mathbb{R}^3 \setminus \Omega, \\
  \partial_n            u &= g & \quad & \text{on } \partial \Omega, \\
  \lim_{|x| \to \infty} u &= 0.
\end{alignat*}
\end{DIFnomarkup}
Here, the notation $\partial_n$ indicates the derivative with respect to the
outward unit normal.

The solution to this problem may be represented as $u \coloneqq \mathcal{S} \mu$, a
single-layer potential over the boundary $\Gamma = \partial \Omega$, using an
unknown density function $\mu$. The properties of the operator $\mathcal{S}$
imply that the Laplace PDE and the far field boundary conditions are immediately satisfied
by this representation in the exterior domain. The Neumann boundary condition, on
$\Gamma$, together with the jump-relations for layer potentials~\cite{kress:2014:integral-equations}
entails that $\mu$ satisfies the integral equation of the second kind
\[
  - \frac{\mu}{2} + \mathcal{S}' \mu = g,
\]
where the operator $\mathcal{S}'$ is defined as
\[
  \mathcal{S}'\mu(x) \coloneqq PV \int_\Gamma \left( \partial_{n(x)} \mathcal{G} (x, y) \right) \mu(y) \, dS(y).
\]
After discretization of the integral equation, its solution by iterative methods
requires repeatedly applying the operator~$\mathcal{S}'$. Similarly, the evaluation of the BVP~solution $u$
requires applying the operator~$\mathcal{S}$. In this section we will focus on
the evaluation of~$\mathcal{S}$, although what is said in this section applies
with little additional work to $\mathcal{S}'$ or other layer potentials.

When the evaluation (or `target') point $\pt{x}$ is sufficiently far from $\Gamma$,
the approximate evaluation of the integral $\mathcal{S} \mu(\pt{x})$ can be
accomplished accurately using a high-order composite quadrature rule for smooth
functions, as the
integrand itself is a smooth function. For $\pt{x}$ nearer to the boundary, it is well
known that the singularity of the integrand presents resolution problems for
smooth quadrature rules, which we address through the use of QBX\@.

In what follows, we use the notation
\[
  B_p(r, c)\coloneqq\{y\in \mathbb R^3: \lvert c-y\rvert_p < r \}
\]
to denote the open ball with respect to the $\ell^p$-norm around center $c$
with radius $r$. $\overline{B_p}(r, c)$ denotes the closure of that ball.
In particular, we make use of $\closedbox(r, c)$ and $\closedball(r,c)$.
$\closedball(r, c)$ denotes the closed Euclidean ball of radius $r$ centered
at $c$.  $\closedbox(r, c)$ denotes the closed cube of radius $r$ centered at
$c$.

% }}}
% -----------------------------------------------------------------------------
\subsection{QBX Discretization}
\label{sec:qbx-background}
% -----------------------------------------------------------------------------
% {{{
The idea of QBX is to use the smoothness of the potential for purposes of close and on-surface evaluation
to recover a
high-order accurate approximation everywhere in the domain. This is accomplished
through formation of a local expansion of the potential near the source geometry
and analytic continuation of the local expansion towards the boundary.

Throughout this paper we make use of spherical harmonic expansions. The
expansion of the Laplace potential in spherical harmonics is based on writing
the Green's function~\eqref{eq:laplace-fs-green} using the following identity valid for $a, b \in
\mathbb{R}^3$ with $\norm{a} < \norm{b}$:
\begin{equation}
  \label{eqn:sph-harm}
  \mathcal{G}(a,b) =
  \sum_{n=0}^\infty
  \frac{1}{2n + 1} \frac{\norm{a}^n}{\norm{b}^{n+1}} \sum_{m=-n}^n
  Y_n^m(\theta_a, \phi_a) Y_n^{-m}(\theta_b, \phi_b).
\end{equation}
Here, $(\theta_a, \phi_a)$ and $(\theta_b, \phi_b)$ refer to the polar and
azimuthal spherical coordinates of, respectively, $a$ and $b$,
i.e.\ $\theta=\cos^{-1}(z/r),\phi=\operatorname{atan2}(y, x)$. The spherical
harmonic function $Y^m_n$ of order $m$ and degree $n$, $|m| \leq n$, is defined
as
\begin{equation}
  Y^m_n(\theta, \phi) \coloneqq \sqrt{\frac{2n+1}{4\pi}
    \frac{\left(n-|m|\right)!}{\left(n+|m|\right)!}}
    \cdot P^{|m|}_n(\cos \theta) e^{i m \phi}
  \label{eq:spherical-harmonics}
\end{equation}
where $P^m_n$ is the associated Legendre function of order $m$ and degree $n$.

The identity~\eqref{eqn:sph-harm} may be used in the formation of a \emph{local
  expansion} centered at a center $c \in \mathbb{R}^3$ as follows. Given a
source point $s \in \mathbb{R}^3$, define a doubly-indexed sequence $\coeffs{L}$
of \emph{local coefficients} by
\begin{equation}
  L^m_n \coloneqq \frac{1}{2n + 1} \frac{1}{\norm{s - c}^{n + 1}} Y^{-m}_n(\theta_{s-c}, \phi_{s-c}), \quad \lvert m \rvert \leq n
  \label{eqn:local-expansion-coeff}
\end{equation}
where the subscripted $\theta_{(\cdot)}$ and $\phi_{(\cdot)}$ from this section
onward refer to the polar and azimuthal spherical coordinates of the vector
argument. Then the local expansion evaluated at a target $t \in
\mathbb{R}^3$ may be written as
\[
  \mathcal{G}(s,t) = \sum_{n=0}^\infty \sum_{m=-n}^n L^m_n \norm{t - c}^n
  Y^m_n(\theta_{t-c}, \phi_{t-c}).
\]
A \emph{$p$-th order local expansion} is one in which the index of the outer
summation goes from $0$ to $p$.  (Some authors,
e.g.~\cite{greengard:1988:thesis,siegel2017local}, follow the convention of
defining the local coefficient~\eqref{eqn:local-expansion-coeff} using the
$Y^m_n$, the complex conjugate
of~$Y^{-m}_n$. Both~\eqref{eqn:local-expansion-coeff} and the latter definition
yield equivalent expansions, since the outer partial sums
of~\eqref{eqn:sph-harm} are real~\cite[eqn.~(20)]{siegel2017local}.)

Next, we describe the details of QBX. The QBX-based approximation of layer
potentials may be thought of as occurring in three distinct steps.

% }}}
% -----------------------------------------------------------------------------
\subsection{First Approximation Step: Truncation}%
\label{sec:approx-truncation}
% -----------------------------------------------------------------------------
% {{{

In the first stage, a local expansion of the potential is formed and truncated.
For a selection of points $\{x_\ictr\}_{\ictr=1}^{\ncenterbases}$ on the surface
$\Gamma$, we define a collection of $N_C$ expansion centers $\ctr$
\begin{equation}
  \label{eq:expansion-centers}
  \ctr \coloneqq x_\ictr \pm r(x_\ictr)\hat n(x_\ictr)
\end{equation}
where $\hat n(x)$ is a unit-length normal vector to the surface $\Gamma$ at $x$,
and $r(x)$ is a yet-to-be-determined expansion radius.

The local coefficients $\coeffs{(L_\ictr^\pm)}$ associated with the expansion at
$c_{\ictr}^{\pm}$ may be defined through the integrals
\begin{equation}%
  \label{eqn:qbxcoeff}
  (L_{\ictr}^\pm)^m_n = \frac{1}{2n+1} \int_\Gamma \frac{\mu(s)}{\norm{s-\ctr}^{n + 1}}
  Y^{-m}_n(\theta_{s - \ctr}, \phi_{s - \ctr}) \, dS(s).
\end{equation}
Then the $p$-th order QBX local expansion at a target $\pt{t}\in
\closedball(r(x_\ictr), \ctr)$ may be evaluated as
\begin{equation}%
  \label{eqn:qbxlocal}
  \mathcal{S} \mu (\pt{t}) \approx \sum_{n=0}^p \sum_{m=-n}^n  (L_\ictr^\pm)^m_n
  \norm{t - \ctr}^n Y^m_n(\theta_{t - \ctr}, \phi_{t - \ctr}).
\end{equation}
The error incurred through the truncation of~\eqref{eqn:qbxlocal} to order $p$
may be as in Lemma~\ref{lem:qbx-truncation-3d}.  (While
the reference~\cite[Thm.~3.1]{epstein:2013:qbx-error-est} discusses the
Helmholtz case, the Laplace case follows analogously.)

\begin{lemma}[QBX truncation error in three dimensions, cf.~{\cite[Thm~3.1]{epstein:2013:qbx-error-est}}]%
  \label{lem:qbx-truncation-3d}%
  Suppose that $\Gamma$ is smooth, non-self-intersecting and
  let $r > 0$.
  Let the local coefficients $\coeffs{L}$ be defined as in~\eqref{eqn:qbxcoeff}
  and the expansion centers $\{\ctr\}_{i=1}^{N_C/2}$ as in~\eqref{eq:expansion-centers}.
  Let $c \in \{c_i^+, c_i^-\}$ be a center for which
  $\overline{B}(r, c) \cap \Gamma = \{ x_i \}$ for some $1 \leq i \leq N_C/2$. Then
  for each $p > 0$ and $\delta > 0$, there is a constant $M_{p,\delta}$ such that
  \begin{equation}
    \left| \mathcal{S}\mu(\pt{x_i}) -
    \sum_{n=0}^p \sum_{m=-n}^n L^m_n
    \norm{x_i - c}^n Y^m_n(\theta_{x_i-c}, \phi_{x_i-c})
    \right| \leq M_{p, \delta} r^{p+1}
    \| \mu \|_{W^{3 + p + \delta, 2}(\Gamma)}.
    \label{eq:truncation-estimate}
  \end{equation}
\end{lemma}
% }}}
% -----------------------------------------------------------------------------
\subsection{Second Approximation Step: Quadrature}%
\label{sec:approx-quadrature}
% -----------------------------------------------------------------------------
% {{{
In the second stage, we apply numerical quadrature to discretize the
integrals for the computation of the expansion coefficients in~\eqref{eqn:qbxcoeff}. We assume that
the smooth, non-self-intersecting surface $\Gamma$ is tessellated into individual,
disjoint surface elements $\Gamma_\iel$ so that
\[
  \Gamma=\bigcup_{\iel=1}^{\nelements} \Gamma_\iel.
\]
Each $\Gamma_\iel$ is described by a smooth mapping function $\mapping_\iel:E \to
\mathbb R^3$, where $E$ is a two-dimensional reference element. We assume that
the mapping Jacobian $\mapping_\iel'$ has full rank everywhere.
The integral~\eqref{eqn:qbxcoeff} can then be split into contributions from
each element as
\[
  (L_{\ictr}^\pm)^m_n
  = \frac{1}{2n+1} \sum_{\iel} \int_{\Gamma_\iel}
  \frac{\mu(s)}{\norm{s-\ctr}^{n+1}}
  Y^{-m}_n(\theta_{s - \ctr}, \phi_{s - \ctr}) \, dS(s),
\]
and, for each element $\Gamma_\iel$, written as an integral over the reference
element $E$ using
\begin{multline}
  \label{eqn:element-integral}
  \int_{\Gamma_\iel}
  \frac{\mu(s)}{\norm{s-\ctr}^{n+1}}
  Y^{-m}_n(\theta_{s - \ctr}, \phi_{s - \ctr}) \, dS(s)
  \\=
  \iint_E \frac{\mu(\mapping_\iel(s_1,s_2))}{\norm{\mapping_\iel(s_1, s_2) - \ctr}^{n+1}}
  \tilde{Y}^{-m}_n(\mapping_\iel(s_1, s_2) - \ctr)
  \surfel{(s_1, s_2)}
  \, ds_1 \, ds_2
\end{multline}
where we have introduced the notation $\tilde{Y}^{-m}_n(x) \coloneqq
Y^{-m}_n(\theta_{x}, \phi_{x})$ for brevity.

The integral~\eqref{eqn:element-integral} may be discretized using quadrature
over the reference element $E$. As an example, we will assume the reference
element is the bi-unit tensor product element $[-1,1]^2$ and consider the
discretization of the integral with a tensor product quadrature rule for smooth
functions. (In practice, our implementation uses a triangular reference element
with nodes and weights based on~\cite{xiao_numerical_2010}.)
A tensor product rule is based on iterated evaluation of a
one-dimensional $q$-point quadrature rule
\begin{equation}
  Q_q\left\{\int_{-1}^1 f(y) \, dy \right\}= \sum_{j=1}^q w_j f(y_j).
  \label{eq:1d-quadrature}
\end{equation}
After repeated application of~\eqref{eq:1d-quadrature}, the
integral~\eqref{eqn:element-integral} becomes
\begin{equation}
  \label{eqn:tensor-product}
  \sum_{j_1=1}^q \sum_{j_2=1}^q w_{j_1} w_{j_2}
  \frac{\mu(\mapping_\iel(y_{j_1},y_{j_2}))}{\norm{\mapping_\iel(y_{j_1}, y_{j_2}) - \ctr}^{n+1}}
  \tilde{Y}^{-m}_{n}(\mapping_\iel(y_{j_1},y_{j_2}) -  \ctr)
  \surfel{(y_{j_1},y_{j_2})}.
\end{equation}
Neglecting geometry, it is straightforward if tedious to obtain estimates of the quadrature
error incurred in~\eqref{eqn:tensor-product}. Such estimates are
roughly analogous to prior results for curves embedded in two
dimensions~\cite{epstein:2013:qbx-error-est}. Compared with the two-dimensional case, the
main difference in the element-wise estimate is the loss of a power of $r$,
owing to the difference in free space Green's functions:
\begin{lemma}[QBX quadrature error for tensor product elements]%
  \label{lem:surface-quad-estimate}
  Let $\Gamma_\iel=[0,h]^2 \times \{ 0 \}$ be a flat, square element and $\mapping_\iel:
  [-1,1]^2 \to [0,h]^2 \times \{0\}$ be given by $\mapping_\iel(x_1, x_2) = \frac{1}{2} h (x_1 + 1,
  x_2 + 1, 0) \in \mathbb{R}^3$.  Let the expansion center be at a distance $r > 0$ from $\Gamma_k$,
  and consider a $q$-point Gauss-Legendre rule with points $\{y_j\}_{j=1}^q$ and weights $\{w_j\}_{j=1}^q$.
  Then there is a constant~$C_{p,q} > 0$ such that for all
  $h > 0$ and $r > 0$
  \begin{multline}%
    \label{eqn:quadrature-est}
    \left|
    \sum_{j_1=1}^q \sum_{j_2=1}^q w_{j_1} w_{j_2}
    \frac{\mu(\mapping_\iel(y_{j_1},y_{j_2}))}{\norm{\mapping_\iel(y_{j_1}, y_{j_2}) - \ctr}^{n+1}}
    \tilde{Y}^{-m}_{n}(\mapping_\iel(y_{j_1},y_{j_2}) -  \ctr)
    \surfel{(y_{j_1},y_{j_2})}
    \right.
    \\
    -
    \left.
    \int_{\Gamma_\iel}
    \frac{\mu(s)}{\norm{s - \ctr}^{n+1}}
    Y^{-m}_n(\theta_{s - \ctr}, \phi_{s - \ctr}) \, dS(s)
    \right|
    \leq
    C_{p,q}
    \left(1+h\right) \left( \frac{h}{4} \right)^{2q+1}
    \frac{1}{r^{n+1}}
    \left[1 + \left(\frac{1}{r}\right)^{2q} \right]
    \| \mu \|_{C^{2q}}.
  \end{multline}
  % NOTE If you update this, also update \ref{eq:quad-estimate-annotated}
\end{lemma}
This estimate pertains to the error in the quadrature contribution of one element
to the coefficient $(L_\ictr^\pm)^m_n$. Similar estimates can be obtained for curved elements,
although one must take into account the effects of the occurring mapping
derivatives, both from the use of the substitution rule, and within the argument
of $\mu$.
When the error contribution of the form~\eqref{eqn:quadrature-est} is summed over all the
elements, this yields a quadrature error estimate for the coefficient $(L_\ictr^\pm)^m_n$. A
factor of $r^{-n}$ in~\eqref{eqn:quadrature-est} is dampened by the term $\norm{t - \ctr}^n$
when evaluating
the summation~\eqref{eqn:qbxlocal} for the local expansion of the single-layer
potential, leaving a quadrature error that scales essentially like
$O((h/(4r))^{2q+1} \| \mu \|_{C^{2q}})$ for small enough $h$ and $r$.
An analysis that yields significantly more precise estimates for tensor product rules over elements
can be found in~\cite{afklinteberg:2016:quadrature-est}.
% }}}
% -----------------------------------------------------------------------------
\subsection{Third Approximation Step: Acceleration}%
\label{sec:approx-acceleration}
% -----------------------------------------------------------------------------
% {{{
The third approximation applied in the rapid, QBX-based evaluation of
layer potentials like~\eqref{eq:slp-definition} arises due to acceleration.
The formation of local expansions~\eqref{eqn:qbxlocal} at all centers
covering a neighborhood of $\Gamma$ requires $O(NM)$ operations, where $N$ is the number
of source points and $M$ is the number of target points. Interpreting~\eqref{eqn:qbxlocal} as the local
expansion of a potential due to a finite set of source charges in space suggests
that such expansion could be amenable to acceleration with the Fast Multipole
Method (FMM).

Recall that the original version of the FMM (e.g.~\cite{greengard:1987:fmm}) is designed to evaluate \emph{point potentials}, which are
potentials of the form
\begin{equation}%
  \label{eqn:point-potential}
  \ptpot(x_i) \coloneqq \sum_{j=1}^N w_j \mathcal{G}(x_i, y_j)
  \quad
  (i=1,\dots,M).
\end{equation}
Here, $\{x_i\}_{i=1}^M \subseteq \mathbb{R}^3$ is the set of target points, and $\{y_j\}_{j=1}^N
\subseteq \mathbb{R}^3$ is the set of source points with the weights
$\{w_j\}_{j=1}^N \subseteq \mathbb{R}$. We shall call FMMs
which evaluate these types of summations `point FMMs' in the remainder of this paper. In
contrast, FMMs directed toward evaluation of QBX expansions for global QBX
(\cite{rachh:2017:qbx-fmm,rahimian:2017:qbkix,gigaqbx2d}) can be described as a
modification of the FMM where QBX centers are treated as a special kind of
target, at which the FMM \emph{forms a local expansion} rather than evaluating
a point potential.

The capability of forming a local expansion of a point potential is an
algorithmic component of the point FMM used in the far field
approximation. Because the algorithmic machinery is already present in the point
FMM, it would be appear to be a fairly natural step to modify the point FMM to
form local expansions at the QBX centers. The first published version of the QBX
FMM~\cite{rachh:2017:qbx-fmm} operationally follows the point FMM algorithm, but
replaces the point evaluations at the QBX centers with the formation of a local
expansion. In particular, this allows it to reuse the intermediate local
expansions formed by the FMM for purposes of QBX evaluation.

Unfortunately, for a given FMM expansion order, the accuracy attained by the
algorithm in~\cite{rachh:2017:qbx-fmm} is generically lower than what would be
expected of a point FMM\@. Specifically, one does not observe a purely additive
error of magnitude proportional to $c^{\pfmm+1}$, where $c$ is a
convergence factor (e.g.\ $c=1/2$ for the Laplace FMM in
two dimensions~\cite{petersen:1995:fmm-error-est} or $c=3/4$ for the Laplace FMM
in three dimensions~\cite{petersen_error_1995}) and $\pfmm$ is the
approximation order used. The reason for this loss of accuracy is discussed in detail
in~\cite{gigaqbx2d}. In short, for accurate evaluation of expansions of the
form~\eqref{eqn:qbxlocal}, \emph{all} local coefficients need to be adequately
approximated. This entails the ability not only to approximate a point potential
but also \emph{its derivatives} to a certain order of accuracy. The point FMM
was not designed with the goal of providing accuracy estimates for this
evaluation pattern. The procedure suggested in~\cite{rachh:2017:qbx-fmm} is to
set the FMM order to $\pfmm' = \pfmm + \padd$, where $\pfmm$ is the FMM order
required for the point FMM to achieve a specified tolerance, and $\padd > 0$ is
an empirically determined quantity that depends on $\pfmm$ and the accuracy
tolerance. This strategy works in practice, although it comes with some
disadvantages. First, higher order multipole and local expansions are expensive
and more difficult to implement stably~\cite{gimbutas_fast_2009}.  Second,
error estimates covering this use are, to the best of our knowledge, not
available.

In~\cite{gigaqbx2d}, an extension to the QBX FMM of~\cite{rachh:2017:qbx-fmm}
was developed to provide accuracy guarantees similar to the point FMM. This was
done starting with the analytical result that QBX disks act like `targets with
extent.' What this suggests in practice is that the `near field' of a QBX disk
should be redefined to be proportional to the size of the disk, so that the
field in a region nearby each disk is evaluated directly. The tree that is built
over the computational domain needs to be aware of this change.

Perhaps the conceptually simplest scheme that matches the accuracy of the point
FMM is to enforce that a QBX disk (in the two dimensional context, or ball in
three dimensions) must be contained entirely inside the box that owns the disk,
so that any QBX disk that cannot fit in a child box remains in the parent
box. However, enforcing that a QBX disk must fit entirely inside a box is
computationally expensive, since disks may be suspended at `high' levels of the
tree (near the root) leading to direct interactions with large parts of the
geometry. To reduce the cost associated with suspending QBX disks at high
levels, the algorithm in~\cite{gigaqbx2d} allows QBX disks to protrude beyond
their boxes by a fixed multiple of the box size. This allows them to settle
lower down the tree (away from the root) at a level where the box size is
commensurate with their diameter. This modification, which was termed a
\emph{target confinement rule}, retains the linear scaling of the FMM under mild
assumptions on the geometry, while also permitting for control over the error
introduced by acceleration.

The scheme described in this paper is a generalization and enhancement of the
scheme described in~\cite{gigaqbx2d}, which we term the `\algbrand\ FMM', for
`Geometric Global Accelerated QBX'. In Section~\ref{sec:accuracy-control}, we
describe a framework for ensuring accuracy in the application of QBX with no
acceleration over arbitrary smooth
geometries. Error estimates for FMM translations are derived in
Section~\ref{sec:error-estimates}. The algorithm is described in
Section~\ref{sec:algorithm}. We close with numerical experiments in
Section~\ref{sec:results}.
% }}}

% #############################################################################
\section{Accuracy Control for QBX on Surfaces}%
\label{sec:accuracy-control}
% -----------------------------------------------------------------------------

\begin{DIFnomarkup}
  % arguments: xshift, name of geometry
  \newcommand{\plotgeometry}[2]{
    \draw[samples=100] plot ({#1+finex(\x)}, {finey(\x)}) node[black,right] {};
    \draw[samples=100] plot ({#1+coarsex(\x)}, {coarsey(\x)}) node[black,right] {#2};
    \foreach \t in {-0.9, -0.675, -0.45, 0, 0.9} {
      \node[fill, coord, black] at ({#1+finex(\t)}, {finey(\t)}) {};
    }
  }

  % args: shift, t
  \def\coarsept#1#2{{#1+coarsex(#2)},{coarsey(#2)}}
  \def\finept#1#2{{#1+finex(#2)},{finey(#2)}}

  % args: shift, t, radius
  \def\coarsectr#1#2#3{{#1+coarsex(#2)+#3*coarsenx(#2)},{coarsey(#2)+#3*coarseny(#2)}}

  % parametric location of QBX center
  \def\qbxctr{-0.5}
  % QBX radii
  \def\initialradius{0.4}
  \def\finalradius{0.175}

  % space between left and right geometries
  \def\geometryspacing{0.5}

  % shift amounts for left and right geometries
  \def\geometrywidth{coarsex(1)-coarsex(-1)}
  \def\lshift{(-\geometrywidth-\geometryspacing/2)}
  \def\rshift{\geometryspacing/2}

  % minimum size multiplier to get correct size
  \def\radiusfudgefactor{8*0.622}

  \begin{figure}
  \centering
  \begin{tikzpicture}[scale=2.5,domain=-1:1]
    % plot geometries
    \plotgeometry{\lshift}{$\Gamma$};
        \plotgeometry{\rshift}{$\Gamma^\text{stage-1}$};

    % left panel boundaries
    \foreach \t in {-0.9, -0.2, 0.9} {
      \node[fill, coord] at (\coarsept{\lshift}{\t}) {};
    }

    % right panel boundaries
    \foreach \t in {-0.9, -0.55, -0.2, 0.9} {
      \node[fill, coord] at (\coarsept{\rshift}{\t}) {};
    }

    % left QBX disk
    \coordinate[draw, coord, qbxcolor] (src) at (\coarsept{\lshift}{\qbxctr});
    \coordinate[draw, coord, qbxcolor] (ctr)
      at (\coarsectr{\lshift}{\qbxctr}{\initialradius});
    \draw[dashed, ->, qbxcolor] (src) -- (ctr);
    \node[draw, qbxcolor, circle,
      minimum size=\radiusfudgefactor*\initialradius cm] at (ctr) {};
    \draw[qbxcolor, ->] (-1.25, 0.45) node[right] {\small (source point)} -- (src);
    \draw[qbxcolor, ->] (-1.25, 0.6) node[right] {\small (QBX center)} -- (ctr);

    % interfering geometry label
    \draw[badcolor, ->]
      (-1, 0.9) node[above,badcolor]  {interfering geometry}
      to[bend left] (\finept{\lshift}{-0.75});

    % refinement flag
    \draw[flagcolor] ({\lshift+0.15},0) rectangle ({\lshift+coarsex(-0.2)-0.05},0.1);
    \draw[->, flagcolor] (-1.2, 0.3)
      node[right] {flagged for bisection} -- ({\lshift+coarsex(-0.2)-0.05}, 0.05);

    % refinement arrows
    \coordinate[coord] (s) at (\coarsept{\lshift}{-0.55});
    \coordinate[coord] (tl) at (\coarsept{\rshift}{-0.725});
    \coordinate[coord] (tr) at (\coarsept{\rshift}{-0.375});
    \draw[gray, dashed, ->] (s) to[bend right=25] (tl);
    \draw[gray, dashed, ->] (s) to[bend right=25] (tr);

    % right QBX disk
    \coordinate[draw, coord, qbxcolor] (src) at (\coarsept{\rshift}{\qbxctr});
    \coordinate[draw, coord, qbxcolor] (ctr)
      at (\coarsectr{\rshift}{\qbxctr}{\finalradius});
    \node[draw, qbxcolor, circle,
      minimum size=\radiusfudgefactor*\finalradius cm] at (ctr) {};
    \draw[dashed, ->, qbxcolor] (src) -- (ctr);
    \node[draw, qbxcolor, circle,
      minimum size=\radiusfudgefactor*\finalradius cm, outer sep=0pt]
      (qbx2) at (ctr) {};

    % no interference label
    \draw[goodcolor, ->]
      (1, 0.6) node[right] {no interference} to[bend right] (qbx2.north);
  \end{tikzpicture}
  \caption{An illustration of stage-1 refinement for a two dimensional geometry.
    Element (panel) boundaries are show with black dots.
    The element associated with the QBX disk shown is bisected due to the presence
    of interfering geometry.}%
  \label{fig:stage1-refinement}
  \vspace{0.5cm}
  \begin{tikzpicture}[scale=2.5,domain=-1:1]
    % plot geometries
    \plotgeometry{\lshift}{$\Gamma^\text{stage-1}$};
    \plotgeometry{\rshift}{$\Gamma^\text{stage-2}$};

    % left panel boundaries
    \foreach \t in {-0.9, -0.55, -0.2, 0.9} {
      \node[fill, coord] at (\coarsept{\lshift}{\t}) {};
    }

    % right panel boundaries
    \foreach \t in {-0.9, -0.55, -0.2, 0.35, 0.9} {
      \node[fill, coord] at (\coarsept{\rshift}{\t}) {};
    }

    % left QBX disk
    \coordinate[draw, coord, qbxcolor] (src) at (\coarsept{\lshift}{\qbxctr});
    \coordinate[draw, coord, qbxcolor] (ctr)
    at (\coarsectr{\lshift}{\qbxctr}{\finalradius});
    \draw[dashed, ->, qbxcolor] (src) -- (ctr);
    \node[draw, qbxcolor, circle,
      minimum size=\radiusfudgefactor*\finalradius cm,outer sep=0pt]
      (qbx) at (ctr) {};

    % insufficient resolution label
    \draw[->, badcolor] (\coarsept{\lshift}{0.35}) to[out=90,in=45] (qbx.north east);
    \node[badcolor, above] at (-1.3, 0.5) {insufficient resolution};

    % refinement flag
    \draw[flagcolor] ({\lshift+coarsex(-0.2)+0.05}, -0.14)
      rectangle ({\lshift+coarsex(0.9)-0.05},0.04);
    \draw[->, flagcolor] (-0.4, 0.8)
      node[left] {flagged for bisection}
      to[bend left] ({\lshift+coarsex(0.9)-0.05}, 0.04);

    % refinement arrows
    \coordinate[coord] (s) at (\coarsept{\lshift}{0.35});
    \coordinate[coord] (tl) at (\coarsept{\rshift}{0.075});
    \coordinate[coord] (tr) at (\coarsept{\rshift}{0.625});
    \draw[gray, dashed, ->] (s) to[bend right=15] (tl);
    \draw[gray, dashed, ->] (s) to[bend right=15] (tr);

    % right QBX disk
    \coordinate[draw, coord, qbxcolor] (src) at (\coarsept{\rshift}{\qbxctr});
    \coordinate[draw, coord, qbxcolor] (ctr)
      at (\coarsectr{\rshift}{\qbxctr}{\finalradius});
    \node[draw, qbxcolor, circle,
      minimum size=\radiusfudgefactor*\finalradius cm] at (ctr) {};
    \draw[dashed, ->, qbxcolor] (src) -- (ctr);
    \node[draw, qbxcolor, circle,
      minimum size=\radiusfudgefactor*\finalradius cm,outer sep=0pt]
      (qbx) at (ctr) {};

    % resolved label
    \draw[->, goodcolor] (\coarsept{\rshift}{0.075})
      to[out=90,in=45] (qbx.north east);
    \draw[->, goodcolor] (\coarsept{\rshift}{0.625})
      to[out=90,in=45] (qbx.north east);
    \node[goodcolor, above] at (1.3, 0.55) {resolved};
  \end{tikzpicture}
  \caption{An illustration of stage-2 refinement for a two dimensional geometry,
    continuing from the previous figure. The element adjacent to
    the QBX disk shown is bisected because of insufficient quadrature
    resolution.}
  \label{fig:stage2-refinement}
  \end{figure}
\end{DIFnomarkup}

% {{{

Since the cost of computational methods dealing with
three-dimensional geometries is typically far greater than that of methods
applied to two-dimensional geometries, and since that cost is directly
related to the resolution supplied, it is not surprising that careful control of
resolution and accuracy plays an important role in maintaining efficiency.
For QBX, two related conditions must be satisfied to ensure
accurate evaluation of the layer potential at any point in $\mathbb{R}^3$. First, the
truncation and quadrature errors at the QBX centers must be adequately
controlled, even given a geometry that does not necessarily satisfy the preconditions of
Lemma~\ref{lem:qbx-truncation-3d} and
Lemma~\ref{lem:surface-quad-estimate}. Second, every target needing QBX
evaluation must be associated to a QBX center.

In this section, we describe a computational framework for establishing these
conditions. While our presentation focuses on the setting of source
\emph{surfaces} embedded in three-dimensional space, the described approach
has an immediate analog for curves embedded in two-dimensional space, permitting
the computationally unified treatment of both cases.

The prior geometry processing scheme introduced in~\cite{rachh:2017:qbx-fmm} had
the potential to cause what one might call a `chain reaction' of refinements,
where a refinement based on insufficient quadrature resolution might trigger an
element bisection, in turn moving expansion centers associated with the
bisected elements, which might trigger further resolution-based refinements, and
so on, in particular on surfaces. The main contribution of this section is a
\emph{multi-stage} approach that not only separates concerns between different
causes for refinement, but also entirely avoids unnecessary `chain reactions'
between them.

% }}}

% -----------------------------------------------------------------------------
\subsection{Overview}%
% -----------------------------------------------------------------------------

We commence our discussion with an outline of a procedure for efficiently
detecting and remedying potential sources of truncation and quadrature
inaccuracy in arbitrary smooth geometries.  From an initial, user-supplied,
unstructured mesh, the process creates a set of three related, unstructured
discretizations satisfying different invariants. We term these the
`\emph{stage-1 discretization}', the `\emph{stage-2 discretization}', and the
`\emph{stage-2 quadrature discretization}'.

For concreteness, we describe these discretization in terms of triangles, with
the understanding that generalizations to other types of reference element
(e.g.\ squares) are expected to be straightforward.  The stage-1 and stage-2
discretizations are interpolatory/unisolvent, i.e.\ a unique polynomial in the
mapped polynomial space $P^{\tgtorder}\circ \mapping_\iel^{-1}$ may be
reconstructed from the nodal degrees of freedom, where $\tgtorder$ is the polynomial degree of the
`target function space' in which layer potentials are evaluated. We choose reference unit nodes
following~\cite{vioreanu_spectra_2014} and use nodal values at their mapped
counterparts as degrees of freedom for the representation of the density and the
geometry. Interpolation operators transport information, particularly on-surface
density values, through the discretizations in the following order:
\[
  \text{stage-1} \quad\to \quad\text{stage-2} \quad\to\quad \text{stage-2 quadrature}.
\]

Shortly, we will summarize the primary features of these discretizations and the
algorithms used to obtain them. Afterwards, the remainder of this section
supplies detailed analysis and algorithms, particularly for stage-1 refinement
and stage-2 refinement (and additionally target association).

A key mechanism for maintaining scalability for algorithms in this section which
require examination of non-local portions of the geometry is the \emph{area
  query}, introduced in~\cite{rachh:2017:qbx-fmm}. Briefly, given a spatial
partitioning of the geometry into an octree and, for each box in the octree, a
stored list of adjacent, equal-or-larger \emph{peer boxes}, an area query can
efficiently find the set of leaf boxes intersecting a cubic region $\{\pt x\in
\mathbb{R}^3: \lvert \pt x-\pt c\rvert_\infty \le r\}$, for given $\pt c$ and
$r$. For details, we refer to Section~\ref{sec:area-query-alg} in the appendix.

\paragraph*{Stage-1 Discretization.} Algorithm~\ref{alg:interfering-sources} of
Section~\ref{sec:truncation-error} produces the stage-1 discretization from the user-supplied mesh. The
stage-1 discretization is a locally refined mesh fitting the geometry
description which ensures that~\begin{inparaenum}[(a)]
  \item sufficient resolution to represent the density and the geometry
  is available, and that~
  \item the assumptions of Lemma~\ref{lem:qbx-truncation-3d} are satisfied,
  i.e.\ specifically that the expansion balls of~\eqref{eq:expansion-centers}
  are undisturbed by quadrature sources (cf.~Section~\ref{sec:truncation-error}).
\end{inparaenum}
Expansion radii $r(x)$ are chosen proportional to a resolution measure of the
stage-1 discretization (cf.\ Section~\ref{sec:quantify-quad-res}).  Thus,
locally bisecting triangular elements~\begin{tikzpicture}
  [scale=0.15,y={(0.5cm,0.866cm)}] % sqrt(3)/2
  \draw ++(0,0)--++(1,0)--++(-1,1)--cycle;
  \draw ++(1,0)--++(1,0)--++(-1,1)--cycle;
  \draw ++(0,1)--++(1,0)--++(-1,1)--cycle;
\end{tikzpicture}
associated with expansion balls disturbed by other geometry will shrink the
associated expansion ball, helping to ensure that, potentially after a number of
refinement cycles, the expansion ball clears the interfering
geometry. Figure~\ref{fig:stage1-refinement} gives an illustration of this
bisection process on a portion of a two-dimensional geometry.

The stage-1 discretization also incorporates a novel `\emph{scaled-curvature
criterion}' to control for truncation error based on an empirically effective
heuristic involving the local curvature of the mesh elements. See
Section~\ref{sec:scaled-curvature}.

\paragraph*{Stage-2 Discretization.} The stage-2 discretization is generated using
Algorithm~\ref{alg:ensure-quad-res} of Section~\ref{sec:accurate-quadrature}, starting with the stage-1 discretization.
The role of the stage-2 discretization is to ensure that enough quadrature
resolution is available to satisfy the resolution requirement implied by the
estimate~\eqref{eqn:quadrature-est} when applied between close elements,
i.e.~that the quadrature contribution to the approximation of the layer
potential is asymptotically as accurate from nearby elements as it is from the
element that spawns the QBX center. As with the stage-1 discretization, this
discretization is obtained through iterative bisection of offending source
elements. As an illustration of the potential issues that this discretization
controls for, consider the two-dimensional geometric situation depicted in
Figure~\ref{fig:stage2-refinement}. In this figure, the situation illustrated on
the left leads to inaccuracies as the contribution from the large source element
is not adequately resolved relative to the size of the target QBX disk, and
bisection suffices to ensure adequate resolution.

\paragraph*{Stage-2 Quadrature Discretization.} The stage-2 quadrature discretization
results from oversampling (i.e., increasing the order) of the source quadrature
nodes of the stage-2 discretization. (Thus, the stage-2 quadrature
discretization shares the same mesh as the stage-2 discretization.) The stage-2
quadrature discretization is optimized for the highest possible quadrature order
achievable at a given node count, to control the quadrature error in
Lemma~\ref{lem:surface-quad-estimate}, at the expense of unisolvence. In three
dimensions, our implementation uses quadrature nodes and weights for the
triangle based on~\cite{xiao_numerical_2010}.

\paragraph*{Target Association.} Lastly, we require a tool to compute a mapping from
targets needing QBX evaluation to QBX centers. Algorithm~\ref{alg:target-assoc}
of Section~\ref{sec:target-assoc} provides this capability. Compared with the
similar target association algorithm in~\cite{rachh:2017:qbx-fmm}, this
algorithm presents a simplified procedure for locating sources or QBX centers
close to a given target, at the expense of performing two area queries instead
of one.

% -----------------------------------------------------------------------------
\subsection{Quantifying Quadrature Resolution on Surfaces}%
\label{sec:quantify-quad-res}
% -----------------------------------------------------------------------------
% {{{
At the core of our accuracy control mechanism lies a measurement of quadrature
resolution in the underlying high-order quadrature used to drive QBX\@. In our
case, these are quadrature rules based on~\cite{xiao_numerical_2010}. To
accomplish this measurement, we define a modified element mapping
$\tilde\mapping_\iel:\tilde E \to \mathbb R^3$, where $\tilde E$ is the
`bi-unit' equilateral triangle with vertices
\[
  v_1=\begin{bmatrix}
    -1\\ -1/\sqrt3
  \end{bmatrix},\quad
  v_2=\begin{bmatrix}
    1\\ -1/\sqrt3
  \end{bmatrix},\quad
  v_3=\begin{bmatrix}
    0\\ 2/\sqrt3
  \end{bmatrix}
\]
serving as the modified reference element.
We define a function
\begin{equation}
  \eta_\iel(x)\coloneqq2\sigma_1(\tilde \mapping_\iel'(\tilde \mapping_\iel^{-1}(x)))
  \quad
  \text{for $x\in\Gamma_\iel$},
  \label{eq:stretch-factor}
\end{equation}
where $\sigma_1(A)$ denotes the largest singular value of a matrix $A$.
The factor of two normalizes out the edge length of $\tilde E$.
$\eta_\iel(x)$ computes an approximate local `stretch factor' of the mapping
$\tilde \mapping_\iel$ at the point $x$.
Since $\eta_\iel$ may be discontinuous between adjacent elements, it is only unambiguously
defined when the point $x$ does not lie on the boundary of $\Gamma_\iel$,
necessitating the subscript $\iel$ to avoid ambiguity.
$\eta_\iel$ can serve as an analog of the `speed' of the
one-dimensional parametrization of a curve segment. It is crucial that
$\tilde E$ be equilateral to ensure that $\eta_\iel$ measures resolution
independently of vertex ordering. We further define
\[
  \eta_\iel\coloneqq\max_{x\in\Gamma_\iel} \eta_\iel(x)
\]
as a per-element maximum of the corresponding per-source-point function.

This resolution measure provides the basis for our choice of the expansion radii
\begin{equation}
  r_\iel\coloneqq\frac12 \etastg1\iel.
  \label{eq:expansion-radii}
\end{equation}
The quantity $\etastg1\iel$ is simply $\eta_\iel$ of~\eqref{eq:stretch-factor}
computed in reference to the stage-1 discretization, defined below.
Allowing a rough analogy between the `panel length' $h_\iel$ of~\cite{rachh:2017:qbx-fmm}
and $\eta_k$ makes the choices of expansion radii of~\cite{rachh:2017:qbx-fmm}
coincide with ours.
% }}}

% -----------------------------------------------------------------------------
\subsection{Stage-1 Refinement: Managing Truncation Error}%
\label{sec:truncation-error}
% -----------------------------------------------------------------------------
% {{{
Lemma~\ref{lem:qbx-truncation-3d} requires that the expansion ball be clear of
source geometry except for the target point. For smooth, non-self-intersecting
geometries, our method ensures that this condition is satisfied without explicit
user involvement, through an approach analogous to that in~\cite{rachh:2017:qbx-fmm}.
Algorithm~\ref{alg:interfering-sources} describes the procedure.

Algorithm~\ref{alg:interfering-sources} operates by using area queries to
find all source geometry that protrudes into the QBX expansion balls
(cf.\ Section~\ref{sec:approx-truncation}) and marking the elements
that spawned the obstructed expansion balls for bisection. Bisection will
lead the expansion radius~\eqref{eq:expansion-radii} to shrink by
way of a reduction of $\eta_\iel$, both of which will drop by a factor of 2 as a
result of bisection. This is repeated until no more interfering geometry is
found. To prevent the source point that spawned the center from being found and
causing refinement, we reduce the size of the queried area by a factor of
$\disturbtol$. In practice, we choose $\disturbtol=0.025$.
The discretization appears to be fairly insensitive to the choice
of this parameter, which is plausible given our
chosen quadrature margins (cf.\ Section~\ref{sec:accurate-quadrature}). Values
as large as 0.2 empirically cause little or no loss in accuracy.

\begin{algbreakable}{Bisect source elements whose expansion balls encounter
    interfering source geometry}%
  \label{alg:interfering-sources}
  \begin{algorithmic}
    \REQUIRE{The geometry discretized as a set of targets, sources, and
      expansion centers.}
    \ENSURE{By repeated bisection that the expansion radii $r_\iel$ are
      sufficiently small that $\closedball(r_\iel, \ctr)\cap
      \Gamma=\{x_\ictr\}$ for $x_\ictr\in\Gamma_\iel$ (as
      $\disturbtol \to 0^+$).}
    \vspace{2ex}
    \REPEAT{}
      \STATE{Create an octree on the computational domain containing
        all sources, expansion balls, and targets.}

      \FORALL{expansion balls $\closedball((1-\disturbtol)r_\iel, \ctr)$}
        \STATE{Perform an area query of radius $r_\iel$ centered at $\ctr$.}
        \IF{the query returned a source point $s$ such that $\norm{\ctr - s} < (1-\disturbtol)r_\iel$}
          \STATE{\textbf{Mark} the element containing $x_\ictr$ for bisection.}
        \ENDIF{}
      \ENDFOR{}
      \IF{elements were marked for bisection}
        \STATE{\textbf{Bisect} the marked elements.}
      \ENDIF{}
    \UNTIL{no elements were marked for bisection}
  \end{algorithmic}
\end{algbreakable}

\begin{figure}%
  \begin{minipage}[t]{0.48\textwidth}%
    \centering
    \includegraphics[width=0.8\textwidth]{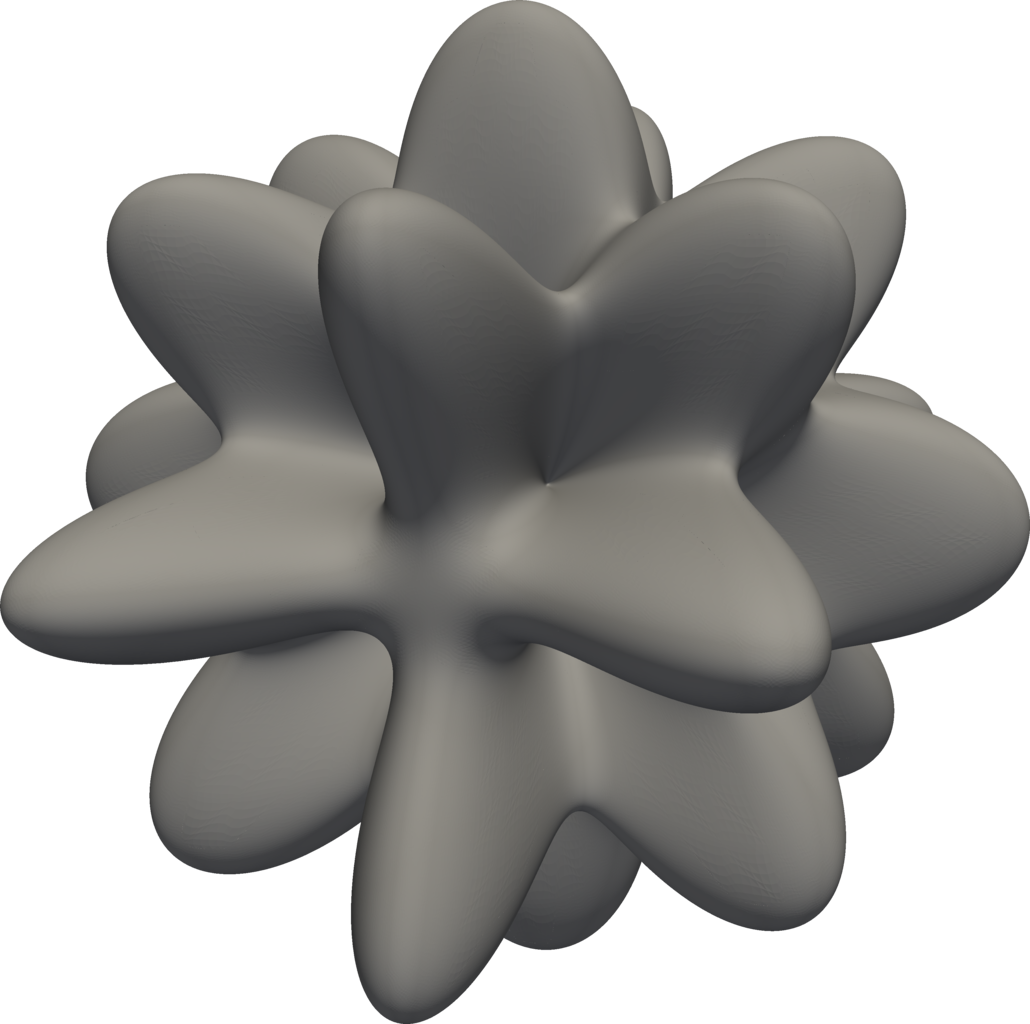}
    \caption{%
      The `urchin' test geometry $\gamma_8$ that we use
      for many computational experiments in this paper.
      See~\eqref{eq:urchin-warping} for the
      warping function used to obtain $\gamma_8$.
      The geometry is represented by an unstructured triangular
      mesh.
      Starting with an icosahedron, the high-order triangular elements are
      repeatedly warped and adaptively bisected to resolve the element mapping
      functions $\mapping_\iel$.  This resolved geometry is then processed
      according to Section~\ref{sec:accuracy-control}.
    }%
    \label{fig:urchin-8}
  \end{minipage}
  \hfill
  \begin{minipage}[t]{0.48\textwidth}%
    \centering
    \includegraphics[width=\textwidth]{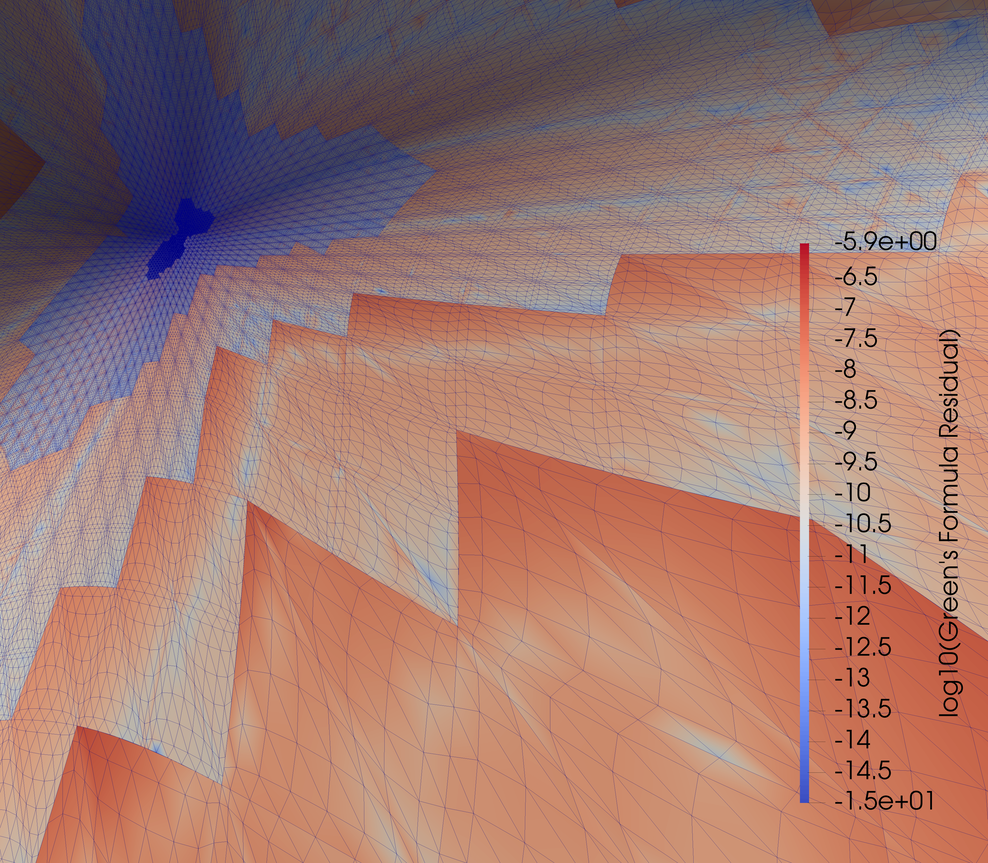}
    \caption{%
      Five levels of scaled-curvature-guided refinement of the stage-1 mesh, shown on a small
      `trough' part of $\gamma_8$.
      See Section~\ref{sec:truncation-error} for details of the refinement method.
      The coloring shows the base-10 logarithm of the residual in Green's
      formula $\mathcal S(\partial_n u)-\mathcal D(u)=u/2$ (i.e., roughly, the
      number of accurate digits). `Red' indicates largest residual,
      corresponding to around six accurate digits.
    }%
    \label{fig:refinement-levels}
  \end{minipage}
\end{figure}

% -----------------------------------------------------------------------------
\subsubsection{Truncation Error and `Scaled-Curvature'}%
\label{sec:scaled-curvature}
% -----------------------------------------------------------------------------

The realization that interference from nearby geometry, through
derivatives of the surface parametrization, contributes to the degradation in
the truncation error motivates a `localized' criterion that controls the `amount
of curvature' on each element of $\Gamma$.

The motivation for this may be most clearly explained by considering QBX in the
two-dimensional case. Consider a smooth closed non self-intersecting curve $\Gamma
\subseteq \mathbb{C}$ with arc length parametrization $w: [0,L] \to \Gamma$ such
that $\closedball(0, r) \cap \Gamma = \{ z \}$ (Figure~\ref{fig:single-layer-2d}). The single-layer potential
evaluated at $z$ for a density function $\sigma: \Gamma \to \mathbb{R}$ takes the form
\[ S \sigma (z) = - \frac{1}{2 \pi} \int_0^L \sigma(w(t)) \log {\lvert w(t) - z \rvert} \, dt. \]
We use the complex-valued logarithm, which satisfies $\Re \log y = \log |y|$ for all $|y| > 0$, to
rewrite the above as
\[
S\sigma (z)
= - \frac{1}{2 \pi} \Re \int_0^L \sigma(w(t)) \log \left( 1 - \frac{z}{w(t)} \right) \, dt
+ S \sigma(0).
\]
By expanding the kernel $\log(1 - z/w(t))$ in a Taylor series, we obtain the
expression for the truncation error in an $n$-th order QBX expansion centered at the
origin, which is:
\[ e_n(z) = - \frac{1}{2 \pi} \Re \int_0^L \sigma(w(t)) \sum_{k=n+1}^\infty \frac{1}{k}
\left( \frac{z}{w(t)} \right)^k \, dt. \]
We integrate this expression by parts $1 \leq p \leq n$ times, assuming $\sigma \in
C^\infty(\Gamma)$.  At each step we replace the term $w^{-k}$ with $-1/((k-1)w')
\partial_t w^{-k+1}$, obtaining~\cite{epstein:2013:qbx-error-est}
\begin{equation*}
e_n(z) = - \frac{1}{2 \pi} \Re \left( (-z)^p \int_0^L
D_t^p[\sigma(w(t))]  \sum_{k=n+1}^\infty \frac{(k-p-1)!}{k!}
\left( \frac{z}{w(t)} \right)^{k-p} \, dt \right)
\end{equation*}
where the differential operator $D_t$ is defined through $D_t g = \partial_t
[g(t)/w'(t)]$. For example, the
values of $D_t^p[\sigma(w(t))]$ for $p = 1$ and $p = 2$ are
\begin{align*}
  % Wolfram alpha:
  % D[s[w[t]]/w'[t], {t, 1}]
  D_t[\sigma(w(t))] &= \sigma'(w(t)) - \frac{\sigma(w(t)) w''(t)}{w'(t)^2}, \\
  % Wolfram alpha:
  % D[D[s[w[t]]/w'[t], {t, 1}]/w'[t], {t,1}]
  D^2_t[\sigma(w(t))] &= \sigma''(w(t))-\frac{2 w''(t) \sigma'(w(t))}{w'(t)^2}-\frac{w^{(3)}(t)
   \sigma(w(t))}{w'(t)^3}+\frac{3 \sigma(w(t)) w''(t)^2}{w'(t)^4}.
\end{align*}
Via these expressions, it is clear that the higher derivatives of
the curve parametrization $w$ have some influence on the truncation error. In the expression
for the truncation error, we assume without loss of generality that $|w'(t)| =
1$. The next higher derivative $w''(t)$, whose magnitude represents the
curvature at parameter $t$, is the first derivative whose magnitude is not controlled.
However, the contribution of the term
$w''(t)$ to the truncation error may be dampened by ensuring that $r$ is chosen
to locally enforce that $|rw''(t)| \le \curvthresh$, for some constant $\curvthresh > 0$, in a
neighborhood of the target.
Many other factors that enter into the
truncation error, so this is at best a heuristic motivation.
Nevertheless, we have been able to develop this insight into a practically useful
criterion.

We realize this criterion in three dimensions as follows. Let $k_1(x)$ and
$k_2(x)$ be the principal curvatures of $\Gamma$ at $x\in\Gamma$. We require
that
\begin{equation}
  \kappa_k(x)\coloneqq\max(|k_1(x)|, |k_2(x)|) \cdot \eta_\iel(x) \le \curvthresh\quad (x\in\Gamma_k),
  \label{eq:curvature-criterion}
\end{equation}
noting that $\kappa_k(x)$ is unit-less
and invariant to scaling. A rough geometric interpretation is that the
condition stipulates that a single
source element may at most cover a certain angle of a tangent circle within
the planes of principal curvature. Since $\eta_\iel(x)$ is reduced by bisecting
elements and since $k_1(x)$ and $k_2(x)$ are independent of parametrization,
$\kappa_k(x)$ can be effectively managed through refinement by bisection.

In our computational experiments, we use
$\curvthresh=0.8$ with good success. As an example, consider the mesh
of Figure~\ref{fig:refinement-levels}, which shows a small section of the test
geometry $\gamma_8$ (shown in Figure~\ref{fig:urchin-8}, see
Section~\ref{sec:results} for a more thorough description). The mesh shown in the
figure exhibits five levels of bisection-based refinement that were triggered
by the criterion~\eqref{eq:curvature-criterion}. The surface coloring in the figure
shows a logarithmic measure of the accuracy with which QBX evaluates layer potentials
on $\Gamma$, where the color red indicates the highest levels of error
encountered, corresponding to roughly six accurate digits. We observe that each
successive level of refinement exhibits growth of the error up to roughly the
`red' level of accuracy, at which point further refinement is triggered.
At least in the scenario shown,~\eqref{eq:curvature-criterion}
exhibits remarkable sharpness and reliability.

We refer to~\eqref{eq:curvature-criterion} as the `\emph{scaled-curvature
criterion}'. Its application requires the reasonably accurate evaluation of two
derivatives of the geometry, which may not be practical in all settings---notably
when $\Gamma$ is discretized using purely affine element mappings.
Nonetheless, the availability of such a criterion yields substantial efficiency
gains in the application of the QBX in high-accuracy settings, leading us to report on its
discovery in this context. Despite strong heuristic motivation and encouraging
computational results, the evidence supporting the criterion is
empirical at this point. Furthermore, since this criterion leads to improvements of the resolution of the integrand,
the application of this criterion also leads
to reduction of the quadrature error, but at present it is not clear which of these
error reduction effects (quadrature or truncation) dominates. We leave a detailed discussion and potential proofs of
its properties for future work.

\begin{DIFnomarkup}
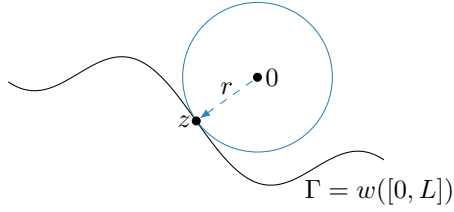
\begin{figure}
  \centering
  \begin{tikzpicture}[scale=2.5,domain=-1:1]
    \def\qbxctr{0}
    \def\radius{0.4}
    \draw[samples=100] plot ({gammax(\x)}, {gammay(\x)}) node[right, yshift=-12pt, xshift=-32] {$\Gamma = w([0,L])$};
    % args: shift, t
    \def\gammapt#1#2{{#1+gammax(#2)},{gammay(#2)}}
    % args: shift, t, radius
    \def\gammactr#1#2#3{{#1+gammax(#2)+#3*gammanx(#2)},{gammay(#2)+#3*gammany(#2)}}
    \coordinate[draw, coord, fill] (ctr) at (\gammactr{0}{\qbxctr}{\radius});
    \node[qbxcolor, draw, circle, minimum size=8*\radius*0.622cm] at (ctr) {};
    \coordinate[draw, coord, fill] (src) at (\gammapt{0}{\qbxctr});
    \draw[<-, dashed, qbxcolor] (src) -- node[midway, above, color=black] {$r$} (ctr);
    \node[right, xshift=1pt] at (ctr) {0};
    \node[left] at (src) {$z$};
  \end{tikzpicture}
  \caption{2D QBX geometric evaluation scenario for the single-layer potential
    $\mathcal{S}\sigma$ in Section~\ref{sec:scaled-curvature}, for a segment of
    the closed curve $\Gamma$.}%
  \label{fig:single-layer-2d}
\end{figure}
\end{DIFnomarkup}

The steps in this section work together to manage truncation error under the
assumption that all coefficient integrals are computed exactly. The
resulting discretization is called the \emph{stage-1 discretization}.
% }}}
% -----------------------------------------------------------------------------
\subsection{Stage-2 Refinement: Ensuring Accurate Coefficient Quadrature}%
\label{sec:accurate-quadrature}
% -----------------------------------------------------------------------------
% {{{
To ensure the accurate computation of the coefficient integrals associated with
the expansion centers spawned by the stage-1 discretization, we introduce a separate
\emph{stage-2 discretization} that may, depending on some criteria, be bisected
into smaller elements than are present in the stage-1 discretization, to provide
additional resolution for the high-order quadrature underlying QBX\@. An analogous
refinement step was present in the geometry processing
in~\cite{rachh:2017:qbx-fmm}, however an important difference between that
scheme and ours is that it used only a single discretization, creating an
artificial interdependence between quadrature-based refinement and
the choice of the expansion radii, impacting the robustness of the refinement
procedure.  Our approach leaves the expansion radii fixed once the stage-1
discretization is determined, removing this unnecessary entanglement of the
two stages.

\newcommand{\stwo}{\text{stage-2}}

In this section, it will be necessary to distinguish between different
refinement iterations of the stage-2 mesh. We
refer to different refinement iterations numbered using a superscript, e.g.\ the
notation for the $k$-th element of the $n$-th iteration is
$\Gamma^{\stwo,n}_k$. Each iteration consists of $K_n$ elements, so that
\( \Gamma = \bigcup_{k=1}^{K_n} \Gamma^{\stwo,n}_k. \)
The initial iteration $\Gamma^{\stwo,1}$ is the same as the stage-1 mesh which
is the output of Algorithm~\ref{alg:interfering-sources}.

\newcommand{\quadorder}{Q}

Let a center $c$ be spawned by some target element $\Gamma^{\stwo,1}_{\iel}$.
We seek to control the quadrature error at $c$ due to a source panel
$\Gamma^{\stwo,1}_j$, with $j \neq \iel$ in general. Our primary concern in this
section is the quadrature error contribution when $\Gamma^{\stwo,1}_j$ is
`close' to $c$, i.e., the minimal Euclidean distance $d_2(\Gamma^{\stwo,1}_{j},
c)$ between the center $c$ and $\Gamma^{\stwo,1}_{j}$ is sufficiently small to
threaten accuracy. We can reexamine and further annotate the quadrature estimate
of Lemma~\ref{lem:surface-quad-estimate} to obtain a quadrature estimate for this
regime. We assume that the quadrature error essentially takes the form
\begin{equation}
  |\text{quadrature error}|
  \le
  C
  \left( \frac{h^{\stwo,1}_j}{4d_2(\Gamma^{\stwo,1}_{j}, c)} \right)^{\quadorder}
  \| \mu \|,
  \quad 1 \leq j \leq K_1
  \label{eq:quad-estimate-annotated}
\end{equation}
where $\quadorder$ is the order of accuracy of the quadrature ($\quadorder=2q+1$
in Lemma~\ref{lem:surface-quad-estimate}).  That is, the main factors governing
the error are the ratio of the source panel size $h^{\stwo,1}_j$ to the center
distance, the density norm $\| \mu \|$ (with the choice of norm depending on the quadrature rule),
and the order of quadrature accuracy $Q$. For simplicity, we may
consider $h^{\stwo,1}_{j}$ asymptotically equivalent to $\etastg{2,1}{j}$, our
quadrature resolution measure.

For concreteness, let a tolerance on the quadrature error \emph{relative to
  $\|\mu\|$} be given by $\epsilon$. We seek to ensure that, for all centers
$c$, refinement produces an iteration $n$ of the discretization such that for
all centers $c \in \{ \ctr \}_{i=1}^{N_C/2}$, we have a
quadrature order $Q$ such that
\begin{equation}
  C
  \left( \frac{h^{\stwo,n}_{j}}{4d_2(\Gamma^{\stwo,n}_{j}, c)} \right)^{\quadorder}
  \le
  C'\left( \frac{\etastg{2,n}{j}}{4d_2(\Gamma^{\stwo,n}_{j}, c)} \right)^{\quadorder}
  \le \epsilon,
  \quad
  1 \leq j \leq K_n.
  \label{eq:quad-accuracy-tolerance}
\end{equation}
The change of leading constant from $C$ to $C'$ absorbs the asymptotic factors
involved when switching from $h^{\stwo,1}_{j}$ to $\etastg{2,1}{j}$.

In the case of the `self-interaction', i.e.\ when the center $c$ was spawned by
a point on the element $\Gamma^{\stwo,1}_{\iel}$, for some $1 \leq \iel \leq
K_1$, we may combine the fact that $d_2(\Gamma^{\stwo,1}_{\iel}, c)=r_\iel$ with
our choice~\eqref{eq:expansion-radii} of $r_\iel$, to obtain
that~\eqref{eq:quad-accuracy-tolerance} implies we must assure
\begin{equation}
  \label{eqn:self-interaction-goal}
  C'
  \left( \frac{\etastg{2,1}{\iel}}{2\etastg1{\iel}} \right)^{\quadorder}
  \le \epsilon.
\end{equation}
For efficiency, we seek to avoid having to refine any element to evaluate its
self-interaction absent other constraints. This implies, temporarily assuming in
this situation that $\etastg{2,1}{\iel}=\etastg1{\iel}$, that there is precisely
one free parameter to attain~\eqref{eqn:self-interaction-goal}, the order of
accuracy of quadrature $\quadorder$.

For only the self-interaction~\eqref{eqn:self-interaction-goal}, it would
suffice to simply choose $\quadorder$ to provide the required accuracy. However,
the global nature of our algorithm compels us to use the same quadrature
resolution for \emph{all} targets, and so this simple strategy is not
necessarily sufficient on its own. Instead, for the benefit of the treatment of
the `non-self interaction' from other elements whose resolution may differ, we
choose a higher value of $\quadorder$ so that the coefficient integrals for a
hypothetical source element larger by a factor of $(4/3)$
would still attain the required level of accuracy in its
coefficient integrals. In other words, we actually choose $Q$ to assure
the stronger bound
\begin{equation}
  C'
  \left( \frac{(4/3)  \etastg{2,1}{\iel}}{2\etastg1{\iel}} \right)^{\quadorder}
  \le \epsilon.
  \label{eq:self-quad-estimate}
\end{equation}

% The `clearance requirement' of Section~\ref{sec:truncation-error}, entails that
% $d(\Gamma_j, c)\ge r_\iel$ for centers spawned by a based point on the element
% $\Gamma_\iel$, so that
% \begin{equation*}
%   C
%   \left( \frac{h_j}{4d_2(\Gamma_j, c)} \right)^{q}
%   \| \mu \|
%   \le
%   C'
%   \left( \frac{\etastg2j}{4r_\iel} \right)^{q}
%   \| \mu \|
%   \quad
%   (j,\iel\in\{1,\dots,\nelements\}),
% \end{equation*}
% ensuring that only elements $\Gamma_j$ with $\etastg2j\ge\etastg1\iel$ are able to pose
% a threat to  the accuracy condition~\eqref{eq:quad-accuracy-tolerance}
% for a center spawned from a base point on $\Gamma_\iel$.

To ensure accuracy of the non-self-interactions, say for all $1 \leq j \leq K_1$, we detect expansion centers
within a distance of $0.5\cdot (3/4)\etastg{2,1}{j}$ from $\Gamma^{\stwo,1}_{j}$ using area
queries originating from source points on each $\Gamma^{\stwo,1}_{j}$.
If any such centers are found, $\Gamma^{\stwo,1}_{j}$ is bisected.
The bisection  reduces $\etastg{2,1}{j}$ by a factor of two, producing
$\etastg{2,2}{j'} = \frac{1}{2} \etastg{2,1}{j}$ for all refined elements $1 \leq j' \leq K_2$ whose parent element is $j$.
Bisection does \emph{not} affect the placement of any
expansion centers, so that only the numerator of the bound~\eqref{eq:quad-estimate-annotated} is affected.

After a sufficient number (say $n$) of iterations, the refinement process
ensures that we always have $d(\Gamma^{\stwo,n}_j,c)\ge 0.5\cdot
(3/4)\etastg{2,n}{j}$, so that
\begin{equation*}
  C
  \left( \frac{h^{\stwo,n}_{j}}{4d_2(\Gamma^{\stwo,n}_{j}, c)} \right)^\quadorder
  \le
  C'
  \left( \frac{\etastg{2,n}{j}}{4 \cdot 0.5 \cdot (3/4) \etastg{2,n}{j}} \right)^\quadorder
  =
  C'
  \left( \frac{(4/3)\etastg{2,n}{j}}{2\etastg{2,n}{j}} \right)^{\quadorder}
  \leq \epsilon,
  \quad
  1 \leq j \leq K_n
\end{equation*}
where $\quadorder$ was chosen above so as to ensure accurate
quadrature in this case, cf.~\eqref{eq:self-quad-estimate} and the text
before it.
One particular consequence of these parameter
choices is that, at a resolution change in the stage-1 discretization
(perhaps as the result of bisection), the stage-2 refinement scheme
just described produces a `buffer zone' of stage-2-refined elements around the
stage-1 refinement fringe.

The halving of $\etastg2{}$ through bisection implies that the set of
`endangered' centers found in the current iteration will be equal to or a
superset of that found in the following iteration.  For smooth,
non-self-intersecting geometries, the associated procedure, detailed in
Algorithm~\ref{alg:ensure-quad-res}, is guaranteed to terminate.  As a last
step, the stage-2 discretization resulting from
Algorithm~\ref{alg:ensure-quad-res} is upsampled to use a sufficient number of
quadrature nodes $\pquad$ to achieve order of accuracy $\quadorder$, obtaining the \emph{stage-2
quadrature} discretization whose nodes are used as source particles for our fast
algorithm, detailed in Section~\ref{sec:algorithm}.

The objective of this contribution is
to clarify the asymptotic relationships between the geometric variables, not to provide
concrete estimates of the constants involved. As such, we give any specific
factors (such as in~\eqref{eq:expansion-radii}) merely for concreteness. We claim that
the choices described here are adequate to illustrate the behavior of the
scheme, and to obtain a practically viable method for layer potential
evaluation. However, we make no claim of optimality for the chosen parameters.
The contribution~\cite{afklinteberg:2016:quadrature-est} contains more precise
bounds suggesting that a fully quantitative understanding may be attainable.
We leave this for future work.

\begin{algbreakable}{Bisect stage-2 source elements until sufficient quadrature
  resolution is available}%
  \label{alg:ensure-quad-res}
  \begin{algorithmic}
    \REQUIRE{The stage-1 discretization has been determined in accordance with
    Section~\ref{sec:truncation-error}.}
    \ENSURE{The quadrature accuracy condition~\eqref{eq:quad-accuracy-tolerance}
      holds for all centers and all source elements $\Gamma_k^\stwo$.}
    \vspace{2ex}
    \STATE{Initialize the stage-2 discretization $\Gamma^\stwo$ to be identical to the
      stage-1 discretization.}
    \REPEAT{}
      \STATE{Create an octree on the computational domain containing
        all source points in the stage-2 discretization and expansion centers.}
      \FORALL{stage-2 elements $\Gamma_k^\stwo$}
      \FORALL{source points $x_i^\stwo\in\Gamma_\iel^\stwo$}
        \STATE{Perform an area query of radius $0.5 \cdot (3/4) \etastg2\iel$ centered at $x_i^\stwo$.}
        \IF{the query returned an expansion center $c$ such that $\norm{c- x_i^\stwo} \leq 0.5 \cdot (3/4)\etastg2\iel$}
          \STATE{\textbf{Mark} the element $\Gamma_k^\stwo$ for bisection.}
        \ENDIF{}
      \ENDFOR{}
      \ENDFOR{}
      \IF{elements were marked for bisection}
        \STATE{\textbf{Bisect} the marked elements.}
      \ENDIF{}
    \UNTIL{no elements were marked for bisection}
  \end{algorithmic}
\end{algbreakable}
% Oversampling

% Two stage discretization
% Refine first stage

% }}}
% -----------------------------------------------------------------------------
\subsection{Associating Targets with QBX Centers}%
\label{sec:target-assoc}
% -----------------------------------------------------------------------------
% {{{
The computed set of quadrature discretizations ensures that the QBX expansions
at the chosen set of centers can be computed accurately (ignoring error from acceleration). A final issue to be
solved by geometry processing is to determine, for each target point, whether
evaluation of the potential with QBX is needed or whether unmodified smooth
high-order quadrature suffices.  Algorithm~\ref{alg:target-assoc} describes a
procedure for associating targets to QBX centers. The algorithm proceeds in two
stages, which consist of \emph{identifying} endangered targets and
\emph{associating} targets to an expansion center using area queries. The
algorithm produces a mapping from targets to associated centers, and it also
flags endangered targets that could not be associated to any centers.

In the first stage, targets that require QBX evaluation are determined based on
their proximity to endangering source
particles. Section~\ref{sec:accurate-quadrature} implies that a `danger
zone' of radius $\rdanger{s} = \eta_\iel/2$ exists around each source particle
$s\in\Gamma_\iel$. Using area queries around each source point of size
$\rdanger{s}$, every target that some source endangers can be identified
efficiently.

In the second stage, an area query around each expansion center
having the same radius as the expansion ball is
used to associate endangered targets to centers. In practice, because the expansion
balls leave gaps in the coverage of the source danger zone, we have found it
useful to allow a target to be matched
to a center $c$ if it is within a ball of radius $r_c(1 +
\tgtassoctol)$, where $\tgtassoctol$ is some tolerance value and $r_c$ is the expansion ball radius.
Although such usage is not necessarily governed by theoretical guarantees,
experience suggests that a small value of $\tgtassoctol$ decreases the chance of
having unassociated endangered targets without appearing to have an adverse
impact on accuracy. If increasing $\tgtassoctol$ still leads to unassociated
targets, another available strategy is to refine near the targets which could
not be associated, or to introduce additional QBX centers.

A natural extension of this algorithm is to include a \emph{side preference} for
each target. This amounts to associating a target only if the side of the center
matches the desired side (interior/exterior) for the target. The need for this
arises when performing on-surface evaluation for layer potentials where the limiting
value depends on the direction of approach. See~\cite{klockner:2013:qbx} for
further discussion.

The algorithm we use in our implementation to obtain the results
Section~\ref{sec:results} is mildly more complicated than
Algorithm~\ref{alg:target-assoc}, owing to being designed to run in
parallel. The main difference is that in our implementation the area query takes
the `point of view' of the targets, rather than the sources or the centers. This
potentially can result in better load distribution when large numbers
of targets are clustered in one part of the geometry. Nevertheless, the output of
our implemented algorithm is functionally identical to output of the algorithm
in this section.

\begin{algbreakable}{%
    Associate near-source/`endangered' target points with QBX
    centers}%
  \label{alg:target-assoc}
  \begin{algorithmic}
    \REQUIRE{The geometry discretized as a set of targets, sources, and
      expansion centers.}
    \REQUIRE{A target association tolerance $\tgtassoctol \geq 0$.}
    \ENSURE{Computes a partial function from targets to
      expansion centers and flags the set of targets that could not
      be associated.}

    \STATE{\vspace{2ex}Create an octree on the computational domain containing
      all sources, targets, and expansion centers.}

    \algstage{Find endangered targets}
    {%
      \FORALL{source points $s\in\Gamma$}
        \STATE{Perform an area query of radius $\rdanger{s}$ centered at $s$.}
        \FORALL{targets $t$ in boxes returned by the query}
          \IF{$\norm{t - s} \leq \rdanger{s}$}
            \STATE{\textbf{Mark} $t$ as endangered.}
          \ENDIF{}
        \ENDFOR{}
      \ENDFOR{}
    }

    \algstage{Find centers for endangered targets}
    {%
      \FORALL{expansion centers $c$}
        \STATE{Perform an area query of radius $r_c (1 + \tgtassoctol)$ centered at $c$.}
        \FORALL{endangered targets $t$ in boxes returned by the query}
          \IF{$\norm{t-c} \leq r_c(1 + \tgtassoctol)$ \AND $c$
            is the closest center to $t$ encountered so far}
            \STATE{\textbf{Associate} $t$ to $c$.}
          \ENDIF{}
        \ENDFOR{}
      \ENDFOR{}
    }

    \algstage{Flag targets that could not be associated}
    {%
      \FORALL{endangered targets $t$}
        \IF{$t$ is not associated to a center}
          \STATE{\textbf{Flag} $t$.}
        \ENDIF{}
      \ENDFOR{}
    }
  \end{algorithmic}
\end{algbreakable}
% }}}

% #############################################################################
\section{Error Estimates for FMM Translations}%
\label{sec:error-estimates}
% -----------------------------------------------------------------------------
% {{{

In~\cite{gigaqbx2d}, error estimates were presented for the \algbrand~FMM that
applied to the 2D Laplace kernel with complex Taylor expansions. In this
section, we present their analogs in three dimensions. We restrict our
attention to the spherical harmonic approximation of the three-dimensional Laplace
potential~\eqref{eq:laplace-fs-green}.
This section lays the groundwork for showing the conditions under which a 3D
\algbrand-style FMM for this potential can be expected to have the same
convergence factor as a 3D point FMM\@.

% }}}

% -----------------------------------------------------------------------------
\subsection{Overview}
% -----------------------------------------------------------------------------

The main difference between `point' estimates for an FMM and the estimates in
this section is that the object being approximated is a local expansion rather
than a `point' value (see
Section~\ref{sec:accuracy-notion-comparison}). Nevertheless, readers familiar
with the error estimates for the point FMM will recognize a number of
similarities in this section with the error estimates from the point case.

\paragraph*{Types of Translations.} First, there is a direct
correspondence in these error estimates with the evaluation scenarios for the
point FMM\@. For instance, see~\cite[Lem.~3.2]{greengard:1988:thesis} for the
local case (cf.~Hypothesis~\ref{hyp:l2p}), \cite[Thm.~3.5.4]{greengard:1988:thesis} for the multipole case (cf.~Hypothesis~\ref{hyp:m2p}),
and~\cite[Thm.~3.5.5]{greengard:1988:thesis} for the multipole-to-local
case (cf.~Hypothesis~\ref{hyp:m2l}).

\paragraph*{`Sized' Targets.} Second, the results in this section
empirically confirm that, for purposes of FMM accuracy, the local expansion that is to be
approximated behaves much like a `sized target'. What this means is that, for a
given evaluation scenario, the accuracy that is expected is similar to the
accuracy expected if the local expansion were a set of targets for point
evaluation.

\paragraph*{Accuracy Dependence on Intermediate Expansion Order.} Last,
the results suggest that the accuracy chiefly depends on the order of the
intermediate multipole/local expansions used and not the `final' order of the
local expansion~(i.e.~the QBX order). The results also
suggest it might be possible to find an estimate independent of the final
expansion order. With regards to the dependence on the order of the intermediate
expansion, the error behavior mirrors the original FMM.

% -----------------------------------------------------------------------------
\subsection{Analytical Preliminaries}
% -----------------------------------------------------------------------------
% {{{

Local expansions have already been introduced in
Section~\ref{sec:qbx-background}. Recall that the $p$-th order local expansion
$L_p$ due to a source $s \in \mathbb{R}^3$ centered at $c \in \mathbb{R}^3$,
with coefficients $\coeffs{L}$ given by~\eqref{eqn:local-expansion-coeff} and
evaluated at a target $t \in \mathbb{R}^3$ takes the form
\[ L_p(t) = \sum_{n=0}^p \sum_{m=-n}^n L^m_n \norm{t-c}^n Y^m_n(\theta_{t-c}, \phi_{t-c}). \]
This expansion converges as long as $\norm{t - c} < \norm{s - c}$.

A \emph{multipole expansion} due to a source $s \in \mathbb{R}^3$ with center $c
\in \mathbb{R}^3$ is defined via the coefficients $\coeffs{M}$ given by
\begin{equation}
  \label{eqn:mpole-expansion-coeff}
   M^m_n \coloneqq \frac{1}{2n+1} \norm{s - c}^n Y^m_n(\theta_{s-c}, \phi_{s-c}).
\end{equation}
The expression $M_p(t)$ for a $p$-th order multipole expansion evaluated at $t \in \mathbb{R}^3$ takes the form
\[
 M_p(t) = \sum_{n=0}^p \sum_{m=-n}^n \frac{M^m_n}{\norm{t-c}^{n+1}} Y^{-m}_n(\theta_{t-c}, \phi_{t-c}).
\]
The multipole expansion converges for $\norm{t - c} > \norm{s - c}$.

\emph{Translation operators} allow for the shifting of centers of expansions, or
the conversion of multipole expansions to local expansions. If the coefficients
of the original expansion are notated as $\coeffs{B}$, the translation operator
amounts to a linear transformation $\coeffs{B} \mapsto \coeffs{(B')}$ to new
coefficients $\coeffs{(B')}$. We will say that the expansion with
coefficients $\coeffs{B}$ has order $p$ if the coefficients $\coeffs{B}$ satisfy
$B^m_n = 0$ for $n > p$. Equivalently, in that case the $\coeffs{B}$ may be
thought of as a vector of $(p + 1)^2$~coefficients~$B_0^0, B_1^{-1}, B_1^0,
B_1^1, \ldots$.

We denote the translation of a local expansion with order $p$ from a center $c
\in \mathbb{R}^3$, to a local expansion with order $q$ at a center $c' \in
\mathbb{R}^3$, by $\coeffs{(B')} = \LToL{c}{c'}{p}{q}
\coeffs{B}$. See~\cite[Lem.~3.2]{greengard:1988:thesis} for the explicit formula
for this operator.

We similarly define the operators $\MToM{c}{c'}{p}{q}$ for
multipole-to-multipole translation and $\MToL{c}{c'}{p}{q}$ for
multipole-to-local translation. An explicit formula for both of these operators
may be found respectively in~\cite[Thm.~3.5.4]{greengard:1988:thesis}
and~\cite[Thm.~3.5.5]{greengard:1988:thesis}.

The following error estimate pertains to the truncation error of multipole and
local expansions.

\begin{proposition}[Accuracy of multipole and local expansions, based on~{\cite[Lem.~3.2.4]{greengard:1988:thesis}}]%
  \label{prop:local-mpole-accuracy}%
  Let $s, c, t \in \mathbb{R}^3$.
  \begin{enumerate}[(a)]
  \item Consider the local expansion of $\mathcal{G}(s, \cdot)$ centered at $c$ and evaluated
  at $t$, with the coefficients $\coeffs{L}$ defined as
  in~\eqref{eqn:local-expansion-coeff}. Let $r = \norm{t-c}$ and $\rho = \norm{s-c}$.
  If $r < \rho$, then
  \[
    \left|
    \mathcal{G}(s, t) - \sum_{n=0}^p \sum_{m=-n}^n L^m_n \norm{t-c}^n Y^m_n(\theta_{t-c}, \phi_{t-c})
    \right| \leq \frac{1}{4 \pi}
    \frac{1}{\rho - r} \left( \frac{r}{\rho} \right)^{p+1}.
  \]
  \item Next, consider the multipole expansion of $\mathcal{G}(s, \cdot)$ centered at $c$ and
  evaluated at $t$, with the coefficients $\coeffs{M}$ defined as
  in~\eqref{eqn:mpole-expansion-coeff}. Let $r = \norm{s-c}$ and $\rho = \norm{t-c}$.
  If $r < \rho$, then
  \[
    \left|
    \mathcal{G}(s, t) - \sum_{n=0}^p \sum_{m=-n}^n \frac{M^m_n}{\norm{t-c}^{n+1}} Y^{-m}_n(\theta_{t-c}, \phi_{t-c})
    \right| \leq \frac{1}{4 \pi}
    \frac{1}{\rho - r} \left( \frac{r}{\rho} \right)^{p+1}.
  \]
  \end{enumerate}
\end{proposition}

% }}}

% -----------------------------------------------------------------------------
\subsection{Accuracy of \algbrand~FMM Translations}
\label{sec:accuracy-notion-comparison}
% -----------------------------------------------------------------------------
% {{{

\begin{figure}
  \def\scaling{0.8}

  \begin{minipage}[t]{0.48\linewidth}
    \begin{tikzpicture}[scale=\scaling, z={(-0.09,-0.09)}]
      % target sphere
      \def\tgtR{2.5}
      \def\tgtX{0}
      \def\tgtY{0}
      \def\tgtZ{0}
      \DrawSphere[localcolor]{\tgtX}{\tgtY}{\tgtZ}{\tgtR}

      \node [circle, minimum size=2*\scaling*\tgtR cm](tgt) at (xyz cs: x=\tgtX, y=\tgtY) {};
      \node [below right, length] at (tgt.south east) {$\closedball(0, R)$};

      % QBX sphere
      \def\qbxR{0.9}
      \def\qbxX{0}
      \def\qbxY{-1.35}
      \def\qbxZ{0}
      \DrawSphere[qbxcolor]{\qbxX}{\qbxY}{\qbxZ}{\qbxR}

      \node [circle, minimum size=2*\scaling*\qbxR cm](qbx) at (xyz cs: x=\qbxX, y=\qbxY) {};

      % source sphere
      \def\srcR{1}
      \def\srcX{-5}
      \def\srcY{0}
      \def\srcZ{0}
      \DrawSphere[srccolor]{\srcX}{\srcY}{\srcZ}{\srcR}

      \node [circle, minimum size=2*\scaling*\srcR cm](src) at (xyz cs: x=\srcX, y=\srcY) {};
      \node [below left, length] at (src.south west) {$\closedball(c, r)$};

      % coordinates
      \coordinate[mark coordinate] (0) at (xyz cs: x=\tgtX, y=\tgtY);
      \draw (0) node[above right, length]{$0$};
      \coordinate[mark coordinate] (c) at (xyz cs: x=\srcX, y=\srcY);
      \draw (c) node[above right, length]{$c$};
      \coordinate[mark coordinate] (cp) at (xyz cs: x=\qbxX, y=\qbxY);
      \draw (cp) node[above right, length]{$c'$};

      % labels
      \draw[<->, length] (c) -- +(\tgtX-\srcX-\tgtR,0)
        node[above, midway]{$\rho$};

      \draw[<->, length] (c) -- +(90:\srcR) node[right, midway]{$r$};
      \draw[<->, length] (0) -- +(90:\tgtR) node[right, midway]{$R$};

      \draw (230:3cm) node[color=qbxcolor] {\text{local}}
        edge[color=qbxcolor, out=0, in=200, ->] (qbx.south west);

      \draw (160:4cm) node[color=srccolor] {\text{multipole}}
        edge[color=srccolor, out=270, in=45, ->] (src.north east);
    \end{tikzpicture}%
    \caption{Geometric depiction of the use of a multipole expansion to
      approximate the local expansion of a potential. The multipole expansion is
      formed at the center $\pt{c}$ and translated to $\pt{c'}$. This provides
      the geometric setting for the situations described
      in~Sections~\ref{sec:mpole-accuracy}~and~\ref{sec:m2l-accuracy}.}%
    \label{fig:mpole}
  \end{minipage}
  \hfill
  \begin{minipage}[t]{0.48\linewidth}
    \begin{tikzpicture}[scale=\scaling, z={(-0.09,-0.09)}]
      % target sphere
      \def\tgtR{2.5}
      \def\tgtX{0}
      \def\tgtY{0}
      \def\tgtZ{0}
      \DrawSphere[localcolor]{\tgtX}{\tgtY}{\tgtZ}{\tgtR}

      \node [circle, minimum size=2*\scaling*\tgtR cm](tgt)
        at (xyz cs: x=\tgtX, y=\tgtY) {};
      \node [below right, length] at (tgt.south east) {$\closedball(0,r)$};

      % QBX sphere
      \def\qbxR{0.9}
      \def\qbxX{0}
      \def\qbxY{-1.35}
      \def\qbxZ{0}
      \DrawSphere[qbxcolor]{\qbxX}{\qbxY}{\qbxZ}{\qbxR}

      \node [circle, minimum size=2*\scaling*\qbxR cm](qbx)
        at (xyz cs: x=\qbxX, y=\qbxY) {};

      \def\srcX{-5}
      \def\srcY{0}

      % coordinates
      \coordinate[mark coordinate] (0) at (xyz cs: x=\tgtX, y=\tgtY);
      \draw (0) node[right, length]{$0$};
      \coordinate[mark coordinate] (src) at (xyz cs: x=\srcX, y=\srcY);
      \node [below, length] at (src.south) {$s$: source};
      \coordinate[mark coordinate] (cp) at (xyz cs: x=\qbxX, y=\qbxY);
      \draw (cp) node[above right, length]{$c$};

      % arrows
      \draw[<->,length] (src) -- (0,0)
        node[above, midway, xshift=-0.5cm]{$\rho$};
      \draw[<->,length] (0,0) -- (90:\tgtR) node[right, midway]{$r$};

      % labels
      \draw (210:4cm) node[color=qbxcolor] {\text{translated local}}
        edge[bend right, ->, out=0, in=180+45, color=qbxcolor]
        (qbx.south west);
      \draw (160:3.5cm) node[color=localcolor] {\text{local}}
        edge[bend left, ->, color=localcolor]
        (tgt.north west);
    \end{tikzpicture}
    \caption{Geometric depiction of the use of an intermediate local expansion
      to approximate the local expansion of a potential. The local expansion is
      formed at the origin and translated to $\pt{c}$. This provides the
      geometric setting for the situation described
      in~Section~\ref{sec:local-accuracy}.}%
    \label{fig:local}
  \end{minipage}
\end{figure}

Recall the difference between the notion of accuracy in a point FMM and the notion of
accuracy used in the \algbrand~FMM\@. A point FMM computes the value of a
potential at a point $x$. The accuracy is measured by
\begin{equation}
  \label{eqn:point-fmm-error}
  \text{Point FMM Accuracy} = \left| \ptpot(\pt{x}) - \tilde{\ptpot}_p(\pt{x}) \right|,
\end{equation}
where $\tilde{\ptpot}_p$ is the point FMM's $p$-th order approximation to
the point potential $\ptpot$~\eqref{eqn:point-potential}.

In contrast, the correct error metric to use in the \algbrand~FMM measures the
FMM's ability to approximate the $q$-th order local expansion of the potential
at a generic target $\pt{x}$ from a generic center $\pt{c}$. Let $\coeffs{L}$ denote
the coefficients of the local expansion of $\ptpot$ centered at $c$, and let
$\coeffs{(\tilde{L}_p)}$ denote the coefficients of the local expansion of the
$p$-th order approximation $\tilde{\ptpot}_p$ centered at $c$. Then the accuracy
may be measured by
\begin{equation}
  \label{eqn:gigaqbx-fmm-error}
  \text{\algbrand~FMM Accuracy} = \left| \sum_{n=0}^q \sum_{m=-n}^n
  (\tilde{L}_p)^m_n \norm{x-c}^n Y^m_n(\theta_{x-c}, \phi_{x-c}) - \sum_{n=0}^q
  \sum_{m=-n}^n L^m_n \norm{x-c}^n Y^m_n(\theta_{x-c}, \phi_{x-c}) \right|.
\end{equation}
The formulas~\eqref{eqn:gigaqbx-fmm-error} and~\eqref{eqn:point-fmm-error} are
related in that the point FMM error~\eqref{eqn:point-fmm-error} is a special
case of~\eqref{eqn:gigaqbx-fmm-error} with an expansion radius of zero. The error
estimates in this section will be of the form~\eqref{eqn:gigaqbx-fmm-error}.

Given the intended usage pattern in the FMM, and using an intermediate expansion
order $p \in \mathbb{N}_0$ and target (QBX) expansion order $q \in \mathbb{N}_0$, we
present a study of the error for the following translation chains:
\begin{itemize}
\item Source $\to$ Multipole($p$) $\to$ Local($q$) (Section~\ref{sec:mpole-accuracy})
\item Source $\to$ Local($p$) $\to$ Local($q$) (Section~\ref{sec:local-accuracy})
\item Source $\to$ Multipole($p$) $\to$ Local($p$) $\to$ Local($q$)
  (Section~\ref{sec:m2l-accuracy}).
\end{itemize}
In this analysis, it suffices to consider at most one intermediate local or
multipole expansion of order $p$. This is because the potential that results via
a sequence of $p$-th order local-to-local translations only depends on the
source and the initial expansion center. Similarly, the potential that results
via a sequence of $p$-th order multipole-to-multipole translations only depends
on the source and the final expansion center.

We refer the reader to Appendix~\ref{sec:fmm-translation-experiments} for the
details of the numerical experiments used to obtain the results in this section.
\begin{remark}[Error estimates for multiple sources]
  The following sections work with a single unit-strength source charge, but can be
  straightforwardly extended to the case of an ensemble of $m$ charges
  $\pt{s}_1, \pt{s}_2, \ldots, \pt{s}_m$ with strengths $q_1, q_s, \ldots,
  q_m$. The corresponding error bound is scaled by $\sum_{i=1}^m |q_i|$.
\end{remark}

% }}}

% -----------------------------------------------------------------------------
\subsubsection{Multipole Accuracy}%
\label{sec:mpole-accuracy}
% -----------------------------------------------------------------------------
% {{{

Recall that a multipole expansion is an `outgoing' approximation to the field
due to a set of sources at any point sufficiently far away from the expansion
center. In this section, we consider the ability of a local expansion obtained through translation from a multipole
expansion to approximate the local expansion of a potential. We make use of the
following geometric situation, illustrated in Figure~\ref{fig:mpole}.
Let $R > 0$ and $\rho > r > 0$. Consider a closed ball of radius $r$ centered at $c$, with
$\norm{c} = R + \rho$, containing a unit-strength source~$s$.
Also let a ball of radius~$R$ centered at the origin contain points~$t,
c'$ satisfying $\norm{c'} \leq R$ and $\norm{c' - t} \leq R - \norm{c'}$.

\def\MultipoleSource{T_p^c}
\def\MultipoleTranslated{T_q^{c'}}
\def\TrueLocal{L_q^{c'}}

Suppose a $p$-th order multipole expansion $M_p$ with coefficients
$\coeffs{(\MultipoleSource)}$ is formed at~$c$ due to the source~$s$. Next
suppose that this is translated to a $q$-th order local expansion with
coefficients $\coeffs{(\MultipoleTranslated)} = \MToL{c}{c'}{p}{q}
\coeffs{(\MultipoleSource)}$. The local expansion of the potential $\mathcal{G}(s, \cdot)$ due
to $s$ centered at $c'$ may be written using the local coefficients
$\coeffs{(\TrueLocal)}$. If one were to use the coefficients
$\coeffs{(\MultipoleTranslated)}$ as approximations to the coefficients
$\coeffs{(\TrueLocal)}$ for evaluation of the local expansion at the target $t$,
the approximation error would be the quantity $E_M(q)$ defined as
\begin{equation}
  \label{eqn:m2p-error}
  E_M(q) = \left|
  \sum_{n=0}^q \sum_{m=-n}^n
  (\MultipoleTranslated)^m_n \norm{t-c'}^n Y^m_n(\theta_{t-c'}, \phi_{t-c'})
  -
  \sum_{n=0}^q \sum_{m=-n}^n
  (\TrueLocal)^m_n \norm{t-c'}^n Y^m_n(\theta_{t-c'}, \phi_{t-c'}) \right|.
\end{equation}

The following observation suggests a bound on $E_M$. Define the error function
\[
  \mathcal E(x) = \mathcal{G}(s, x) - M_p(x).
\]
Observe that the quantity~\eqref{eqn:m2p-error} is the magnitude of the $q$-th
order local expansion of $\mathcal E$ centered at $c'$ and evaluated at $t$. It
can be shown that this expansion must converge, in the sense that~$ \lim_{q \to
  \infty} E_M(q) = \mathcal E(t)$.  Therefore, by
Proposition~\ref{prop:local-mpole-accuracy}, we may expect that for
every $\epsilon > 0$, it is the case that for all sufficiently large $q$ that
\[
  E_{M}(q) \leq (1 + \epsilon) \left((4\pi)^{-1}/(\rho - r) \right) \left(r/\rho\right)^{p+1}.
\]
It is also easy to see that~$E_M(0) \leq ((4\pi)^{-1}/(\rho - r)) (r/\rho)^{p+1}$,
as the $0$-th~order expansion is just the quantity $\mathcal E(\pt{0})$, which satisfies this
bound by Proposition~\ref{prop:local-mpole-accuracy}. Given this asymptotic behavior
and the fact that it holds for $q = 0$, it is at
least plausible that the bound~$E_M(q) \leq C ((4\pi)^{-1}/(\rho - r)) (r/\rho)^{p+1}$, for
some $C > 0$, should hold for all~$q$. We formulate this statement as the
following hypothesis.
\begin{hypothesis}[Source $\to$ Multipole($p$) $\to$ Local($q$)]%
  \label{hyp:m2p}
  For the situation described above, there exists a constant~$C > 0$ independent
  of~$R$,~$p$,~$q$,~$\rho$,~$r$,~$s$,~$c'$,~and~$t$ such that the error in the
  multipole-mediated approximation to the local expansion of
  the potential satisfies the bound
  \[
    \left|
    \sum_{n=0}^q \sum_{m=-n}^n
    (\MultipoleTranslated)^m_n \norm{t-c'}^n Y^m_n(\theta_{t-c'}, \phi_{t-c'})
    -
    \sum_{n=0}^q \sum_{m=-n}^n
    (\TrueLocal)^m_n \norm{t-c'}^n Y^m_n(\theta_{t-c'}, \phi_{t-c'}) \right|
    \leq
    \frac{1}{4 \pi} \frac{C}{\rho - r} \left(\frac{r}{\rho}\right)^{p+1}.
  \]
\end{hypothesis}
This hypothesis would imply that multipole translation in the \algbrand~FMM
obeys essentially the same bound as the corresponding point FMM error.  This
bound is true if and only if the approximations to $\mathcal E$ never
significantly `overshoot' the true value of $\mathcal E(\pt{t})$. The
mathematical details of the proof of this hypothesis turn out to be fairly
involved and its complete mathematical resolution is pending. However the
numerical evidence is strongly in favor of this bound. We found, as described in
Appendix~\ref{sec:fmm-translation-experiments}, that the inequality in
Hypothesis~\ref{hyp:m2p} holds numerically with $C = 1.002$. Although the tests
we performed cannot be exhaustive, the small value of~$C$ obtained by numerical
testing is consistent with the truth of the hypothesis.  We
use~Hypothesis~\ref{hyp:m2p} as an error estimate in the remainder of this paper
and we expect to present a proof of this hypothesis at a later date.
% }}}

% -----------------------------------------------------------------------------
\subsubsection{Local Accuracy}%
\label{sec:local-accuracy}
% -----------------------------------------------------------------------------
% {{{

Recall that the role of the local expansion is complementary to the multipole
expansion, since the local expansion represents the potential in a neighborhood
of an expansion center. This section considers the case where a local expansion
is formed due to a potential at one center, subsequently translated to a second center, and used to approximate the
local expansion of the potential at the second center.  We make use of the following geometric
situation, illustrated in Figure~\ref{fig:local}. Let $\rho > r > 0$ and
suppose that a source particle is placed at $s$, with $\norm{s} = \rho$.  Let
$t, c \in \closedball(0, r)$ with $\norm{t - c} \leq r - \norm{c}$.

\def\LocalSource{T_p^0}
\def\LocalTranslated{T_q^c}
\def\TrueLocal{L_q^{c}}

The potential due to the source $s$ can be described in a $q$-th order local expansion
centered at $c$ with coefficients $\coeffs{(\TrueLocal)}$. Consider a $p$-th
order local expansion of the potential centered at the origin with coefficients
$\coeffs{(\LocalSource)}$. Suppose this expansion is translated to a $q$-th
order expansion at $c$ with coefficients $\coeffs{(\LocalTranslated)}$ given by
$\coeffs{(\LocalTranslated)} = \LToL{0}{c}{p}{q} \coeffs{(\LocalSource)}$. If the
coefficients $\coeffs{(\LocalTranslated)}$ are used in place of
$\coeffs{(\TrueLocal)}$ as an approximation to the $q$-th order expansion for
evaluation at a target $t$, the approximation error is the quantity $E_L(q)$
given by
\begin{equation}%
  \label{eqn:l2p-error}
  E_L(q) = \left|
  \sum_{n=0}^q \sum_{m=-n}^n
  (\LocalTranslated)^m_n \norm{t-c}^n Y^m_n(\theta_{t-c}, \phi_{t-c})
  -
  \sum_{n=0}^q \sum_{m=-n}^n
  (\TrueLocal)^m_n \norm{t-c}^n Y^m_n(\theta_{t-c}, \phi_{t-c}) \right|.
\end{equation}
The asymptotic behavior of $E_L(q)$ is similar to that of $E_M(q)$ in the
previous section. In other words, for fixed $p$, we have $ E_L(0) \leq
((4\pi)^{-1} /(\rho - r)) (r / \rho)^{p+1}$ and, for all $\epsilon > 0$ it is the
case for sufficiently large $q$ that $E_L(q) \leq (1 + \epsilon) ((4 \pi)^{-1}
/(\rho - r)) (r / \rho)^{p+1}$. This motivates the following hypothesis about the
behavior of $E_L(q)$ for all $q$.

\begin{hypothesis}[Source $\to$ Local($p$) $\to$ Local($q$)]%
  \label{hyp:l2p}
  For the situation described above, there exists a constant~$C > 0$ independent
  of~$p$,~$q$,~$\rho$,~$r$,~$s$,~$c$,~and~$t$ such that the error in the
  local-mediated approximation to the local expansion of the potential satisfies
  the bound
  \[
    \left|
    \sum_{n=0}^q \sum_{m=-n}^n
    (\LocalTranslated)^m_n \norm{t-c}^n Y^m_n(\theta_{t-c}, \phi_{t-c})
    -
    \sum_{n=0}^q \sum_{m=-n}^n
    (\TrueLocal)^m_n \norm{t-c}^n Y^m_n(\theta_{t-c}, \phi_{t-c}) \right|
    \leq \frac{1}{4 \pi} \frac{C}{\rho - r} \left(\frac{r}{\rho}\right)^{p+1}.
  \]
\end{hypothesis}

Similar to the previous section, we evaluated the truth of this hypothesis
numerically. We observed that the estimate for $E_L(q)$ empirically
satisfies Hypothesis~\ref{hyp:l2p} with $C = 1.001$.  As a result, we make use
of Hypothesis~\ref{hyp:l2p} as an error estimate in this paper.
% }}}

% -----------------------------------------------------------------------------
\subsubsection{Multipole-to-Local Accuracy}%
\label{sec:m2l-accuracy}
% -----------------------------------------------------------------------------
% {{{

\def\MultipoleSource{T_p^c}
\def\FirstLocalTranslated{T_p^0}
\def\SecondLocalTranslated{T_q^{c'}}
\def\TrueLocal{L_q^{c'}}
\def\TrueLocalAtOrigin{L_p^{0}}

In this section, we work with the same geometrical situation as given in
Section~\ref{sec:mpole-accuracy} and illustrated in Figure~\ref{fig:mpole}.
We consider the
accuracy achieved when a multipole expansion is converted into a local expansion
which is then translated to a third center, and used to approximate the local expansion
of the potential at that center. Similar to
Section~\ref{sec:mpole-accuracy}, consider a $p$-th order multipole expansion
due to the source~$s$ formed at the center~$c$ with
coefficients~$\coeffs{(\MultipoleSource)}$. Suppose this expansion is translated
to a $p$-th order local expansion $\FirstLocalTranslated$ centered at the
origin, with coefficients $\coeffs{(\FirstLocalTranslated)} =
\MToL{c}{0}{p}{p}\coeffs{(\MultipoleSource)}$.  Then suppose that this expansion
is translated to a $q$-th order local expansion centered at another center~$c'$,
with the coefficients $\coeffs{(\SecondLocalTranslated)}$ given by
$\coeffs{(\SecondLocalTranslated)} = \LToL{0}{c'}{p}{q}
\coeffs{(\FirstLocalTranslated)}$.

The coefficients $\coeffs{(\SecondLocalTranslated)}$ are those
of a local expansion that approximates the $q$-th order local expansion of the
potential $\mathcal{G}(s, \cdot)$, centered at $c'$. Let $\coeffs{(\TrueLocal)}$ be the local
coefficients of the approximated expansion. The error is given by
\begin{equation}
  \label{eqn:m2l-error}%
  E_{\mathit{M2L}}(q) = \left|
  \sum_{n=0}^q \sum_{m=-n}^n
  (\SecondLocalTranslated)^m_n \norm{t-c'}^n Y^m_n(\theta_{t-c'}, \phi_{t-c'})
  -
  \sum_{n=0}^q \sum_{m=-n}^n
  (\TrueLocal)^m_n \norm{t-c'}^n Y^m_n(\theta_{t-c'}, \phi_{t-c'}) \right|.
\end{equation}
In the limit as $q \to \infty$ we have
\[
  \lim_{q \to \infty} E_{\mathit{M2L}}(q) = |\mathcal{G}(s, t) - \FirstLocalTranslated(\pt{t})|.
\]
To simplify the analysis of this quantity, we introduce the $p$-th order local
expansion of $\mathcal{G}(s, \cdot)$ at the origin, denoted by $\TrueLocalAtOrigin$.  By the
triangle inequality,
\[
  | \mathcal{G}(s, t) - \FirstLocalTranslated({t}) | \leq
  | \mathcal{G}(s, t) - \TrueLocalAtOrigin({t}) | + | \TrueLocalAtOrigin({t}) - \FirstLocalTranslated({t}) |.
\]
The quantity $|\mathcal{G}(s, t) - \TrueLocalAtOrigin(t)|$ may be bounded by
Proposition~\ref{prop:local-mpole-accuracy}.  Since the quantity $|
\TrueLocalAtOrigin(t) - \FirstLocalTranslated(t)|$ is the difference between the
local expansion of a multipole and the local expansion of the point potential at
the origin, we may bound it with Hypothesis~\ref{hyp:m2p}. This yields the bound
\[
  | \mathcal{G}(s, t) - \FirstLocalTranslated(t) | \leq
  \frac{1}{4 \pi} \left[
  \frac{1}{\rho - r} \left( \frac{R}{R + (\rho - r)} \right)^{p+1} +
  \frac{C}{\rho - r} \left( \frac{r}{\rho} \right)^{p+1}
  \right].
\]
Similar to Sections~\ref{sec:mpole-accuracy} and~\ref{sec:local-accuracy}, this
quantity is an asymptotic bound on $E_\mathit{M2L}$ in $q$. This leads to the
following hypothesis.
\begin{hypothesis}[Source $\to$ Multipole($p$) $\to$ Local($p$) $\to$ Local($q$)]%
  \label{hyp:m2l}
  For the situation described above, there exists a constant~$C > 0$ independent
  of~$R$,~$p$,~$q$,~$\rho$,~$r$,~$s$,~$c$,~$c'$,~and~$t$ such that the error in the multipole and
  local mediated approximation to the local expansion of the potential satisfies
  the bound
  \begin{multline*}
    \left|
      \sum_{n=0}^q \sum_{m=-n}^n
      (\SecondLocalTranslated)^m_n \norm{t-c'}^n Y^m_n(\theta_{t-c'}, \phi_{t-c'})
      -
      \sum_{n=0}^q \sum_{m=-n}^n
      (\TrueLocal)^m_n \norm{t-c'}^n Y^m_n(\theta_{t-c'}, \phi_{t-c'})
    \right| \\
    \leq
    \frac{C}{4\pi} \left[
    \frac{1}{\rho - r} \left( \frac{R}{R + (\rho - r)} \right)^{p+1} +
    \frac{1}{\rho - r} \left( \frac{r}{\rho} \right)^{p+1} \right].
  \end{multline*}
\end{hypothesis}

Numerically, we observed that Hypothesis~\ref{hyp:m2l}
appears to hold with $C = 1.001$. Similar to Sections~\ref{sec:mpole-accuracy}
and~\ref{sec:local-accuracy}, the low value of $C$ that was observed is
consistent with the truth of the hypothesis. We therefore make use of
Hypothesis~\ref{hyp:m2l} as an error estimate in the remainder of this paper.
% }}}

% #############################################################################
\section{The \algbrand~Algorithm in Three Dimensions}%
\label{sec:algorithm}
% -----------------------------------------------------------------------------

This section is concerned with the precise statement of the \algbrand~algorithm
and its complexity and accuracy analysis. The algorithm is presented in
Sections~\ref{sec:gigaqbx-defs} and~\ref{sec:gigaqbx-statement}. The accuracy
and complexity analyses follow in Sections~\ref{sec:accuracy}
and~\ref{sec:complexity}.

% -----------------------------------------------------------------------------
\subsection{Overview}
% -----------------------------------------------------------------------------

For the benefit of readers familiar with the point FMM, the original QBX
FMM~\cite{rachh:2017:qbx-fmm}, or the two-dimensional version
of~\algbrand\ presented in~\cite{gigaqbx2d}, we point out the main differences
in this section.

\paragraph*{Target Confinement Rule.}
To prevent inaccurate contributions from entering the QBX local expansion, while
still maintaining the efficiency enabled by the use of a tree, the design of the
GIGAQBX algorithm adopts the point of view of QBX centers as `targets with
extent' that each have their own near-field. The realization of this idea is
that \algbrand~only permits QBX expansion balls to exist in a box if they do not
extend beyond an ($\ell^2$) radius surrounding the box, where the length of the
radius is proportional to the box size, so that the near-field of the box ends
up being an over-approximation of the near-field of the target with extent.

We call this modification the target confinement rule.  During box subdivision
in tree construction, if a ball cannot be placed in the child box due to this
restriction, it remains in the parent box.

\paragraph*{Particles Owned by Non-Leaf Boxes.}
It follows from the previous paragraph that the \algbrand~algorithm, unlike the point FMM, allows for
particles (specifically, QBX centers) to be `owned' by non-leaf `ancestor'
boxes. The most important implication of this design is that interaction lists
involving direct evaluations at particles (List 1 and List 3), as well as the
FMM step of evaluation of far-field local expansions, must be redefined to
incorporate the possibility of evaluation at non-leaf boxes.

\paragraph*{Two-Away Near Neighborhood.}
To obtain a good convergence factor in two or three dimensions, it is convenient to
consider the `near-field' to consist of both a box's nearest
neighbors and also its second nearest neighbors.
This is not a new modification, having been present in the original
three-dimensional FMM~\cite{greengard:1988:thesis}.

\paragraph*{`Close' and `Far' Lists.}
In order to actually obtain the accuracy guarantees provided by the target
confinement rule, we disallow certain box near-field interactions that are too
close to the target confinement region from using expansion mediation that would
otherwise be mediated by expansions in the point FMM\@. Specifically, the fields
associated with List 3 and List 4 are subdivided into `close' and `far' lists,
where the close lists are evaluated directly via point-to-QBX-local
interactions, and the far lists maintain sufficient separation to allow for
normal expansion mediation.

\paragraph*{Changes from the 2D Version of \algbrand.}
Perhaps the most significant difference with the two-dimensional
version~\cite{gigaqbx2d} is the use of an $\ell^2$ target confinement region to
confine the QBX expansion centers, whereas the previous version used an
$\ell^\infty$ (square) region. The use of the $\ell^2$ region improves the
efficiency of the scheme, especially in three dimensions. Other than this, the
definitions of the interaction lists and the statement of the algorithm itself
are largely unchanged from the two-dimensional case.

% -----------------------------------------------------------------------------
\subsection{Definitions and Interaction Lists}%
\label{sec:gigaqbx-defs}
% -----------------------------------------------------------------------------

% -----------------------------------------------------------------------------
\subsubsection{Computational Domain}
% -----------------------------------------------------------------------------

The computational domain for the algorithm is an octree whose axis-aligned root
box contains all sources, targets, and expansion centers (`particles') as well
as the entirety of each expansion disk. Starting with the root box, the octree
is refined by repeated subdivision of boxes that contain more than $\nmax$
particles until no more non-empty boxes can be produced, pruning any empty
childless boxes. A childless box is also called a \emph{leaf box}. Each box,
including non-leaves, conceptually `owns' a subset of particles. Upon
subdivision, non-expansion center particles are placed from the parent into the spatially
appropriate child box. Expansion centers are only placed in the child box if the
expansion ball can fit within the target confinement region of the child.
Otherwise, they remain owned by the parent.

% -----------------------------------------------------------------------------
\subsubsection{Notation}
% -----------------------------------------------------------------------------
% {{{
For a box $b$ in the octree, we will use $|b|$ to denote the ($\ell^\infty$)
radius of $b$.  The \emph{target confinement region} (`TCR', also $\tcr(b)$) of
a box $b$ with center $c$ is $\closedball(\sqrt{3} |b|(1 + t_f), c)$, where
$t_f$ is the \emph{target confinement factor} (`TCF'),
where $\sqrt{3} |b|$ is half the box diagonal.

The \emph{$k$-near~neighborhood} of a box $b$ with center $c$ is the region
$\closedbox(|b|(1 + 2k), c)$.

The \emph{$k$-colleagues} of a box $b$ are boxes of the same level as $b$ that
are contained inside the $k$-near~neighborhood of $b$.  $T_b$ denotes the set of
$2$-colleagues of a box $b$.

We say that two boxes are \emph{$k$-well-separated} if they are on the same
level and are not $k$-colleagues.

The parent of $b$ is denoted $\parent(b)$. The set of ancestors is
$\ancestors(b)$. The set of descendants is $\descendants(b)$. $\ancestors$ and
$\descendants$ are also defined in the natural way for sets of boxes.

A box owning a point or QBX center target is called a \emph{target box}. A box
owning a source quadrature node is called a \emph{source box}. Ancestors of
target boxes are called \emph{target-ancestor boxes}.

\begin{definition}[Adequate separation relation, $\adequatesep$]
  We define a relation $\adequatesep$ over the set of boxes and target
  confinement regions within the tree, with $a \adequatesep b$ to be understood
  as `$a$ is adequately separated from $b$, relative to the size of $a$'.

  We write $a \adequatesep \tcr(b)$ for boxes $a$ and $b$ if the $\ell^2$
  distance from the center of $a$ to the boundary of $\tcr(b)$ is at least
  $3|a|$.

  We write $\tcr(a) \adequatesep b$ for boxes $a$ and $b$ if the $\ell^\infty$
  distance from the center of $a$ to the boundary of $b$ is at least $3|a|(1 +
  t_f)$. (This implies that the $\ell^2$ distance is at least $3|a|(1 + t_f)$.)

  We write $a \not \adequatesep b$ to denote the negation of $a \adequatesep b$.
\end{definition}

Because the size of the TCR is proportional to the box size, $\parent(a)
\adequatesep b$ implies $a \adequatesep b$. We call this property the
`monotonicity' of `$\adequatesep$'.

% }}}

% -----------------------------------------------------------------------------
\subsubsection{Conventional Interaction Lists}
% -----------------------------------------------------------------------------
% {{{

The four conventional interaction lists in the FMM are defined in this section,
with two modifications to the standard definition.
First, non-leaf boxes are allowed as
target boxes. Thus, lists normally associated with only leaf boxes (Lists 1 and
3) may be associated with arbitrary boxes in the tree. Second, our definition
makes use of a near~neighborhood of a box that is two boxes wide.

List 1 consists of interactions with adjacent boxes. These interactions are
carried out directly, without acceleration through any expansions.

\begin{definition}[List 1, $\ilist{1}{b}$]\label{def:list-1}
  For a target box $b$, $\ilist{1}{b}$ consists of all leaf boxes from among
  $\descendants(b) \cup \{b\}$ and the set of boxes adjacent to $b$.
\end{definition}

List 2 consists of interactions with same-level boxes. These interactions have
sufficient separation for an accurate multipole-to-local mediation. List 2 is
\emph{downward-propagating}, which means that the interactions received by the box are
translated downward to the box's descendants via local-to-local translation.

\begin{definition}[List 2, $\ilist{2}{b}$]\label{def:list-2}
  For a target or target-ancestor box $b$, $\ilist{2}{b}$ consists of the children of the
  $2$-colleagues of $b$'s parent that are $2$-well-separated from $b$.
\end{definition}

List 3 consists of interactions where the source box is in the near field
($T_b$) of the target box, but not adjacent to the target box. Unlike List 2,
the separation is insufficient for accurate multipole-to-local mediation. List
3 is not downward-propagating, and it is usually mediated with a
multipole-to-target interaction.

\begin{definition}[List 3, $\ilist{3}{b}$]\label{def:list-3}
  For a target box $b$, a box $d \in \descendants(T_b)$ is in $\ilist{3}{b}$ if
  $d$ is not adjacent to $b$ and, for all $w \in \ancestors(d) \cap
  \descendants(T_b)$, $w$ is adjacent to $b$.
\end{definition}

The following are immediate consequences of this definition:
\begin{itemize}
  \item Any box in $\ilist{3}{b}$ is strictly smaller than $b$.
  \item Any box $d \in \ilist{3}{b}$ is separated from $b$ by at least the width
    of $d$.
  \item List 3 of $b$ contains the immediate children of non-adjacent
    $2$-colleagues of $b$.
\end{itemize}

List 4 consists of interactions where the target box is in the near field of the
source box, but not adjacent to it. Like List 3, the separation is insufficient
for accurate multipole-to-local interaction. Unlike List 3, List 4 \emph{is}
downward-propagating, and one may form a local expansion of the field from the
source box that can be propagated to the descendants.
\begin{definition}[List 4, $\ilist{4}{b}$]\label{def:list-4}
  For a target or target-ancestor box $b$, a source box $d$ is in List 4 of $b$ if
  $d$ is a $2$-colleague of some ancestor of $b$ and $d$ is adjacent to
  $\parent(b)$ but not $b$ itself. Additionally, a source box $d$ is in
  $\ilist{4}{b}$ if $d$ is a $2$-colleague of $b$ and $d$ is not adjacent to $b$.
\end{definition}
The following are immediate consequences of this definition:
\begin{itemize}
  \item Any box in $\ilist{4}{b}$ is at least as large as $b$.
  \item Any box in $\ilist{4}{b}$ is separated from $b$ by at least the width of
    $b$.
  \item For any $d \in \ilist{4}{b}$, either $b \in \ilist{3}{d}$ or $d$ is a
    $2$-colleague of $b$.
\end{itemize}

Our FMM does not make use of List 3 and List 4 directly. Instead, the field from
these lists is sub-partitioned into `close' and `far' lists. This is explained in
the next section.
% }}}
% -----------------------------------------------------------------------------
\subsubsection{Close and Far Lists}
% -----------------------------------------------------------------------------
% {{{

Because of inadequate separation from the target confinement region, our algorithm
cannot make use of the interaction lists $\ilist{3}{b}$ and $\ilist{4}{b}$ as
they would be used normally: via, respectively, a multipole-to-target or
source-to-local interaction. To manage this issue while still
maintaining the benefits of FMM acceleration, we partition the contribution of
the sources contained in these lists into `close' and `far' lists. The field due
to a `close' list cannot be mediated through intermediate expansions for
accuracy reasons, and so it is evaluated directly (i.e.,\ using point-to-QBX-local expansion) at the QBX expansion
centers. In contrast, the field due to a `far' list is sufficiently far from the
target confinement region that mediation via an intermediate expansion is
permissible from the standpoint of accuracy.

For case of List 3, the `close' list consists of source boxes in the near field of
the target box which are not adequately separated from the TCR of the target
box. Interactions from these source boxes must be accumulated directly. The `far' list
is the smallest possible `complement' of this list in the sense that the close
and far lists must cover the entire near field mediated by List~3, and the `far' lists contains
boxes that are as large as permissible given the target confinement
restrictions. The `far' list may be mediated via multipole-to-target
interaction.

\begin{definition}[List 3 close, $\ilist{3close}{b}$]\label{def:list-3-close}
  For a target box $b$, a leaf box $d$ is said to be in $\ilist{3close}{b}$ if $d
  \in \descendants(\ilist{3}{b}) \cup \ilist{3}{b}$ such that $d \not \adequatesep
  \tcr(b)$.
\end{definition}

\begin{definition}[List 3 far, $\ilist{3far}{b}$]\label{def:list-3-far}
  For a target box $b$, a box $d$ is in $\ilist{3far}{b}$ if $d \in
  \descendants(\ilist{3}{b}) \cup \ilist{3}{b}$ such that $d \adequatesep \tcr(b)$
  and, for all $w \in \ancestors(d) \cap (\descendants(\ilist{3}{b}) \cup
  \ilist{3}{b})$, $w \not \adequatesep \tcr(b)$.
\end{definition}

While $\ilist{3close}{b}\cup \ilist{3far}{b}\subseteq \ilist{3}{b}$ does not
hold in general, it is generally the case that $\ilist{3close}{b}\cup
\ilist{3far}{b} \subseteq \descendants(\ilist{3}{b}) \cup \ilist{3}{b}$.

For the case of List 4, the `close' list consists of source boxes from which the
TCR of the target box is not adequately separated. List 4 close is evaluated
directly \emph{only at the targets in the box}, and is not
downward-propagating. List 4 far is downward-propagating. It consists of boxes
form which the TCR of the target box is adequately separated. By monotonicity of
`$\adequatesep$', this means that the TCR of the descendants is also adequately
separated from the sources, ensuring accurate downward propagation. To ensure
that the field of List~4~close is propagated to the descendants, a box in List~4~close
of the parent is placed in List~4~far of a descendant as soon as the TCR of the descendant
is adequately separated from it.

\begin{definition}[List 4 close, $\ilist{4close}{b}$]\label{def:list-4-close}
  Let $b$ be a target or target-ancestor box. A box $d$ is in
  $\ilist{4close}{b}$ if for some $w \in \ancestors(b) \cup \{b\}$ we have $d
  \in \ilist{4}{w}$ and furthermore $\tcr(b) \not \adequatesep d$.
\end{definition}

\begin{definition}[List 4 far, $\ilist{4far}{b}$]\label{def:list-4-far}
  Let $b$ be a target or target-ancestor box. A box $d \in \ilist{4}{b}$ is in
  List 4 far if $\tcr(b) \adequatesep d$.  Furthermore, if $b$ has a parent, a
  box $d \in \ilist{4close}{\parent(b)}$ is in List 4 far if $\tcr(b)
  \adequatesep d$.
\end{definition}

As with List 3, it is generally not the case that $\ilist{4close}{b} \cup
\ilist{4far}{b} \subseteq \ilist{4}{b}$. However, $\ilist{4close}{b} \cup
\ilist{4far}{b}$ only contains boxes from a List 4 of $b$ or an ancestor of $b$.

\begin{remark}[Performance optimization for List 3 far]%
  \label{rem:list3far-to-list3close}%
  One can always remove a box from $\ilist{3far}{b}$ and place its leaf
  descendants in $\ilist{3close}{b}$, without adverse impact on
  accuracy since a multipole-to-QBX-local interaction is replaced with a direct one.
  When the number of sources in the descendants is small,
  doing this decreases the computational cost associated with evaluation
  of the potential due to the sources. We make use of this possibility in the
  complexity analysis.
\end{remark}
% }}}
% -----------------------------------------------------------------------------
\subsection{Formal Statement of Algorithm}%
\label{sec:gigaqbx-statement}
% -----------------------------------------------------------------------------

% {{{ algorithm

\subsubsection{Notation}

We make use of the following notation, which is a slightly modified version of
the notation used in~\cite{gigaqbx2d}. The following notation refers to `point'
potentials evaluated at a target not requiring QBX owned by a box $b$:
\begin{inparaenum}[(a)]
\item $\Potnear_b(t)$ denotes the potential at a target point $t$ due to all
  sources in $\ilist{1}{b} \cup \ilist{3close}{b} \cup \ilist{4close}{b}$; and
\item $\PotW{b}(t)$ denotes the potential at a target $t$ due to
  all sources in $\ilist{3far}{b}$.
\end{inparaenum}
Let $c$ be a QBX center owned by box $b$. The following notation refers to
potentials evaluated with QBX mediation:
\begin{inparaenum}[(a)]
\item $\Lqbxnear{c}(t)$ denotes the (QBX) local expansion of the potential at
  the center $c$, evaluated at target $t$, due to all sources in $\ilist{1}{b}
  \cup \ilist{3close}{b} \cup \ilist{4close}{b}$;
\item $\LqbxW{c}(t)$ denotes the (QBX) local expansion at the center $c$,
  evaluated at $t$, due to all sources in $\ilist{3far}{b}$; and
\item $\Lqbxfar{c}(t)$ denotes the (QBX) local expansion at the center $c$,
  evaluated at $t$, due to all sources not in $\ilist{1}{b} \cup \ilist{3}{b}
  \cup \ilist{4close}{b}$.
\end{inparaenum}
Lastly, given a box $b$, $\Mpole_b$ and $\Locfar_b$ refer respectively to the
multipole and local expansions associated with the box.

\subsubsection{Algorithmic Parameters}

The parameters to the algorithm are $\pfmm$, the FMM order; $\pqbx$, the QBX
order; $\pquad$, the upsampled quadrature node count; and the target confinement
factor $t_f$. The choice of these parameters is based on the splitting of the
overall error in the scheme into truncation error
(Lemma~\ref{lem:qbx-truncation-3d}), quadrature error
(Lemma~\ref{lem:surface-quad-estimate}), and acceleration error
(Theorem~\ref{thm:accuracy}).

Of these sources of error, perhaps the one that is most straightforward to control with algorithmic parameters is
the acceleration error. Given a tolerance $\epsilon > 0$, choosing $t_f \leq
0.85$ allows setting $\pfmm \approx \lvert \log_{4/3} \epsilon \rvert$ to
guarantee that the relative error for this component is on the order of
$\epsilon$. (See Theorem~\ref{thm:accuracy} below.)

In contrast, the topic of parameter selection for QBX and quadrature order
remains an area of active research. On the subject of quadrature order, we have
found the advice in~\cite{gigaqbx2d}, which suggests the choice of a generically
high order to ensure the smallness of the quadrature error term, to be a useful
guide. Assuming the quadrature error is suitably controlled, we have observed
that $\pfmm$ is a practical upper bound on $\pqbx$, as the error due to
acceleration empirically appears to decrease most slowly of the various sources of error
(see Section~\ref{sec:results}).

The complete statement of the algorithm is given in Algorithm~\ref{alg:gigaqbx}.

\begin{algbreakable}{\algbrand~FMM in Three Dimensions}%
  \label{alg:gigaqbx}%
  \begin{algorithmic}
    \REQUIRE{The maximum number of FMM targets/sources $\nmax$ per
      box for octree refinement and a target confinement factor $t_f$ are
      chosen.}
    \REQUIRE{The input geometry and targets are preprocessed according
      to Section~\ref{sec:accuracy-control}.}
    \REQUIRE{Based on the precision $\epsilon$ to be achieved, a QBX
      order $\pqbx$, an FMM order $\pfmm$, and an oversampled quadrature node
      count $\pquad$ are chosen (see Sections~\ref{sec:background}~and~\ref{sec:accuracy-control}).}
    \ENSURE{An accurate approximation to the potential at all target points
      is computed.}

    \algstage{Stage 1: Build tree}
    {%
    \STATE{Create a octree on the computational domain containing all sources,
      targets, and QBX centers.}
    \REPEAT{}
    \STATE{Subdivide each box containing more than $n_\text{max}$ particles into
      eight children, pruning any empty child boxes. If an expansion center
      cannot be placed in a child box with target confinement factor $t_f$ due to its radius,
      it remains in the parent box.}
    \UNTIL{each box can no longer be subdivided or an iteration produced only
    empty child boxes}
    }

    \algstage{Stage 2: Form multipoles}{%
    \FORALL{boxes $b$}
    \STATE{Form a $\pfmm$-th order multipole expansion $\Mpole_b$ centered at $b$ due to
      sources owned by $b$.}
    \ENDFOR{}
    \FORALL{boxes $b$ in postorder}
    \STATE{For each child of $b$, shift the center of the multipole expansion at
      the child to $b$. Add the resulting expansions to $\Mpole_b$.}
    \ENDFOR{}
    }

    \algstage{Stage 3: Evaluate direct interactions}{%
    \FORALL{boxes $b$}
    \STATE{For each conventional target $t$ owned by $b$, add to $\Potnear_b(t)$
      the contribution due to the interactions from sources owned by boxes in
      $\ilist{1}{b}$ to $t$.}
    \ENDFOR{}
    \FORALL{boxes $b$}
    \STATE{For each QBX center $c$ owned by $b$, add to the expansion
      $\Lqbxnear{c}$ the contribution due to the interactions
      from $\ilist{1}{b}$ to $c$.}
    \ENDFOR{}
    }

    \algstage{Stage 4: Translate multipoles to local expansions}{%
    \FORALL{boxes $b$}
    \STATE{For each box $d \in \ilist{2}{b}$, translate the multipole expansion
      $\Mpole_{d}$ to a local expansion centered at $b$. Add the resulting
      expansions to obtain $\Locfar_b$.}
    \ENDFOR{}
    }

    \algstage{Stage 5(a): Evaluate direct interactions due to $\ilist{3close}{b}$}{%
    \STATE{Repeat Stage 3 with $\ilist{3close}{b}$ instead of $\ilist{1}{b}$.}
    }

    \algstage{Stage 5(b): Evaluate multipoles due to $\ilist{3far}{b}$}{%
    \FORALL{boxes $b$}
    \STATE{For each conventional target $t$ owned by $b$, evaluate the multipole
      expansion $\Mpole_{d}$ of each box $d \in \ilist{3far}{b}$
      to obtain $\PotW{b}(t)$.}
    \ENDFOR{}
    \FORALL{boxes $b$}
    \STATE{For each QBX center $c$ owned by $b$, add to the expansion
      $\LqbxW{c}$ the contributions due to
      the multipole expansion $\Mpole_{d}$ of each box $d \in \ilist{3far}{b}$.}
    \ENDFOR{}
    }

    \algstage{Stage 6(a): Evaluate direct interactions due to $\ilist{4close}{b}$}{%
    \STATE{Repeat Stage 3 with $\ilist{4close}{b}$ instead of $\ilist{1}{b}$.}
    }

    \algstage{Stage 6(b): Form locals due to $\ilist{4far}{b}$}{%
    \FORALL{boxes $b$}
    \STATE{Convert the field of every particle owned by boxes in $\ilist{4far}{b}$
      to a local expansion about $b$. Add to $\Locfar_b$.}
    \ENDFOR{}
    }

    \algstage{Stage 7: Propagate local expansions downward}{%
    \FORALL{boxes $b$ in preorder}
    \STATE{For each child $d$ of $b$, shift the center of the local expansions
      $\Locfar_b$ to the child. Add the resulting expansions to
      $\Locfar_d$ respectively.}
    \ENDFOR{}
    }

    \algstage{Stage 8: Form local expansions at QBX centers}{%
    \FORALL{boxes $b$}
    \STATE{For each QBX center $c$ owned by $b$, translate $\Locfar_b$ to
      $c$, obtaining $\Lqbxfar{c}$.
      }
    \ENDFOR{}
    }

    \algstage{Stage 9: Evaluate final potential at targets}{%
    \FORALL{boxes $b$}
    \STATE{For each conventional target $t$ owned by $b$, evaluate $\Locfar_b(t)$.}
    \STATE{Add $\Potnear_b(t), \PotW{b}(t), \Locfar_b(t)$ to obtain the
      potential at $t$.}
    \ENDFOR{}
    \FORALL{boxes $b$}
    \STATE{For each target $t$ associated to a QBX center $c$ owned by $b$,
      add $\Lqbxnear{c}(t), \LqbxW{c}(t), \Lqbxfar{c}(t)$ to obtain the QBX local
      expansion due to $c$ evaluated at $t$.}
    \ENDFOR{}
    }
  \end{algorithmic}
\end{algbreakable}
% }}}

% -----------------------------------------------------------------------------
\subsection{Accuracy}%
\label{sec:accuracy}
% -----------------------------------------------------------------------------
% {{{

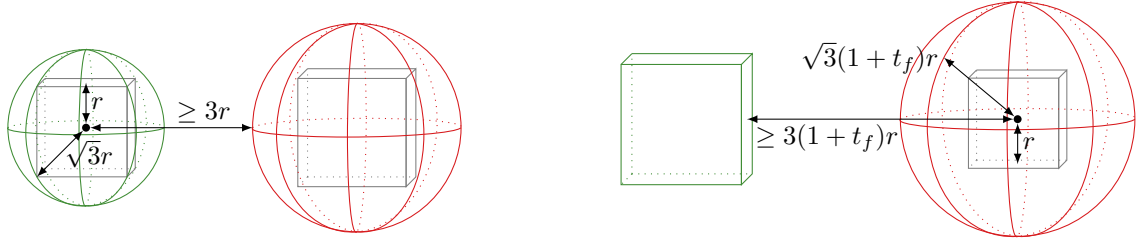
\begin{figure}[h]
  % {{{ M2L figure
  \centering
  \begin{tikzpicture}[scale=0.9, z={(-0.09,-0.09)}]
    \def\tcf{0.5}

    \DrawCube[cubecolor]{0}{0}{0}{1}
    \DrawCube[cubecolor]{-6}{0}{0}{1}

    % intervening cubes
    \begin{scope}[canvas is xz plane at y=-1]
      \draw[cubecolor, dotted] (-1,1) -- +(-4,0);
      \draw[cubecolor, dotted] (-1,-1) -- +(-4,0);
    \end{scope}

    \begin{scope}[canvas is xz plane at y=1]
      \draw[cubecolor, dotted] (-1,1) -- +(-4,0);
      \draw[cubecolor, dotted] (-1,-1) -- +(-4,0);
    \end{scope}

    \begin{scope}[canvas is yz plane at x=-3]
      \draw[cubecolor, dotted] (-1,-1) rectangle (1,1);
    \end{scope}

    \DrawSphere[localcolor]{0}{0}{0}{sqrt 3*(1+\tcf)}
    \DrawSphere[srccolor]{-6}{0}{0}{sqrt 3}

    \draw[<->, length]
      (0,0) --
      (xyz cs: x=-1-\tcf, y=1+\tcf, z=1+\tcf)
      node [right, xshift=0.2cm] {$\sqrt{3} (1 + t_f) r$};

    \draw[<->, length]
      (0,0) --
      (xyz cs: y=-1)
      node [right, midway] {$r$};

    \coordinate[mark coordinate] (t) at (0, 0);
    \coordinate[mark coordinate] (s) at (xyz cs: x=-6);

    \draw[<->, length] (t) -- (s) node[below, midway, xshift=-0.9cm] {$6r$};

    \draw[<->, length] (s) -- +(xyz cs: x=1, y=1, z=1)
      node [left, midway] {$\sqrt{3}r$};
  \end{tikzpicture}
  \[
    \displaystyle
    \frac{\text{\color{localcolor}furthest target}}
    {\text{\color{srccolor}closest source}}
    \leq
    \frac{\sqrt{3}(1+t_f)}{6 - \sqrt{3}}
  \]
  \caption{Convergence factor calculation for List 2. On the left is the source
    box and on the right are the target box and target confinement region.}%
  % }}}
  \label{fig:list2}
\end{figure}

\begin{figure}[ht]
  \begin{minipage}[b]{.48\linewidth}
  \centering
  \begin{tikzpicture}[scale=0.8,z={(-0.09,-0.09)}]
    \def\srcr{0.75};

    \coordinate[mark coordinate] (box1ctr) at (0,0,0);
    \DrawCube[cubecolor]{0}{0}{0}{\srcr};
    \DrawSphere[srccolor]{0}{0}{0}{\srcr*sqrt 3};

    \DrawCube[cubecolor]{4.5}{0}{0}{0.9};
    \DrawSphere[localcolor]{4.5}{0}{0}{sqrt 3};

    \draw[<->, length]
      (box1ctr) -- +(4.5-sqrt 3,0) node[anchor=south, pos=0.7] {$\geq 3 r$};
    \draw[<->, length] (box1ctr) -- ++(0, \srcr) node[pos=0.5, anchor=west] {$r$};
    \draw[<->, length]
      (box1ctr) -- (-\srcr,-\srcr,\srcr) node[pos=0.5, anchor=west] {$\sqrt{3} r$};
  \end{tikzpicture}
  \[
  \displaystyle
  \frac{\text{\color{srccolor}furthest source}}
  {\text{\color{localcolor}closest target}}
  \leq
  \frac{\sqrt{3}}{3}
  \]
  \caption{%
    Convergence factor calculation for List 3 far.
    The source box is on the left and the target box and target confinement
    region are on the right.
  }%
  \label{fig:list3}
  \end{minipage}
  \hfill
  \begin{minipage}[b]{.48\linewidth}
  \centering
  \begin{tikzpicture}[scale=0.8,z={(-0.09,-0.09)}]
    \def\tgtr{0.75};

    \coordinate (box1ctr) at (-1.5,0,0);
    \coordinate[mark coordinate] (box2ctr) at (4,0,0);

    \DrawCube[srccolor]{-1.5}{0}{0}{1};

    \DrawCube[cubecolor]{4}{0}{0}{\tgtr};
    \DrawSphere[localcolor]{4}{0}{0}{\tgtr * 1.5 * sqrt 3};

    \draw[<->, length] (box2ctr) -- ++(0,-\tgtr,0) node[right, midway] {$r$};

    \draw[<->, length] (box2ctr) -- ++(-\tgtr*1.5, +\tgtr*1.5, \tgtr*1.5)
      node[left] {$\sqrt{3}(1 + t_f)r$};

    \draw[<->, length] (box2ctr) -- (-0.5,0) node[anchor=north, pos=0.7] {$\geq 3(1 + t_f)r$};
  \end{tikzpicture}
  \[
  \displaystyle
  \frac{\text{\color{localcolor}furthest target}}
  {\text{\color{srccolor}closest source}}
  \leq
  \frac{\sqrt{3}}{3}
  \]
  \caption{%
    Convergence factor calculation for List 4 far.
    The source box is on the left and the target box and target confinement
    region are on the right.
  }%
  \label{fig:list4}
  \end{minipage}
\end{figure}

We carry out accuracy estimates of algorithm assuming the truth of the hypotheses
in Section~\ref{sec:error-estimates}. These hypotheses imply that the accuracy
of the FMM in approximating the unaccelerated version of the potential is mainly
determined by the choice of $t_f$, with a smaller value of $t_f$ leading to more
accurate results. Choosing a value of $t_f \lessapprox 0.85$ recovers a
convergence factor approximately the same as the 1-away point FMM, which is
$3/4$~\cite{petersen_error_1995}.

\begin{theorem}[Accuracy estimate for \algbrand~algorithm]%
  \label{thm:accuracy}
  Fix a target confinement factor $0 \leq t_f < 2 \sqrt{3} - 2 \approx
  1.47$. Let $R$ denote the radius of the smallest box in the tree and let
  $\omega = \min \left(3 - \sqrt{3}, 6 - 2 \sqrt{3} - \sqrt{3} t_f \right)$.
  Assuming the truth of
  Hypotheses~\ref{hyp:m2p},~\ref{hyp:l2p},~and~\ref{hyp:m2l}, there exists a
  constant $M > 0$ such that for every target point $x \in \mathbb{R}^3$, we
  have
  \[
    \left|
    \mathcal{S}_{\mathrm{QBX}(\pqbx,N)} \mu(x)
    - \mathcal{G}_\pfmm [\mathcal{S}_{\mathrm{QBX}(\pqbx,N)}] \mu(x)
    \right|
    \leq
    \frac{M A}{\omega R}
    \max \left(\frac{\sqrt{3}(1+t_f)}{6 - \sqrt{3}},
      \frac{\sqrt{3}}{3}\right)^{\pfmm+1}.
  \]
  Here $\mathcal{S}_{\mathrm{QBX}(\pqbx,N)} \mu(x)$ denotes the $\pqbx$-th order
  approximation to the single-layer potential given by QBX,
  $\mathcal{G}_\pfmm[\cdot]$ denotes the approximation formed by the
  \algbrand~FMM of order $\pfmm$, and
  \(
     A = \| \mu \|_\infty \sum_{i=1}^N |w_i|,
  \)
  where the $\{w_i\}_{i=1}^N$ are the quadrature weights. The constant $M$ is
  independent of $t_f$, $\omega$, $\mu$, the particle distribution, and the QBX
  and FMM orders.

  In particular, for $t_f \leq (6\sqrt{3} - 7)/4 \approx 0.85$, we have
  \[
    \left|
    \mathcal{S}_{\mathrm{QBX}(\pqbx,N)} \mu(x)
    - \mathcal{G}_\pfmm [\mathcal{S}_{\mathrm{QBX}(\pqbx,N)}] \mu(x)
    \right|
    \leq \frac{M A}{\omega R} \left( \frac{3}{4} \right)^{\pfmm + 1}.
  \]
\end{theorem}

\begin{proof}
  Without loss of generality, we will assume that $x$ is associated to a QBX
  center $c$. The proof of this statement follows from applying the results of
  Section~\ref{sec:error-estimates} to the definitions of the interaction lists in
  Section~\ref{sec:algorithm}. The potential at $c$ is the sum of the contributions
  $\Lqbxnear{c}(t)$, $\LqbxW{c}(t)$, and $\Lqbxfar{c}(t)$.

  The potential due to $\Lqbxnear{c}(t)$ must arrive via a direct interaction.
  This contribution incurs no (acceleration) error.

  The potential due to $\LqbxW{c}(t)$, which arrives via a multipole-to-QBX-local
  interaction, incurs an error of at most
  \begin{equation}%
    \label{eqn:list3-error}
    \frac{CA}{(3 - \sqrt{3}) R} \left( \frac{\sqrt{3}}{3} \right)^{\pfmm + 1}
  \end{equation}
  where $C > 0$ is some constant. See Hypothesis~\ref{hyp:m2p} and Figure~\ref{fig:list3}.

  The potential due to $\Lqbxfar{c}(t)$ arrives via an interaction of $\ilist{4far}{b'}$
  or $\ilist{2}{b'}$, where $b'$ is the box that owns $c$ or an ancestor box. The
  contribution that comes from $\ilist{4far}{b'}$ arrives via a local-to-QBX-center
  interaction, and incurs an error of at most
  \begin{equation}%
    \label{eqn:list4-error}
    \frac{CA}{(3 - \sqrt{3}) R} \left( \frac{\sqrt{3}}{3} \right)^{\pfmm + 1}
  \end{equation}
  where $C > 0$ is some constant. See Hypothesis~\ref{hyp:l2p} and Figure~\ref{fig:list4}.

  Lastly, the contribution due to all $\ilist{2}{b'}$ interactions
  arrives via a multipole-to-local-to-QBX-center interaction and incurs an error
  of at most
  \begin{multline}%
    \label{eqn:list2-error}
    \frac{CA}{(6 - 2 \sqrt{3} - \sqrt{3} t_f) R}
    \left(
    \left( \frac{\sqrt{3}(1 + t_f)}{6 - \sqrt{3}} \right)^{\pfmm + 1}
    +
    \left(\frac{\sqrt{3}}{6 - \sqrt{3}(1 + t_f)} \right)^{\pfmm + 1}
    \right) \\
    \leq
    \frac{2CA}{(6 - 2 \sqrt{3} - \sqrt{3} t_f) R}
    \left( \frac{\sqrt{3}(1 + t_f)}{6 - \sqrt{3}} \right)^{\pfmm + 1}
  \end{multline}
  where $C > 0$ is some constant.
  To obtain this bound, we have used the stated assumption that $t_f < 2 \sqrt{3} - 2$.
  See Hypothesis~\ref{hyp:m2l} and
  Figure~\ref{fig:list2} for the convergence factor calculation.

  The result follows from
  combining~\eqref{eqn:list2-error},~\eqref{eqn:list3-error},~and~\eqref{eqn:list4-error}.
\end{proof}
% }}}

% -----------------------------------------------------------------------------
\subsection{Complexity}%
\label{sec:complexity}
% -----------------------------------------------------------------------------
% {{{

\def\nstages{9}

\def\EtoEcost{\pfmm^3}
\def\PtoEcost{\pfmm^2}
\def\PtoPcost{\pqbx^2}

\begin{table}
  \centering
  \caption{Complexity of each stage of the \algbrand\ algorithm.}%
  \label{tab:complexity-analysis}
  \begin{tabular}{ccp{0.5\textwidth}}
    \toprule
    Stage & Modeled Operation Count & Note\\

    \midrule

    Stage 1 & $NL$ & There are $N$ total particles with at most $L$ levels of
    refinement. \\

    Stage 2 & $N_S \PtoEcost + N_B \EtoEcost$ & $N_S \PtoEcost$ for
      forming multipoles and the rest for shifting multipoles upward, with each
      shift costing $\EtoEcost$. \\

    Stage 3 & $(27 (N_C + N_S) \nmax + N_C M_C) \PtoPcost$ &
      Lemma~\ref{lem:list1-complexity} \\

    Stage 4 & $875 N_B \EtoEcost$ & Lemma~\ref{lem:list2-complexity} \\

    Stage 5 & $N_C M_C \pqbx^2 + 124L N_S \nmax \PtoPcost$ &
      Lemma~\ref{lem:list3-complexity} \\

    Stage 6 & $375 N_B \nmax \PtoEcost + 250 N_C \nmax \PtoPcost$ &
      Lemma~\ref{lem:list4-complexity} \\

    Stage 7 & $8 N_B \EtoEcost$ & The cost of shifting a local expansion
      downward is $\EtoEcost$. There are at most $8$ children per box. \\

    Stage 8 & $N_C \EtoEcost$ & Cost of translating the box local
      expansions to $N_C$ centers. \\

    Stage 9 & $N_T \PtoPcost$ & Cost of evaluating QBX expansions at
      $N_T$ targets. \\
    \bottomrule
  \end{tabular}
\end{table}

\begin{table}
  \centering
  \caption{Parameters to the complexity analysis.}%
  \label{tab:complexity-params}
  \begin{tabular}{lp{3in}}
    \toprule
    Parameter & Note \\
    \midrule
    $\pfmm$ & FMM order. \\
    $\pqbx$ & QBX order. \\
    $\nmax$ & Desired bound on the number of particles per box (in some
    cases there may be more, see Section~\ref{sec:particle-distribution}). \\
    $t_f$ & Target confinement factor. \\
    $N_B$ & Number of boxes in the tree. \\
    $L$ & Number of levels in the tree. \\
    $N_S$ & Number of sources in the tree. \\
    $N_C$ & Number of QBX centers in the tree. \\
    $N_T$ & Number of targets in the tree. \\
    $N$ & $N \coloneqq N_S + N_C + N_T$; the number of `particles'. \\
    $M_C$ & The average number of source particles inside $\closedbox(4
    \sqrt{3} r_c t_f, c)$, taken over all QBX centers $c$ (see also Section~\ref{sec:particle-distribution}). \\
    \bottomrule
  \end{tabular}
\end{table}

In this section we discuss the (time) complexity of the \algbrand~algorithm in
three dimensions. The costs presented in this section report the asymptotic number of
`modeled floating point operations' (or `modeled flops') performed by the
algorithm. The parameters we introduce in the complexity analysis are
summarized in Table~\ref{tab:complexity-params}.
Table~\ref{tab:complexity-analysis} and Theorem~\ref{thm:complexity} in this
section provide a summary of the complexity analysis.  For extended details, see
Appendix~\ref{sec:detailed-complexity-analysis}.

% -----------------------------------------------------------------------------
\subsubsection{Assumptions}%
\label{sec:complexity-assumptions}
% -----------------------------------------------------------------------------

We make a number of simplifying assumptions in our complexity analysis.

We assume that all targets have been assigned to a QBX center, so that all
evaluation at targets is done in Stage~9. This is the primary usage pattern for
on-surface evaluation of a layer potential.

We assume, consistent with the hypotheses of Theorem~\ref{thm:accuracy}, that
$t_f < 2$. An assumption of this nature is useful for the analysis of List 4
(Lemma~\ref{lem:list4-complexity}).

We assume the use of spherical harmonic expansions throughout the algorithm. The
cost of translation is modeled using `point-and-shoot' translation operators
(see Section~\ref{sec:translation-complexity} for details). We also assume that
$\pqbx \leq \pfmm$.

% -----------------------------------------------------------------------------
\subsubsection{Summary}
% -----------------------------------------------------------------------------
% {{{

Theorem~\ref{thm:complexity} summarizes the complexity of
Algorithm~\ref{alg:gigaqbx}. The cost of the tree build phase (Stage 1) and the
evaluation phase of the algorithm (Stages 2--\nstages) are treated
separately. Under broadly applicable assumptions, the evaluation phase can be
shown to run in time that is proportional to the number of
particles. Nevertheless, the proportionality constant is affected by the details
of the particle distribution in two ways. First, the average size of the
`near~neighborhoods' of QBX centers affects the number of direct interactions in
the algorithm. This is measured by the parameter $M_C$. Second, the number of
boxes in the tree, $N_B$, affects the number of intermediate expansions that are
formed by the algorithm. This parameter is also determined by the details of the
particle distribution.

In the final statement of the complexity analysis, we make the simplifying
assumption that $M_C = O(1)$ and $N_B = O(N)$.

\begin{theorem}[Complexity estimate for \algbrand~algorithm]%
  \label{thm:complexity}%
  \begin{enumerate}[(a)]
    \item The cost in modeled flops of the tree build phase of~the~\algbrand~FMM is $O(NL)$.
    \item Assume that $\pfmm = O(\lvert \log \epsilon \rvert)$. For a fixed
      value of $\nmax$, the cost in modeled flops of the evaluation stage of the
      \algbrand~FMM is
      \(
        O((N_C + N_S + N_B) \lvert \log \epsilon \rvert^3
        + (N_C M_C + NL) \lvert \log \epsilon \rvert^2
        + N_T \lvert \log \epsilon \rvert^2).
      \)
      With a level-restricted octree and $t_f < \sqrt{3} - 1$, the modeled cost is
      \(
        O((N_C + N_S + N_B) \lvert \log \epsilon \rvert^3
        + N_C M_C \lvert \log \epsilon \rvert^2
        + N_T \lvert \log \epsilon \rvert^2).
      \)
      Assuming that the particle distribution satisfies $N_B = O(N)$ and $M_C =
      O(1)$, the worst-case
      modeled cost using a level-restricted octree and~$t_f < \sqrt{3} - 1$ is
      linear~in~$N$ (with a constant dependent on the particle distribution and
      the desired accuracy $\epsilon$).
  \end{enumerate}
\end{theorem}

\begin{proof}
  The estimate for (a)~follows from the cost of Stage 1 as listed in
  Table~\ref{tab:complexity-analysis}.  This estimate for~(b) follows from
  adding up the costs of Stages~2--\nstages\ as found in
  Table~\ref{tab:complexity-analysis}. The linear running time in the case of a
  level-restricted octree and $t_f < \sqrt{3} - 1$ follows from
  Remark~\ref{rem:level-restriction}.
\end{proof}

% }}}

% }}}

% #############################################################################
\section{Numerical Experiments}%
\label{sec:results}
% -----------------------------------------------------------------------------
% {{{ parameters

\def\nmaxgigaqbx{\num{512}}
\def\urchintestarms{8}
\def\urchinquadorder{8}
\def\urchinovsmp{5}
\def\urchintotalquadorder{32}
\def\urchingreentestnodesperelement{\num{295}}
\def\urchingreentestnelementsone{\num{48500}}
\def\urchingreentestnelementstwo{\num{277712}}
\def\expnmax{\num{512}}

We use a family of smooth `urchin' test geometries $\gamma_k$
given analytically in spherical coordinates
% theta and phi are swapped compared to meshmode
$(r_k, \theta,\phi)$ by prescribing $r_k$ as a function of $(\theta,\phi)$, where
\begin{align}
  \label{eq:urchin-warping}
  r_k(\theta,\phi) &= 0.2 + \frac{\Re Y_{k}^{\lfloor k/2\rfloor}(\theta,\phi) - m_k}{M_k-m_k},\\
  \substack{M_k\\m_k} &= {\substack{\max\\\min}}_{\theta\in[0,\pi],\phi\in[0,2\pi]}
    \Re Y_{k}^{\lfloor k/2\rfloor}(\theta,\phi),
  \notag
\end{align}
using the definition of spherical harmonics from~\eqref{eq:spherical-harmonics}.
Figure~\ref{fig:urchin-8} gives a visual impression of $\gamma_8$.

To obtain an accurate unstructured triangular mesh of $\gamma_k$, we use an icosahedron
as a starting point. Each of the icosahedron's faces is equipped with a mapping
$\mapping_\iel\in (P^\urchinquadorder)^3$ and the expansion of $\mapping_\iel$
in orthogonal polynomials on the triangle~\cite{dubiner_spectral_1991,koornwinder_two-variable_1975}
is computed. While the $\ell^2$-norm of the coefficients of the mapping
corresponding to the polynomials of the two highest total degrees exceeds
$10^{-10}$ times the $\ell^2$-norm of all coefficients of the mapping,
the element is bisected, and the warping function~\eqref{eq:urchin-warping} is
(nodally) reevaluated.

% }}}
% -----------------------------------------------------------------------------
\subsection{Accuracy}
% -----------------------------------------------------------------------------
% {{{

{%
  % https://tex.stackexchange.com/questions/334269/using-dcolumn-siunitx-for-aligning-numbers-with-uncertainty-and-boldface/334323#334323
  \renewrobustcmd{\bfseries}{\fontseries{b}\selectfont}
  \newcommand{\converged}[1]{\bfseries #1}
  \sisetup{%
    table-format = 1.2e-1,
    table-number-alignment = center,
    table-sign-exponent = true,
    scientific-notation = true,
    round-mode = places,
    round-precision = 2,
    detect-weight = true,
    mode = text,
  }
  \begin{table}
    \centering
    \caption{$\ell^\infty$ error in Green's formula $\mathcal S(\partial_n
      u)-\mathcal D(u)=u/2$, scaled by $1/\|u\|_\infty$, for the
      `urchin' $\gamma_{\urchintestarms}$, using the \algbrand\ algorithm.  $\pfmm$ denotes the FMM
      order and $\pqbx$ the QBX order.  The geometry was discretized with
      $\urchingreentestnelementsone$ triangles for the stage-1 discretization,
      and
      $\urchingreentestnelementstwo$ triangles for the stage-2 discretization, with
      $\urchingreentestnodesperelement$ nodes per
      element. An idealized a-priori estimate
      for the 1-away point FMM error~\cite{petersen_error_1995}
      is included in the first column for comparison.  Entries in bold indicate
      that the FMM error is negligible compared to the other error contributions.}%
    \label{tab:urchin-accuracy}
    \begin{tabular}{ScSSSS}
      \toprule
      {$(3/4)^{\pfmm+1}$} & {$\pfmm$} & {$\pqbx=3$} & {$\pqbx=5$} & {$\pqbx=7$} & {$\pqbx=9$}\\
      \midrule
      0.31640625           &     3 & 0.008291143698645285 & 0.00968404877505063 & 0.009151145525972935 & 0.00917659527996213 \\
      0.177978515625       &    5 & 0.0014306317807280255 & 0.0026739663967821794 & 0.0028549448708101214 & 0.0027803566630102924 \\
      0.04223513603210449  &     10 & \converged{6.076861932032935e-05} & 6.438110189318733e-05 & 0.0001270450424234368 & 0.0001466052625046138 \\
      0.010022595757618546 &   15 & \converged{6.0760217036864854e-05} & \converged{6.3807673565196985e-06} & 3.237558247597872e-06 & 7.07096870849642e-06 \\
      0.002378408954200495 &  20 & \converged{6.076021987995081e-05} & \converged{6.3807144934427785e-06} & 1.413902385243144e-06 & 2.509328730510307e-07 \\
      \bottomrule
    \end{tabular}
  \end{table}
}

We use an analogous procedure to the one from~\cite{gigaqbx2d} to test the
accuracy of the algorithm of this paper through a sequence of
experiments. With $u$ a
harmonic function defined inside $\gamma_8$ and extending smoothly to the boundary,
we make use of \emph{Green's formula}. Because of smoothness, $u$ has
a well-defined normal derivative $\partial_n u$ at the boundary. Then Green's formula
(e.g.~\cite[Theorem 6.5]{kress:2014:integral-equations}) states that for $x\in\Gamma$,
\[ \mathcal{S} (\partial_n u)(x) - \mathcal{D}(u)(x) = \frac{u(x)}{2}. \]
We use the residual in this identity as a measure for the accuracy
that our scheme achieves in the evaluation of layer potential evaluations.
The achieved accuracy in Green's formula is predictive of the accuracy
one might achieve in the solution of boundary value problems. Data to support
this assertion (in two dimensions) is presented in~\cite{gigaqbx2d}.

Letting $u$ be the potential due to a charge located outside
$\Gamma$ at $(3,1,2)^T$, we evaluate $\mathcal{S} (\partial_n u) -
\mathcal{D}(u)$ using our scheme and report the error in the discrete
$\ell^\infty$-norm.
The error reported is the absolute error scaled by $1/\|u\|_\infty$. We use
the urchin geometry $\gamma_\urchintestarms$ and test with various
combinations of QBX order $\pqbx$ and FMM order $\pfmm$. $\gamma_{8}$ was
discretized with $\urchingreentestnelementsone$ triangles for the stage-1
discretization, and $\urchingreentestnelementstwo$ triangles for the stage-2
discretization, with $\urchingreentestnodesperelement$ nodes per element,
to eliminate the influence of quadrature error as a confounding factor
in this experiment.

To `balance' the algorithm, we compute a modeled flop count
(cf.\ Sections~\ref{sec:complexity} and~\ref{sec:exp-cost}) and chose the value
of $\nmax$ that minimizes this modeled cost. In our experiments,
$\nmax=\expnmax$ was the approximate minimizer.  We choose the target confinement factor as
$t_f=0.9$.

Table~\ref{tab:urchin-accuracy} shows the results of these experiments for the
\algbrand\ FMM, scaled by the norm of the test function $u$ and varying $\pqbx$
across columns and $\pfmm$ across rows.  We show table entries in bold if no
decrease in error is observed for at least one subsequent value of $\pfmm$.

The error in each entry of the table may be interpreted as the additive
contribution of truncation error (Lemma~\ref{lem:qbx-truncation-3d}), quadrature
error (Lemma~\ref{lem:surface-quad-estimate}), and FMM acceleration error
(Theorem~\ref{thm:accuracy}). The latter two sources of error are present even
without FMM acceleration. Thus, reading down a given column (with $\pqbx$ held
fixed and $\pfmm$ varying), we do not expect the error to decrease below a fixed
amount, which we term the `unaccelerated QBX error.' This quantity empirically
corresponds to the error shown in bold.

The large decrease in the error as $\pfmm$ increases suggests that as long as
the error still decreases with $\pfmm$, we can assume that the acceleration
error is the dominant error component. Reading across a row of the table (with
$\pfmm$ held fixed and $\pqbx$ varying), we observe that if an entry appears to
be dominated by acceleration error (i.e.\ is not in bold), the errors in the row
are very roughly of the same order of magnitude. This is consistent with
Theorem~\ref{thm:accuracy}, which implies an FMM acceleration error bound that
is independent of the QBX order.

For our chosen value of $t_f$, Theorem~\ref{thm:accuracy} roughly establishes
$\|u\|_\infty {(3/4)}^{\pfmm+1}$ as a bound on the absolute error incurred by
acceleration, neglecting a number of other factors given in the precise
statement of the theorem.  We show ${(3/4)}^{\pfmm+1}$ in the left column of the
table. Importantly, this value is an upper bound for the values in its row, and
thus also a bound on the acceleration error. The strict obedience to this bound
also confirms that the algorithm does not require an FMM order increase to
maintain accuracy (cf.\ \cite{rachh:2017:qbx-fmm} and the discussion in
Section~\ref{sec:approx-acceleration}).

The actual acceleration component of the error in
Table~\ref{tab:urchin-accuracy} in fact appears to decrease more rapidly than
the first column. A similar phenomenon was observed for the two-dimensional case
in~\cite{gigaqbx2d}. This is not entirely unexpected, as the error in the
potential is a weighted average of the individual errors due to the source
particles, which are likely to be separated more generously from the target than
the worst-case estimates in Section~\ref{sec:error-estimates} assume.

In summary, the results in this table empirically confirm the validity of
Theorem~\ref{thm:accuracy} as well as of an additive error model:
\begin{equation}
  |\text{total error}| \leq |\text{unaccelerated QBX error}| + \|u\|_\infty {(3/4)}^{\pfmm+1}.
  \label{eq:gigaqbx-sum-bound}
\end{equation}

% }}}
% -----------------------------------------------------------------------------
\subsection{A BVP with Complex Geometry for the Helmholtz Equation}%
\label{sec:helmholtz-bvp}
% -----------------------------------------------------------------------------
% {{{
\begin{figure}[ht]
  \centering
  \includegraphics[width=0.7\textwidth]{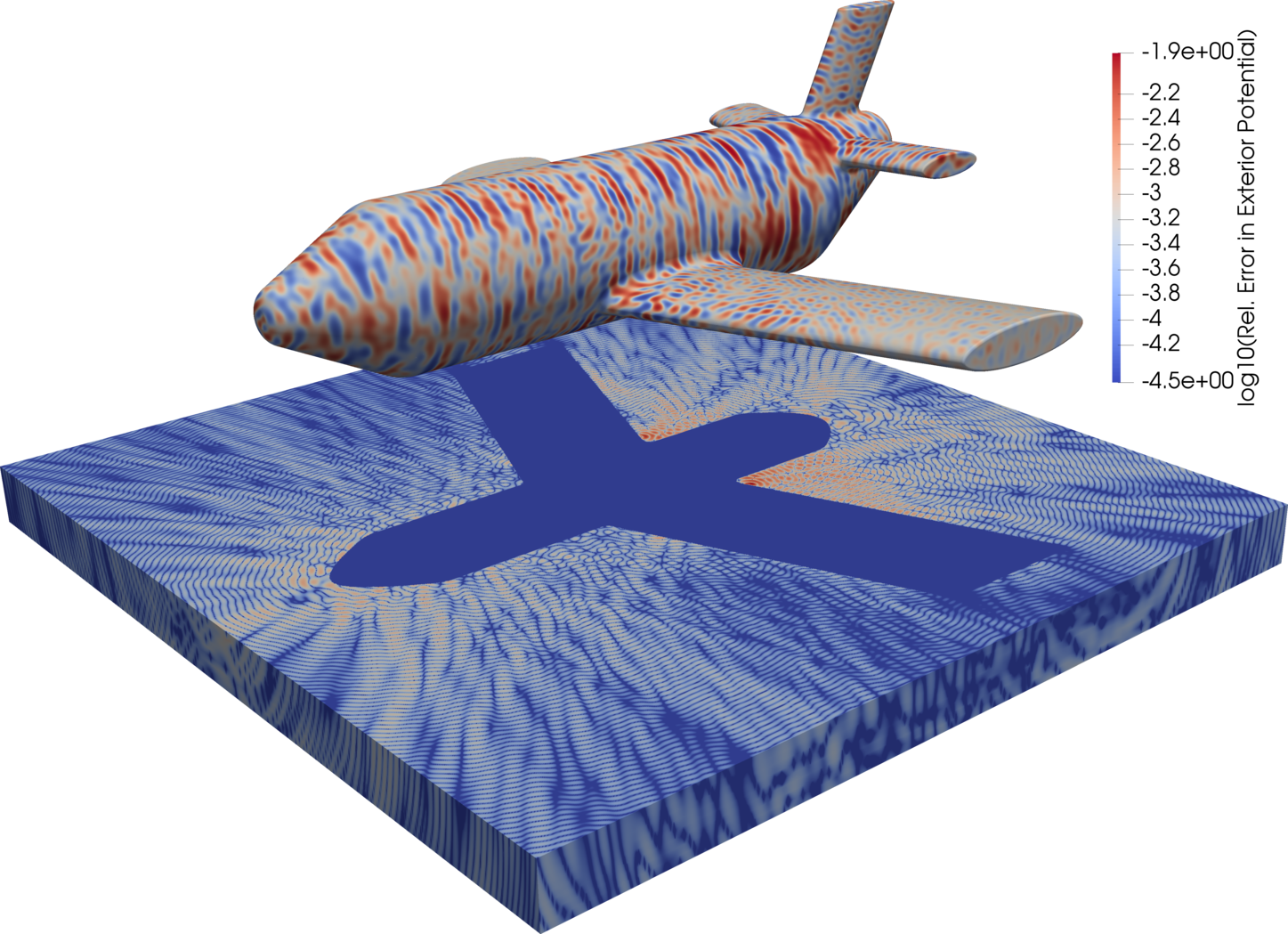}
  \caption{%
    An exterior Dirichlet boundary value problem for the Helmholtz equation
    solved on a `toy plane' geometry. The shading on the geometry itself
    reflects the obtained density $\mu$ using a Brakhage-Werner
    representation~\cite{brakhage_uber_1965} $-\mathcal D\mu+i\mathcal S \mu$.
    The volume visualization illustrates the logarithm of error in the computed
    exterior potential. The top of the volume visualization represents a cut
    roughly at `wing' level of the source geometry. The maximal relative
    $\ell^\infty$ error observed anywhere in the exterior computational domain
    (which extends to cover the entire geometry, including at points on or near
    the surface) was $1.38\times 10^{-2}$.  Section~\ref{sec:helmholtz-bvp}
    describes the computational setup in more detail.
    The potential is evaluated at \num{104947200} targets in the volume, with
    \num{11482688} source points, \num{1230288} QBX centers, and \num{615144}
    on-surface targets.
  }%
  \label{fig:betterplane}
\end{figure}

To support the assertion that our algorithm is broadly applicable and robust,
we demonstrate its use on a challenging, moderate-frequency boundary value
problem for the Helmholtz equation. While we have discussed a version of the
algorithm for the Laplace equation, a direct analog of our algorithm is
applicable for the Helmholtz and many other related elliptic PDEs, assuming the
availability of translation operators with suitable complexity. We expect
our complexity and accuracy analysis to carry over to the case of the Helmholtz
equation with only minor changes. Empirical evidence suggests that this is the
case.

We solve an exterior Dirichlet boundary value problem
\begin{DIFnomarkup}
\begin{alignat*}{2}
  \left( \triangle + k^2 \right) u &= 0 & \quad & \text{in } \mathbb R^3 \setminus \Omega, \\
  u                                &= f & \quad & \text{on } \partial \Omega , \\
  \lim_{r\to\infty} r \left(\frac{\partial}{\partial r} - ik \right) u
                                   &= 0
\end{alignat*}
\end{DIFnomarkup}
where $\Omega \subset \mathbb R^3$ is a closed, bounded region with
smooth boundary $\Gamma = \partial \Omega$. $\Omega$ is given by the geometry
\texttt{surface-3d/betterplane.brep} from~\cite{geometries_repo_2018}. We obtain
a surface mesh consisting of triangles with second-order polynomial mapping
functions for the geometry using Gmsh~\cite{geuzaine_gmsh_2009}.
The geometry (`nose' to `tail') is approximately 19 units long and 20 units wide
(`wingtip' to `wingtip').
The original geometry has \num{37244} triangles, the stage-1 mesh has \num{45755} order~2 triangles (with 6 nodes per element),
and the stage-2 mesh has \num{102524} order-2 triangles; the triangles of the stage-2 quadrature
discretization each have \num{112} nodes. The Helmholtz parameter was set to $k=20$.

We use a Brakhage-Werner representation~\cite{brakhage_uber_1965} to solve for
boundary values obtained from a point potential emanating from a number of
sources in the `tail' of the geometry. Using $L^2$-weighted degrees of
freedom~\cite{bremer_nystrom_2011}, GMRES~\cite{saad_gmres_1986} attained
a decrease in the residual norm by a factor of $10^{-5}$ in 79 iterations. The
calculation took around two days on a dual-socket Intel Xeon~E5-2650~v4 machine.
The Helmholtz translation operators used in the algorithm were those from
FMMLIB3D~\cite{gimubtas_fmmlib,gimbutas_fast_2009}.

We verify that the potential obtained from the boundary value solve matches the
point potential through point evaluations in the volume, obtaining roughly two
digits of accuracy. Details of the relative error in the potential evaluation in
the volume can be found in Figure~\ref{fig:betterplane}.

% }}}
% -----------------------------------------------------------------------------
\subsection{Cost and Scaling}%
\label{sec:exp-cost}
% -----------------------------------------------------------------------------
% {{{
\begin{figure}[t]%
  \centering
  %\captionsetup{width=0.48\textwidth}

  \input{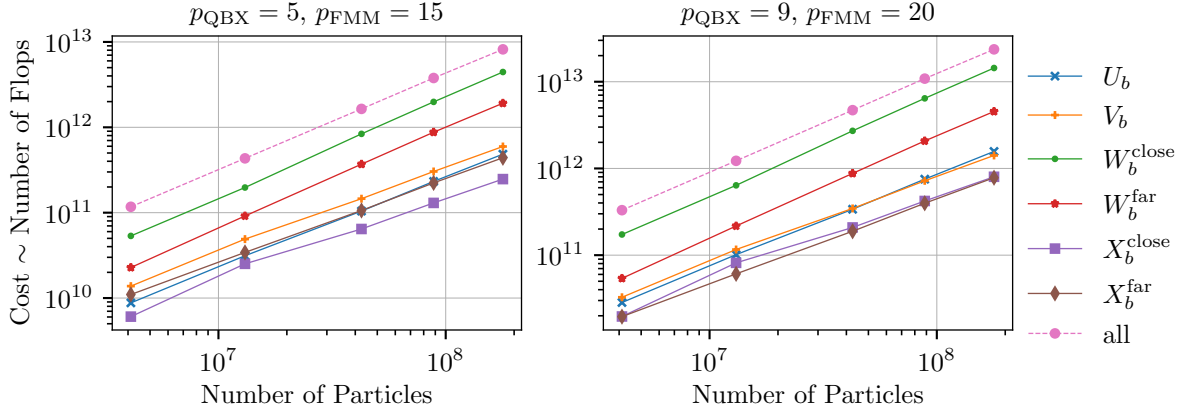}
  \caption{Modeled operation counts for the \algbrand~FMM for evaluating the single-layer
    potential on a sequence of `urchin' geometries of increasing particle
    count. The operations are counted according to the model presented in
    Table~\ref{tab:ilist-operation-counts}. Here, $\nmax = \nmaxgigaqbx$ and $t_f =
    0.9$. The scaling test used the `urchin' geometries
    $\gamma_2,\gamma_4,\dots,\gamma_{10}$.}%
  \label{fig:complexity-gigaqbx}
\end{figure}

\begin{table}%
  \begin{minipage}[t]{0.3\textwidth}%
    % FIXME top-align doesn't totally work
    \centering%
    \caption{Cost per interaction list entry modeled in
      Figure~\ref{fig:complexity-gigaqbx}, i.e.\ for a single (source box, target box) interaction
      list pair.  $\pfmm$ = FMM order and $\pqbx$ = QBX order. $n_s$ = number of sources
      in the source box and $n_t$ = number of QBX centers in the target box.}%
    \label{tab:ilist-operation-counts}
    \begin{tabular}{lc}
      \toprule
      List & Cost \\
      \midrule
      $\ilist{1}{b}$ & $\pqbx^2 n_s n_t$ \\
      $\ilist{2}{b}$ & $\pfmm^3$ \\
      $\ilist{3close}{b}$ & $\pqbx^2 n_s n_t$ \\
      $\ilist{3far}{b}$ & $\pfmm^3 n_t$ \\
      $\ilist{4close}{b}$ & $\pqbx^2 n_s n_t$ \\
      $\ilist{4far}{b}$ & $\pfmm^2 n_s$ \\
      \bottomrule
    \end{tabular}
  \end{minipage}%
  \hfill%
  \begin{minipage}[t]{0.65\textwidth}%
  {%
    \sisetup{%
      table-format = 1.2e-1,
      table-number-alignment = center,
      table-sign-exponent = true,
      scientific-notation = true,
      round-mode = places,
      round-precision = 2,
      detect-weight = true,
      mode = text,
    }%
      \centering%
      \def\pfmmshort{p}%
      \def\pqbxshort{q}%
      \caption{%
        A comparison of (modeled, cf.~Table~\ref{tab:ilist-operation-counts})
        cost for the GIGAQBX FMM between the use of the conventional
        $\ell^\infty$ box extent norms (analogous to~\cite{gigaqbx2d}) and the
        $\ell^2$ box extent norms introduced in this article.
        Entries in the table show modeled floating point operations
        in the sense of Section~\ref{sec:complexity}.
        The experiment used the `urchin' geometry $\gamma_8$ discretized with
        48,500~stage-1 elements and 277,712~stage-2 elements, where each of
        the latter had 295 nodes.
        The evaluated columns use $\pfmm=15$ and $\pqbx=5$, corresponding
        to around five digits of accuracy following
        Table~\ref{tab:urchin-accuracy}.  For brevity, we let $p=\pfmm$ and $q=\pqbx$.
        Note that the rows shown do \emph{not} add up to the shown total. The
        latter includes minor contributions  to the overall cost (such as the
        upward and downward passes) that we have omitted.
      }%
      \label{tab:l2-comparison}
      \begin{tabular}{llSlS}
        \toprule
        {} & {$\ell^\infty$ (sym.)} & {$\ell^\infty$} & {$\ell^2$ (sym.)} & {$\ell^2$}\\
        \midrule
        $U_b$ & $2.77\times 10^{10} {\pqbxshort}^{2}$ & 691861569975 & $9.29\times 10^{9} {\pqbxshort}^{2}$ & 232238607650\\
        $V_b$ & $6.14\times 10^{7} {\pfmmshort}^{3}$ & 207237676500.0 & $9.00\times 10^{7} {\pfmmshort}^{3}$ & 303722581500.0\\
        $W_b^\mathrm{close}$ & $8.88\times 10^{10} {\pqbxshort}^{2}$ & 2218766863925 & $7.94\times 10^{10} {\pqbxshort}^{2}$ & 1985187158425\\
        $W_b^\mathrm{far}$ & $4.56\times 10^{8} {\pfmmshort}^{3}$ & 1538782623000.0 & $2.59\times 10^{8} {\pfmmshort}^{3}$ & 875782648125.0\\
        $X_b^\mathrm{close}$ & $2.62\times 10^{9} {\pqbxshort}^{2}$ & 65459800350 & $5.20\times 10^{9} {\pqbxshort}^{2}$ & 130105170000\\
        $X_b^\mathrm{far}$ & $9.51\times 10^{8} {\pfmmshort}^{2}$ & 213896794050 & $9.92\times 10^{8} {\pfmmshort}^{2}$ & 223278102225\\
        \midrule
        \textbf{Total} &  & 4965633909425.0 &  & 3780554811300.0\\
        \bottomrule
      \end{tabular}
  }
  \end{minipage}
\end{table}

It remains to examine both the computational cost and the scaling thereof that
the algorithm achieves on geometries of varying size. Rather than relying on
wall time (which is sensitive to machine details as well as varying levels of optimization
and code quality), we present an abstract operation count intended to
asymptotically match the number of floating point operations, similarly to the
approach of Section~\ref{sec:translation-complexity}. We account for each entry
in the interaction lists of Section~\ref{sec:interaction-lists} with the counts
shown in Table~\ref{tab:ilist-operation-counts}. We use the `urchin' test
geometries $\gamma_2,\gamma_4,\dots,\gamma_{10}$ for this computational
experiment.

We show cost data for two pairs of $(\pqbx,\pfmm)$, corresponding to different
accuracies. The first, $(\pqbx,\pfmm)=(5,15)$, corresponds to roughly five digits
of accuracy following Table~\ref{tab:urchin-accuracy}, whereas the second,
$(\pqbx,\pfmm)=(9,20)$, corresponds to around seven digits of accuracy.

As in point FMMs, the main tuning parameter that may be used to balance
various cost contributions and minimize computational cost is $\nmax$, the
maximal number of particles per box. We chose $\nmax$ to minimize the modeled
computational cost, obtaining a value of $\nmax=512$. The TCF $t_f$ mainly
trades off cost and accuracy, we choose $t_f=0.9$.
We show graphs of computational cost across geometry sizes in
Figure~\ref{fig:complexity-gigaqbx}.

Unlike in two dimensions, we observe that $\ilist{3close}{}$ and, to a lesser
extent, $\ilist{3far}{}$ dominate the run time of the algorithm. This is not
entirely unexpected, as the size of the TCR, naturally larger by particle
count in three dimensions, makes its influence felt. A further factor in the
large contribution of $\ilist{3far}{}$ is the high cost of translations even
when the target (QBX) expansion is of comparatively low order,
cf.\ Section~\ref{sec:translation-complexity}.

In accordance with the results of Section~\ref{sec:complexity}, the experiments
support the conclusion that the algorithm exhibits linear scaling in the number
of source and target particles, with one decade of geometry growth (indicated by
the vertical grid lines) leading to one decade of cost growth (indicated by
horizontal grid lines).

\subsection{Cost Implications of the \texorpdfstring{$\ell^2$}{l2}-Based Target Confinement Region}

Next, we seek to understand the impact of the change in the shape of the TCR,
which was box-shaped and defined by the $\ell^\infty$-norm in the earlier
version of our algorithm~\cite{gigaqbx2d}, but which now is spherical and
measured by an $\ell^2$-norm to better match the actual region of convergence of
the obtained local expansions. Table~\ref{tab:l2-comparison} summarizes the
results of an experiment determining the comparative cost of both approaches.
Both versions of the algorithm were balanced individually before conducting the
experiments, in both cases $\nmax=512$ turned out to be near-optimal.  First, we
observe that the algorithmic change has led to a reduction of (modeled)
computational cost by around 25 per cent. We note a marked increase in the cost
contribution of the $\ilist{2}{}$ list, as well as marked decreases in the cost of
the $\ilist{1}{}$ and $\ilist{3far}{}$ lists, all of which are indicative of the
higher efficiency of the method with the $\ell^2$ TCR\@.

% }}}

% #############################################################################
\section{Conclusion}%
\label{sec:conclusion}
% -----------------------------------------------------------------------------
% {{{

This paper introduces a fast algorithm for the accurate evaluation of layer
potentials in three dimensions using Quadrature by Expansion (QBX).

Our work builds on and extends the \algbrand~algorithm in two
dimensions~\cite{gigaqbx2d}. Many features of the algorithm carry over broadly
unchanged from the two dimensional setting. However, some parts have required
careful redevelopment. A practical QBX implementation must provide a mechanism
to control for truncation error, quadrature error, and error introduced by FMM
acceleration. To address these challenges in three dimensions, our work combines
new error estimates for FMM translations in three dimensions, a new local
refinement criterion for truncation error control based on scaled-curvature, and a novel adaptive refinement scheme for achieving
source quadrature resolution. In a series of numerical experiments, we
demonstrate that this combination can achieve high accuracy for layer potential
evaluation on complicated geometries. In particular, we show how the FMM
acceleration recovers similar levels of acceleration error as the point FMM\@.
The numerical evidence
for the usefulness of our error control strategies is robust. A rigorous
mathematical treatment of these error control strategies appears eminently
feasible in some cases, such as for translation operators. We leave this as a
subject of future investigation.

Additionally, we describe a benign set of sufficient conditions on the geometry
under which the running time of the algorithm has linear complexity. By counting
modeled flops on large scale geometries, the scaling of the algorithm is shown
to be linear in practice. The most expensive part of the algorithm
is the QBX near field evaluations. Fortunately, strategies such as changing the
shape of the target confinement region are available to reduce this
expense. Further optimizations to the scheme are the subject of ongoing
investigation.

% }}}

% #############################################################################
\section*{Acknowledgments}
% #############################################################################
The authors' research was supported by the National Science Foundation under
award numbers DMS-1418961 and DMS-1654756.  Any   opinions,   findings,   and
conclusions, or   recommendations expressed in this article are those of the
authors and do not necessarily reflect the views of the National Science
Foundation; NSF has not approved or endorsed its content.
Part of the work was performed while the authors were participating in
the HKUST-ICERM workshop `Integral Equation Methods, Fast
Algorithms and Their Applications to Fluid Dynamics and Materials
Science' held in 2017. The authors would like to thank Hao Gao for
identifying a software issue in a prior version of this work.

\appendix
% #############################################################################
\section{Area Queries}%
\label{sec:area-query-alg}
% -----------------------------------------------------------------------------
% {{{

\emph{Area queries} were introduced in~\cite{rachh:2017:qbx-fmm} in two dimensions.
We describe their (largely straightforward) three-dimensional generalization in this section. They form the core
mechanism on which the many of the geometric operations in this article are
performed. Given a center $c$ and a radius $r$, the area query
computes the set of leaf boxes that intersect $\closedbox(r, c)$. It is assumed
that $c$ falls inside the computational domain and that $r$ is at most the
radius of the tree.

Since the area query as a primitive is used to retrieve sets of sources and
targets by way of their containing boxes, it may appear flawed that the area
query only considers leaf boxes, when the fast algorithm of
Section~\ref{sec:algorithm} permits targets (specifically, those with extent,
i.e.\ QBX expansion balls) to occur in non-leaf boxes. This is not an issue
since all application scenarios of the area query
(Sections~\ref{sec:truncation-error},~\ref{sec:accurate-quadrature},
and~\ref{sec:target-assoc}) use it to find point-shape objects (i.e.\
objects without extent) which are necessarily found in leaf boxes of the tree.

An area query proceeds by descending the tree towards the query center $c$ until
the descent has reached a box whose size is commensurate with the size of the
query box $\closedbox(r, c)$. This box is referred to as the \emph{guiding
  box}. Specifically, the guiding box is the smallest box whose
$1$-near~neighborhood contains $\closedbox(r, c)$. Once this box has been found,
only the leaf descendants of the $1$-near~neighborhood of $b$ need to be checked
for intersection with the query box. The full procedure to carry out an area
query is given in Algorithm~\ref{alg:area-query}.

Recall that a \emph{colleague} of a box (what we also refer to as a
$1$-colleague) is a box on the same level as $b$ that is adjacent to $b$. In a
non-adaptive tree, the $1$-near~neighborhood is specified by a box $b$ and its
set of colleagues.  In an adaptive tree, this is no longer the case as some
colleagues may be missing. The notion of a \emph{peer box} is a generalization
of a colleague which allows for larger boxes to stand in as colleagues if
necessary. This makes it useful for the area query.

\begin{definition}[Peer box~\cite{rachh:2017:qbx-fmm}]
  Let $b$ be a box in an octree. A box $c$ is a \emph{peer box} of $b$ if
  \begin{inparaenum}[(a)]
    \item $c$ is $b$ or adjacent to $b$,
    \item the size of $c$ is at least the size of $b$, and
    \item no child of $c$ satisfies the previous criteria.
  \end{inparaenum}
\end{definition}

\begin{algbreakable}{Area Query}%
  \label{alg:area-query}
  \begin{algorithmic}
    \REQUIRE{A center $c$ and a query radius $r$.}
    \ENSURE{Computes the set of leaf boxes which intersect $\closedbox(r,c)$.}

    \algstage{Find the guiding box $b$}
    {%
      \STATE{$b$ $\gets$ the root box.}
      \LOOP{}
        \IF{$|b| < r \leq 2|b|$ \OR $b$ has no child containing $c$}
        \STATE{\textbf{break}}
        \ENDIF{}
        \STATE{$b$ $\gets$ the child of $b$ containing $c$.}
      \ENDLOOP{}
    }

    \algstage{Check leaf descendants of $b$'s peers}
    {%
      \FORALL{peers $p$ of $b$}
        \FORALL{leaf descendants $l$ of $p$}
          \IF{$\closedbox(r, c) \cap l$ is nonempty}
            \STATE{\textbf{Add} $l$ to the output set.}
          \ENDIF{}
        \ENDFOR{}
      \ENDFOR{}
    }
  \end{algorithmic}
\end{algbreakable}
% }}}

% #############################################################################
\section{Numerical Experiments in Support of FMM Translation Error Estimates}%
\label{sec:fmm-translation-experiments}
% -----------------------------------------------------------------------------

\begin{table}
  \centering
  \begin{tabular}{cl}
    \toprule
    Parameter & Choice \\
    \midrule
    $R$ & $\{ 0.1, 1, 10 \}$ \\
    $\rho$ & $\{0.1, 1, 10\}$ \\
    $r$ & $\{0.25\rho, 0.5\rho, 0.75\rho\}$ \\
    $q$ & $\{3, 5, 10, 15, 20\}$ \\
    $p$ & $\{3, 5, 10, 15, 20\}$ \\
    \bottomrule
  \end{tabular}
  \caption{Summary of parameters chosen for the numerical experiments used to
    obtain the results in
    Sections~\ref{sec:mpole-accuracy},~\ref{sec:local-accuracy},~and~\ref{sec:m2l-accuracy}.}%
  \label{tab:numerical-experiment-params}
\end{table}

The code used for the numerical experiments performed to obtain the results of
Section~\ref{sec:error-estimates} is available
at~\cite{translation_accuracy_repo_2018}. In this appendix, we describe the
procedure the code uses.

% -----------------------------------------------------------------------------
\subsection{Multipole and Multipole-to-Local Accuracy}
% -----------------------------------------------------------------------------

% m2qbxl

\def\nspherepts{42}
\def\nballpts{57}
\def\nmtoptests{675}
\def\nmtopsamples{42^2 \cdot 57}
\def\nmtoltests{675}
\def\nmtolsamples{42^2 \cdot 57}
\def\nltoptests{225}
\def\nltopsamples{42 \cdot 57}

We use the notation of
Section~\ref{sec:mpole-accuracy}. Hypothesis~\ref{hyp:m2p} pertains to the
accuracy of approximating a local expansion using an intermediate multipole
expansion.

As a numerical experiment, we test the truth of this hypothesis at selected
values of the parameters $(R, r, \rho, p, q)$. For a given value of these
parameters, estimate of the value $E_M(q)$~\eqref{eqn:m2p-error} is produced by sampling this value at
$\nmtopsamples$ tuples of the form $(s, c, c', t)$. The values of the parameters
$(R, r, \rho, p, q)$ are chosen according to
Table~\ref{tab:numerical-experiment-params}. In total, $\nmtoptests$ parameter
tuples are tested.

The details of the sampling procedure are as follows. A multipole expansion
center is placed at~$c = (0, 0, R + \rho)$.  First, $\nspherepts$~sources~$s$
are selected from the sphere of radius~$r$ centered at~$c$. Second,
$\nballpts$~centers~$c'$ are selected in the ball of radius $R$ centered at the
origin.  Third, for each center~$c'$, $\nspherepts$~targets~$t$ are selected
from the sphere of radius $R - |\pt{c'}|$ centered at $c'$.

The points selected on the sphere are selected to be approximately
equispaced~\cite{deserno:sphere} and included the poles of the sphere. The
points in the ball are selected from concentric spheres inside the ball.

For each sampled value $(s, c, c', t)$, the multipole expansion due to~$\pt{s}$
is formed at~$\pt{c}$ and translated to a local expansion centered at~$\pt{c'}$.
The quantity $E_M(q)$ is evaluated at the target~$\pt{t}$. We save the largest
observed value of $E_M(q)$ and take this as an upper bound on that quantity for
the given set of parameters $(R, r, \rho, p, q)$.

Hypothesis~\ref{hyp:m2l} in Section~\ref{sec:m2l-accuracy} pertains to the
accuracy of a local expansion approximated by an intermediate multipole and
local expansion. This hypothesis is the same geometric setting as
Hypothesis~\ref{hyp:m2p}. We follow an identical sampling procedure to obtain a
numerical estimate of the quantity $E_{\mathit{M2L}}(q)$~\eqref{eqn:m2l-error}.

% -----------------------------------------------------------------------------
\subsection{Local Accuracy}
% -----------------------------------------------------------------------------

Section~\ref{sec:local-accuracy} uses a different geometrical scenario from the
multipole case. Hypothesis~\ref{hyp:l2p} in this section pertains to the
accuracy of approximating a local expansion using an intermediate local
expansion.

This hypothesis is tested at selected values of the parameters $(r, \rho, p,
q)$. For a given value of these parameters, we obtain an estimate of the
quantity $E_L(q)$~\eqref{eqn:l2p-error} by sampling a set of $\nltopsamples$ tuples of the form
$(\pt{s}, \pt{c}, \pt{t})$ using a sampling procedure. The parameters are taken
according to Table~\ref{tab:numerical-experiment-params}, with a total of
$\nltoptests$ parameter tuples tested.

The details of the sampling procedure are as follows. First, the source
$\pt{s}$ is placed at $(0, 0, \rho)$. Second, $\nballpts$~centers~$\pt{c}$ are
chosen from inside the ball of radius $r$ centered at the origin. Third, each
center~$\pt{c}$, $\nspherepts$~targets~$\pt{t}$ is chosen from the sphere of
radius $r - |\pt{c}|$ centered at $\pt{c}$.

For each value of the tuple $(\pt{s}, \pt{c}, \pt{t})$, a $p$-th order local
expansion due to~$\pt{s}$ is formed at the origin. The local expansion is
subsequently translated to a $q$-th order local expansion at~$\pt{c}$. The
quantity~$E_L(q)$ is evaluated at $\pt{t}$. The largest observed value
of~$E_L(q)$ is taken as an estimate of the upper bound on the quantity.

% #############################################################################
\section{Detailed Complexity Analysis}%
\label{sec:detailed-complexity-analysis}
% -----------------------------------------------------------------------------

This section provides the details of the complexity analysis from
Section~\ref{sec:complexity}, under the assumptions highlighted in
Section~\ref{sec:complexity-assumptions}. In
Section~\ref{sec:translation-complexity} and~\ref{sec:particle-distribution} we
review the complexity of translations and the effect of the particle
distribution. Section~\ref{sec:interaction-lists} provides the details
supporting the analysis in Table~\ref{tab:complexity-analysis}.

We use the parameters from Table~\ref{tab:complexity-params} throughout this
section. In addition to the parameters from this table, we make use of the
following notation. First, we let $S$ and $C$ denote, respectively the sets of
sources and QBX centers in the tree. Second, given a particle $p$ in the tree,
we let $b_p$ denote the box owning $p$.

% -----------------------------------------------------------------------------
\subsection{Complexity of Translation Operators}%
\label{sec:translation-complexity}
% -----------------------------------------------------------------------------
% {{{

A $p$-th order multipole/local expansion requires $(p+1)^2$ expansion
coefficients. The $(p+1)^2$ corresponding basis functions for the coefficients
may be evaluated in $O(p^2)$ time using well-known
recurrences~\cite[Ch.\ 14]{nist:dlmf}. As a result, we model the cost of forming
or evaluating a $p$-th order multipole/local expansion in spherical harmonics as
$p^2$ operations, which is correct to leading order.

The cost of translations of spherical harmonic expansions (multipole/local $\to$
local) may be modeled as follows. With a simple extension to the commonly used
`point and shoot' translation
scheme~\cite{gumerov:3d-translation-ops-comparison}, a $p$-th order expansion
can be translated into a $q$-th order expansion in the following three steps:
\begin{enumerate}
  \item Rotate the coordinate system so that the translation direction is along
    the $z$-axis at a cost of $O(p^3)$ operations.
  \item Translate the expansion at a cost of $O(p q^2)$ operations.
  \item Rotate the coordinate system back into the original at a cost of
    $O(q^3)$ operations.
\end{enumerate}
It follows that the complexity of a translation operator is $O(\max(p,q)^3)$.
If we assume that $\pqbx \leq \pfmm$, the cost of translating a $\pfmm$-th order
expansion into a $\pqbx$-th order expansion can be modeled as $\EtoEcost$
operations.

% }}}

% -----------------------------------------------------------------------------
\subsection{Effect of Particle Distribution}%
\label{sec:particle-distribution}
% -----------------------------------------------------------------------------
% {{{

The running time of the \algbrand~FMM cannot be entirely independent of the
particle distribution. Unlike the point FMM, the algorithm may place more than
$\nmax$ particles in a box. This occurs due to clustering of QBX centers in
boxes because of the target confinement rule. This phenomenon cannot be
disregarded as it occurs in practice even with smooth geometries.

To handle this in the complexity analysis, we find it useful to distinguish
between QBX centers that have `settled' in a leaf box and those that are
`suspended', i.e.~that cannot be placed in a lower box due to target confinement
restrictions. By definition, only the latter kind of centers can cluster in the
tree beyond $\nmax$ particles per box.

\begin{definition}[Leaf-settled / suspended centers]
  Suppose the tree has been constructed following Algorithm~\ref{alg:gigaqbx}.
  Call a center $c$ owned by a box $b$ a \emph{suspended center} if $c$ cannot be
  placed in any hypothetical child box of $b$ due to target confinement
  restrictions. A center that is not suspended in any box is called
  \emph{leaf-settled}.
\end{definition}

To bound the number of algorithmic operations involving suspended QBX centers,
we introduce a parameter into the complexity analysis that corresponds to the
average size of a `neighborhood' of a suspended QBX center---in other words,
the average number of sources with which a QBX center must interact
directly. While this quantity may at first seem to be tied to the tree structure
imposed on the geometry, it is possible to bound this quantity independently of
the tree. A tree-independent bound on this quantity follows from the following
lemma.

\begin{lemma}[Size of a suspended QBX expansion ball relative to box neighborhood]%
  Let $c$ be a suspended QBX center of radius $r_c$ owned by the box $b_c$.
  Then the closed cube $\closedbox(4 \sqrt{3} r_c / t_f, c)$ is,
  geometrically, a superset of the $2$-near~neighborhood of $b_c$.
\end{lemma}

\begin{proof}
  For any box $b$, the Euclidean distance from a point $x$ inside $b$ to a point
  on the boundary of $\tcr(b)$ is at least $\sqrt{3} |b| t_f$. This distance
  is minimized when $x$ is a box corner.

  Since $c$ is suspended, $c$ cannot fit in any hypothetical child box of $b_c$,
  which has radius $|b_c| / 2$.  It follows from the previous observation that
  \[
    r_c > \sqrt{3} |b_c| t_f / 2.
  \]
  Regardless of where $c$ is located in $b_c$, $\closedbox(c, 6|b_c|)$ is a
  superset of the $2$-near~neighborhood of $b$. The claim follows by observing
  $4 \sqrt{3} r_c / t_f > 6 |b_c|$.
\end{proof}

The quantity $M_C$ is then defined as follows:
\[
  M_C \coloneqq \frac{1}{N_C} \sum_{c \in C} \left| S \cap
  \closedbox(4 \sqrt{3} r_c / t_f, c) \right|
\]
where $r_c$ is the radius of the center $c$. The following proposition is an
immediate consequence of the definition of $M_C$ and the previous lemma.

\begin{proposition}[Bound on $2$-near-neighborhood interactions for QBX expansions]%
  \label{prop:ncmc}%
  The number of source-center pairs $(s, c)$, such that $c$ is a suspended QBX center and
  $s$ is a source particle in the $2$-neighborhood of the box owning $c$, is at
  most $N_C M_C$.
\end{proposition}
% }}}

% -----------------------------------------------------------------------------
\subsection{Complexity of Algorithmic Stages Associated with Interaction Lists}%
\label{sec:interaction-lists}
% -----------------------------------------------------------------------------
% {{{

\begin{proposition}[Number of larger leaf boxes in the $1$-neighborhood of a box]%
  \label{prop:num-adjacent-big-leaves}%
  Let $b$ be a box. There are at most $27$ leaf boxes at least at large as $b$
  intersecting the $1$-near~neighborhood of $b$.
\end{proposition}

\begin{proof}
  Let $l$ be such a leaf box. If $l \neq b$, choose a box $c_l$ which is a
  colleague of $b$ that is geometrically contained inside $l$. The mapping $l
  \mapsto c_l$ is injective, and $b$ has at most $3^3-1=26$ colleagues.
\end{proof}

% list 1
\begin{lemma}[List 1 complexity]%
  \label{lem:list1-complexity}%
  The amount of work done in Stage 3 (direct evaluation of the potential
  from adjacent source boxes) is at most
  \[
    (27 (N_C + N_S) \nmax  + N_C M_C) \PtoPcost.
  \]
\end{lemma}

\begin{proof}
  Define the set $U$ as $U = \{ (s, c) \mid b_s \in U_{b_c} \}$. Each
  source-center interaction costs $\PtoPcost$ operations. The number of Stage 3
  interactions is $|U|$, so the cost of Stage 3 is at most $\PtoPcost |U|$. $U$
  may be written as a disjoint union $U = U_\mathrm{big} \cup U_\mathrm{small}$,
  where $U_\mathrm{big}$ contains all pairs $(s, c)$ such that $|b_s| \geq
  |b_c|$.

  For any center $c$, Proposition~\ref{prop:num-adjacent-big-leaves} implies
  that there may be at most $27 \nmax$ sources in leaf boxes at least
  as large as $b_c$ contributing to the potential via $U_{b_c}$. Thus
  $|U_\mathrm{big}| \leq 27 N_C \nmax$.

  In a similar way, one can show that there are at most $27 N_S \nmax$ pairs
  $(s, c) \in U_\mathrm{small}$ such that $c$ is a leaf-settled center. By
  Proposition~\ref{prop:ncmc}, there are at most $N_C M_C$ pairs $(s, c) \in
  U_\mathrm{small}$ such that $c$ is a suspended center. It follows that
  $|U_\mathrm{small}| \leq N_C M_C + 27 N_S \nmax$.
\end{proof}

% list 2
\begin{lemma}[List 2 complexity]%
  \label{lem:list2-complexity}%
  The amount of work done in Stage 4 (translation of multipole to local
  expansions) is at most $875 N_B \EtoEcost$.
\end{lemma}

\begin{proof}
  There are at most $N_B$ boxes. The size of List 2 for a box $b$ is at most
  $10^3 - 5^3 = 875$, since there are at most $10^3$ descendants of
  $2$-colleagues of the parent of $b$, of which $5^3$ are $2$-colleagues of $b$
  itself so they cannot be in List 2 of $b$. Each multipole-to-local translation
  costs $\EtoEcost$ operations.
\end{proof}

% list 3
\begin{lemma}[List 3 complexity]%
  \label{lem:list3-complexity}%
  The amount of work done in Stage 5 (evaluation of List 3 close and far) is at
  most \[N_C M_C \PtoPcost + 124 L N_S \nmax \PtoPcost.\]
\end{lemma}

\begin{proof}
  We make the simplifying assumption that \emph{all} Stage 5 interactions are
  mediated by List 3 close. This assumption \emph{will not lead to an
    undercount} of the cost of the Stage 5 interactions, if the optimization in
  Remark~\ref{rem:list3far-to-list3close} has been applied. Recall that a List 3
  far interaction is a multipole-to-target interaction, and a List 3 close
  interaction is a source-to-target interaction.  The only way the above
  assumption could undercount the cost of Stage 5 is if a List 3 far interaction
  were more expensive than the equivalent interactions that occur with the leaf
  descendants using List 3 close.  But if a List 3 far interaction is more
  expensive than List 3 close interaction, the interaction of the former type
  may be converted to the latter type with no loss in accuracy.

  Under this assumption, Proposition~\ref{prop:ncmc} implies that the cost of
  all Stage 5 interactions aimed at suspended centers is at most $N_C M_C \PtoPcost$.

  Now, consider the Stage 5 interactions aimed at leaf-settled centers. Let $s$
  be a source owned by a box $b_s$. If $s$ interacts via List 3 close with a
  leaf-settled center $c$, then $c$ must be owned by a box that is a
  $2$-colleague of either $b_s$ or an ancestor of $b_s$. A box has at most $5^3
  - 1 = 124$ boxes that are $2$-colleagues. Since each source-center interaction
  costs $\PtoPcost$ and there are at most $\nmax$ leaf-settled centers per box,
  the cost of all interactions aimed at leaf-settled centers is at most $124 L
  N_S \nmax \PtoPcost$.
\end{proof}

\begin{remark}[Effect of level-restriction on List 3 complexity]%
  \label{rem:level-restriction}%
  The factor of $L$ in Lemma~\ref{lem:list3-complexity} suggests that Stage 5
  has a worst-case superlinear scaling. A number of modifications to the
  algorithm are available that can provably remove the asymptotic factor of
  $L$. For instance, a cost estimate for Stage 5 that is independent of $L$ may
  be derived assuming that the tree is level-restricted, meaning that adjacent
  leaves differ by at most one level, and that $t_f < \sqrt{3} - 1 \approx
  0.73$.

  For $t_f < \sqrt{3} - 1$, $\tcr(b)$ is contained strictly inside the
  $1$-neighborhood of $b$. If the tree is level-restricted, this implies that if
  $b$ is a leaf box, any box in $\ilist{3close}{b} \cup \ilist{3far}{b}$ cannot
  be more than a constant factor smaller than $b$. This further implies that for
  a leaf box $b$, the quantity $|\ilist{3close}{b} \cup \ilist{3far}{b}|$ is at
  most a constant that depends on $t_f$ and on the dimension. This can be used
  to construct a cost estimate independent of $L$. We leave the details of the
  derivation to the reader.

  Using a level-restricted tree does not impact the asymptotic scaling of any
  other stage of the algorithm. Any octree tree may be converted to be
  level-restricted by repeatedly subdividing the larger of the leaf boxes that
  violate the level-restriction criterion. The level-restricted tree that results
  has a constant factor as many
  boxes. See~\cite{moore:1995:cost-of-balancing-generalized-quadtrees} for
  details.
\end{remark}

% list 4
\begin{lemma}[List 4 complexity]%
  \label{lem:list4-complexity}%
  The cost of all Stage 6 interactions (evaluation of the potential due to List
  4 close and far) is at most
  \[
    375 N_B \nmax \PtoEcost + 250 N_C \nmax \PtoPcost.
  \]
\end{lemma}

\begin{proof}
  First, we show $|X_b| \leq 125$.  Every box in $X_b$ is a leaf that is
  either a $2$-colleague of $b$ not adjacent to $b$, or adjacent to the parent
  of $b$ and at least as large as the parent of $b$. There are at most $5^3 -
  3^3 = 98$ boxes that fall into the first category and, by
  Proposition~\ref{prop:num-adjacent-big-leaves}, at most $27$ boxes that fall
  into the second category.

  Next, we show that $\ilist{4close}{b} \subseteq \ilist{4}{b} \cup
  \ilist{4}{\parent(b)}$, which implies that $|\ilist{4close}{b}| \leq
  250$. Recall that $\ilist{4close}{b}$ must be a subset of the List 4's of the
  ancestors of $b$. If $b'$ is an ancestor of $b$, then $b$ must be separated by
  an $\ell^\infty$ distance of $2^{k+1}|b|$ from any box in $\ilist{4}{b'}$. In
  particular, consider a box $e \in \ilist{4}{g}$, where $g$ is $k \geq 2$
  levels above $b$. Then $e$ will be separated by an $\ell^\infty$ distance of
  at least $8|b|$ from $b$. It follows that the $\ell^\infty$ distance from the
  center of $b$ to the boundary of $e$ is at least $9|b|$, which is at least
  $3(1 + t_f)|b|$ (since $t_f < 2$ by assumption). Thus $\tcr(b) \adequatesep e$
  from the definition of `$\adequatesep$'. It follows that $\ilist{4close}{b}$
  is disjoint from the List 4 of a grandparent of $b$ or above.

  Finally, $|\ilist{4far}{b}| \leq 375$ follows since, by definition,
  $\ilist{4far}{b} \subseteq X_b \cup \ilist{4close}{\parent(b)}$.

  The cost estimate follows since, for List 4 close, each center will interact
  directly with at most $250 \nmax$ source particles, with each interaction
  costing $\PtoPcost$. For List 4 far, each box will interact with at most $375
  \nmax$ source particles, at a cost of $\PtoEcost$ per interaction.
\end{proof}

% }}}

% #############################################################################
\printbibliography{}
% -----------------------------------------------------------------------------

\end{document}

% vim: foldmethod=marker:textwidth=80